\documentclass{article}

\usepackage[applemac]{inputenc}
\usepackage[english]{babel}
\usepackage{amssymb}
\usepackage{amsmath}
\usepackage{latexsym}
\usepackage{marvosym}
\usepackage{amsfonts}
\usepackage{graphicx}
\usepackage{color}
\usepackage{fullpage}
\usepackage{fourier-orns}
\usepackage{eurosym}
\usepackage{enumerate}
\usepackage{fancyhdr}
\usepackage{bbold}
\usepackage{tikz}
\usepackage[affil-it]{authblk}

\usepackage{enumitem}

\usepackage{floatrow}
\usepackage{graphicx}
\usepackage[rightcaption]{sidecap}
\usepackage[pdftex=true,hyperindex=true,colorlinks=true,linkcolor=blue,bookmarks=true]{hyperref}
\hypersetup{
backref=true,
pagebackref=true,
hyperindex=true,
colorlinks=true,
breaklinks=true, 
urlcolor= blue,
linkcolor= blue,
citecolor=blue,
bookmarks=true,
bookmarksopen=true,
}

\newtheorem{definition}{Definition}[section]
\newtheorem{prop}[definition]{Proposition}
\newtheorem{cor}[definition]{Corollary}

\newtheorem{hyp}[definition]{Assumption}
\newtheorem{hyps}[definition]{Assumptions}
\newtheorem{rem}[definition]{Remark}
\newtheorem{theo}[definition]{Theorem}

\newcommand{\ud}{\mathrm{d}}

\newcommand{\R}{\mathbb{R}}
\newcommand{\prob}{\mathbb{P}}

\DeclareMathOperator*{\argmax}{arg\,max}

\newcommand{\hiddensubsubsection}[1]{
    \stepcounter{subsubsection}
    \subsubsection*{B.\arabic{subsection}.\arabic{subsubsection}\hspace{1em}{#1}}
}

\title{Optimal choice among a class of nonparametric estimators of the jump rate for piecewise-deterministic Markov processes}

\date{}
\author{Romain Aza\"{\i}s%
	\thanks{Electronic address: \texttt{romain.azais@inria.fr}; Corresponding author} }

\author{Aur\'elie Muller-Gueudin%
	\thanks{Electronic address: \texttt{aurelie.gueudin@univ-lorraine.fr}}}

\affil{\small{Inria Nancy-Grand Est, Team BIGS, Institut \'Elie Cartan de Lorraine, Vandoeuvre-l\`es-Nancy, France}}

\begin{document}

\maketitle

\vspace{-1cm}

\abstract{A piecewise-deterministic Markov process is a stochastic process whose behavior is governed by an ordinary differential equation punctuated by random jumps occurring at random times.
We focus on the nonparametric estimation problem of the jump rate for such a stochastic model observed within a long time interval under an ergodicity condition.
We introduce an uncountable class (indexed by the deterministic flow) of recursive kernel estimates of the jump rate and we establish their strong pointwise consistency as well as their asymptotic normality. We propose to choose among this class the estimator with the minimal variance, which is unfortunately unknown and thus remains to be estimated. We also discuss the choice of the bandwidth parameters by cross-validation methods.}

\smallskip


\smallskip


\setcounter{tocdepth}{2}
\tableofcontents

\section{Introduction}

Piecewise-deterministic Markov processes (PDMP's in abbreviated form) have been introduced in the literature by Davis in \cite{Dav} as a general class of continuous-time non-diffusion stochastic models, suitable for modeling deterministic phenomena in which the randomness appears as point events. The motion of a PDMP may be defined from three local characteristics: the flow $\Phi(x,t)$, the jump rate $\lambda(x)$ and the transition measure $\mathcal{Q}(x,\ud y)$. Starting from some initial value $X_0$, the process evolves in a deterministic way following $\Phi(X_0,t)$ until the first jump time $T_1$ which occurs either when the flow reaches the boundary of the state space or before, in a Poisson-like fashion with non homogenous rate $\lambda(\Phi(X_0,t))$. In both cases, the post-jump location of the process at time $T_1$ is governed by the transition distribution $\mathcal{Q}(\Phi(X_0,T_1),\ud y)$ and the motion restarts from this new point as before. This family of stochastic models is well-adapted for tackling various problems arising for example in biology \cite{bertail2010,crudu2012convergence,norris2013exploring,othmer2013excitation,radulescu2007theoremes,RL2014,tindall2008overview}, in neuroscience \cite{genadot2012} or in reliability \cite{fcg,Chiquet,MR2528336,DeS}. Indeed, most of applications involving both deterministic motion and punctual random events may be modeled by a PDMP. Typical examples are growth-fragmentation models composed of deterministic growths followed by random losses. For example \cite{RL2014}, the size of a cell grows exponentially in time, next the cell divides into two offsprings whose size is about half size of the parent cell, and so on.

\smallskip

Proposing efficient statistical methods for this class of stochastic models is therefore of a great interest. Nevertheless, the very particular framework involving both deterministic motion and punctual random jumps imposes to consider specific methods. For instance, the authors of \cite{AzaisSJOS} have shown that the well-known multiplicative intensity model developed by Aalen \cite{AalPHD} for estimating the jump rate function does not directly apply to PDMP's. Alternative approaches should be thus proposed. In the present paper, we focus on the recursive nonparametric estimation of the jump rate of a PDMP from the observation of only one trajectory within a long time interval. More precisely, the purpose of this work is to show how one may obtain by kernel methods a class of consistent estimators for the jump rate, and how one may choose among this class in an optimal way. To the best of our knowledge, the nonparametric estimation of the jump rate in a general framework has never been investigated.

\smallskip

As PDMP's may model a large variety of problems, some methods have been developed by many authors for their statistical inference. As presented before, the randomness of a PDMP is governed by two characteristics: the transition kernel $\mathcal{Q}(x,\ud y)$ and the jump rate $\lambda(x)$. As a consequence, two main questions arise in the estimation problem for such a process, namely the statistical inference for both these features. On the one hand, a few papers investigate some nonparametric methods for estimating the transition function of a PDMP either for a specific model \cite{MR2528336} or in a more general setting for a $d$-dimensional process \cite{AzaisESAIM14}. On the other hand the estimation of the jump rate $\lambda(x)$ or of the associated density function has been more extensively studied by several authors. Without attempting to give an exhaustive survey of the literature on this topic, one may refer the reader to \cite{AzaisSJOS,DoumBer,DoumSiam,Jacobsen,krell} and the references therein. In the book \cite{Jacobsen}, the author studies likelihood processes for observation of PDMP's which could lead to inference methods in a parametric or semi-parametric setting. The papers \cite{DoumBer,DoumSiam} deal with the nonparametric estimation for some PDMP's used in the modeling of a size-structured population observed along a lineage tree. In both these articles, the authors rely on the specific form of the features of the process of interest in order to derive the asymptotic behavior of their estimation procedure. These techniques have been generalized in \cite{krell} to introduce a nonparametric method for estimating the jump rate in a specific class of one-dimensional PDMP's with monotonic motion and deterministic breaks, that is to say when the transition measure $\mathcal{Q}(x,\ud y)$ is a Dirac mass at some location depending on $x$. The procedures developed in these papers \cite{DoumBer,DoumSiam,krell} are obviously of a great interest but strongly use the particular framework involved in the investigated models and are thus not well adapted in a more general setting. In \cite{AzaisSJOS}, the authors show that the famous multiplicative intensity model only applies for estimating the jump rate of a modified version of the underlying PDMP. This leads to a statistical method for approximating the conditional density associated with the jump rate for a process defined on a bounded metric state space.

\smallskip

A main difficulty throughout the present paper and the articles \cite{AzaisESAIM14,AzaisSJOS} is related to the presence of deterministic jumps when the path reaches the boundary of the state space. This feature is often used for modeling a deterministic switching when the quantitative variable rises over a certain threshold \cite{MR1679540,subtilin}. In a statistical point of view, the interarrival times are therefore right-censored by a deterministic clock depending on the state space, which leads to some technical difficulties. We would like to emphasize that the techniques developed in the references \cite{DoumBer,DoumSiam,Jacobsen,krell} do not take into account the likely presence of forced jumps in the dynamic.

\smallskip

One may also find in the literature a few papers \cite{agTest,bertail2008,BdSD,br2012} which focus on the estimation of various functionals for this family of stochastic models. More precisely, the authors of \cite{BdSD,br2012} provide numerical methods for the expectations and for the exit times of PDMP's. In addition, the article \cite{bertail2008} deals with a PDMP introduced for modeling the temporal evolution of exposure to a food contaminant and consider a simulation-based statistical inference procedure for estimating some functionals such as first passage times. Plug-in methods have been studied in \cite{agTest} for a non-ergodic growth-fragmentation model which is absorbed under a certain threshold. In many aspects our approach and these papers are different and complementary. Indeed they are devoted to the estimation of some functionals of PDMP's while we focus on the direct estimation of the primitive data of such a process.

\smallskip

In this article, we introduce a three-dimensional kernel estimator computed from the observation of the embedded Markov chain of a PDMP composed of the post-jump locations $Z_n$ and the travel times $S_{n+1}$ along the path $\Phi(Z_n,t)$. We establish its pointwise consistency as well as its asymptotic normality in Theorem \ref{theo:3d}. The estimate that we consider is recursive: it may be computed in real-time from sequential data, which may be relevant in many applications. We deduce from this result two first corollaries about the nonparametric estimation of the conditional density $f(x,t)$ of the interarrival time $S_{n+1}$ at time $t$ conditionally on the event $\{Z_n=x\}$ (see Corollary \ref{cor:fxt}) and its survival function $G(x,t)=\prob(S_{n+1}>t\,|\,Z_n=x)$ (see Corollary \ref{cor:Gxt}). We also investigate in Corollary \ref{cor:CV:lambdaPhi} the asymptotic behavior of an estimator for the composed function $\lambda\circ\Phi(x,t)$ obtained as the ratio $f(x,t)/G(x,t)$. We derive in $(\ref{eq:def:LAMBDAnx})$ an uncountable class (indexed by the states $\xi$ hitted by the reverse flow $\Phi(x,-t)$ for some time $t$) of consistent estimates of the jump rate $\lambda(x)$. In other words, for each $\xi=\Phi(x,-t)$, we get a good estimate of $\lambda(x)$. We show how one may choose among this class of estimators by minimizing their asymptotic variance. We state in $(\ref{eq:asymptoticvariance})$ that this procedure is equivalent to maximize the criterion $\nu_\infty(\xi)G(\xi,\tau_x(\xi))$ along the curve $\Phi(x,-t)$, i.e. $\xi=\Phi(x,-\tau_x(\xi))$, where $\nu_\infty(\xi)$ denotes the invariant measure of the post-jump locations $Z_n$ and $\tau_x(\xi)$ is the only deterministic time to reach $x$ following $\Phi(\xi,t)$. The choice of this criterion is far to be obvious without precisely computing the limit variance in the central limit theorem presented in Corollary \ref{cor:CV:lambdaPhi}. Indeed, a naive criterion to maximize is the invariant distribution $\nu_\infty(\xi)$ along $\Phi(x,-t)$: the larger $\nu_\infty(\xi)$ is, the larger the number of data around $\xi$ is and the higher the quality of the estimation should be. Nonetheless, this simple criterion does not take into account that the estimate also depends on the time of interest $\tau_x(\xi)$ (see Remark \ref{rem:choicecriterion}). This question is also investigated from a numerical point of view in Subsection \ref{ss:tcp} where we show on synthetic data that the choice of the criterion $\nu_\infty(\xi)G(\xi,\tau_x(\xi))$ is better than the naive one $\nu_\infty(\xi)$. The bandwidths in kernel methods are free parameters that exhibit a strong influence on the quality of the estimation. We discuss the choice of the bandwidth parameters by a classic procedure that consists in minimizing the Integrated Square Error, computed here along the reverse flow $\Phi(x,-t)$: we introduce a cross-validation procedure in this Markov setting and we prove its convergence in Propositions \ref{prop:estimise:kappa} and \ref{prop:estimise:F}. Finally, we would like to highlight that the regularity conditions that we impose are non restrictive. In particular, neither the deterministic exit time from the state space is assumed to be a bounded function, nor the transition kernel is supposed to be lower-bounded, as is the case in \cite{AzaisSJOS} (see eq. (2) and Assumptions 2.4). In addition, the forms of the transition measure and of the deterministic flow are not specified.

\smallskip

The sequel of the paper is organized as follows. We begin in Section \ref{sec:2} with the precise formulation of our framework (see Subsection \ref{ss:def:not}) and the main assumptions that we need in this article (see Subsection \ref{ss:ass}). Section \ref{sec:3} is devoted to the presentation of the statistical procedure and the related results of convergence. More precisely, a three-dimensional kernel estimator for the inter-jumping times is introduced and investigated in Subsection \ref{sec:31}. We derive a class of estimators of the jump rate and propose how to choose among it in Subsection \ref{sec:32}. The crucial choice of the bandwidth parameters is studied in Subsection \ref{sec:33}. Finally, a self-contained presentation of the whole estimation procedure is provided in Subsection \ref{Algo}, then illustrated in the sequel of Section \ref{sec:numericalillustration} on three different application scenarios, with various sample sizes and state space dimensions, involving both simulated and real datasets. More precisely, we focus on the TCP window size process used for modeling data transmission over the Internet in Subsection \ref{ss:tcp}. Estimation of bacterial motility is tackled in Subsection \ref{ss:bac}, while acceleration of fatigue crack propagation is considered in Subsection \ref{ss:fcg}. The proofs and the technicalities are postponed in Appendix \ref{sec:app:1}, \ref{sec:appendix:theo3d} and \ref{sec:app:proofprop} at the end of the paper.


\section{Problem formulation}
\label{sec:2}

This section is devoted to the definition of the PDMP's and to the presentation of the main assumptions that we impose in the paper.

\paragraph{List of notations:} In this paper $\mathcal{B}(\mathbb{R}^d)$ denotes the Borel algebra of $\mathbb{R}^d$ endowed with the Euclidean norm $|\cdot|$. In addition, the Lebesgue measure on $\mathcal{B}(\mathbb{R}^d)$ is denoted by $\lambda_d(\ud x)$, with the particular notation in the one-dimensional case $\lambda_1(\ud x)=\ud x$.  The ball of $\mathbb R^d$ with radius $r$ and center $x$ is denoted by $B_d(x,r)$. The closure of a set $E$ is denoted by $\overline{E}$ while $\partial E$ stands for its boundary.

\subsection{Definition and notation}
\label{ss:def:not}

The motion of a PDMP may be described as the solution of an ordinary differential equation $\Phi$ punctuated at random times by random jumps governed by a transition kernel $\mathcal{Q}$ (see Figure \ref{fig:schpdmp}). The random jumps occur either when the deterministic motion hits the boundary of its state space, or before, with non homogeneous rate $\lambda$ taken along the curve defined by the differential equation $\Phi$.

\smallskip

\begin{figure}[h]
\centering
\includegraphics[width=5cm]{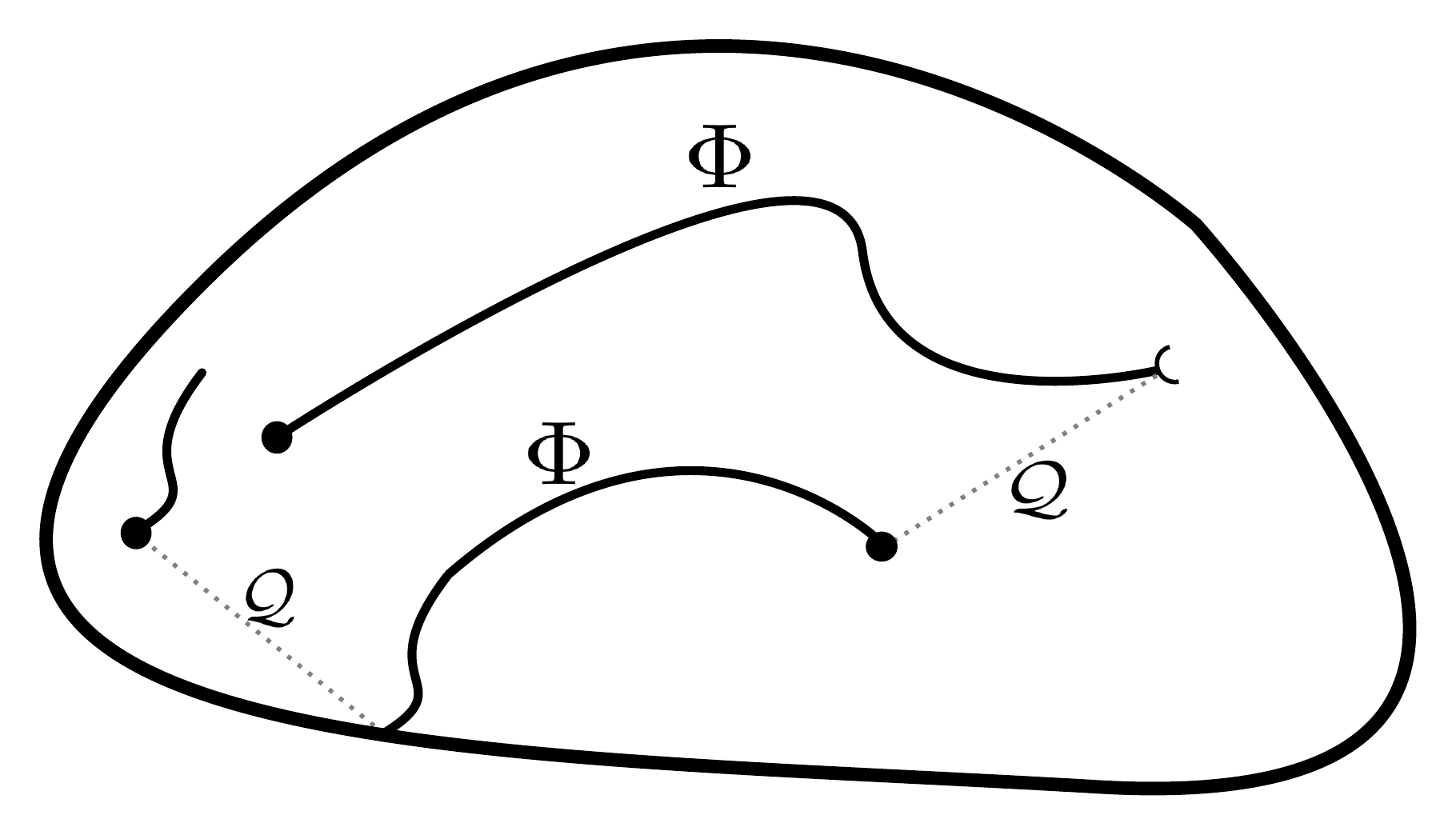}
\caption{Schematic path of a PDMP. This paper is devoted to the nonparametric estimation of the state-de\-pen\-dent rate $\lambda$, that governs the spontaneous generation of jumps, from a long-time trajectory.}
\label{fig:schpdmp}
\end{figure}

\smallskip

\noindent
More precisely, a PDMP $(X_t)_{t\geq0}$ on $(\mathbb{R}^d,\mathcal{B}(\mathbb{R}^d))$ is defined from its three local characteristics $(\lambda,\mathcal{Q},\Phi)$:
\begin{itemize}
\item $\Phi:\mathbb{R}^d\times\mathbb{R}\to\mathbb{R}^d$ is the deterministic flow. It satisfies the semigroup property,
$$\forall\,x\in\mathbb{R}^d,~\forall\,t,\,s\in\mathbb{R},~\Phi(x,t+s) = \Phi(\Phi(x,t),s) .$$
\item $\lambda:\mathbb{R}^d\to\mathbb{R}_+$ is the jump rate.
\item $\mathcal{Q}:(\mathbb{R}^d,\mathcal{B}(\mathbb{R}^d))\to[0,1]$ is the transition kernel.
\end{itemize}
We define the deterministic exit times of $E$ for the flow and for the reverse flow as, for any $x\in E$,
\begin{equation}\label{eq:def:tplustmoins}
t^+(x) = \inf\left\{ t>0~:~\Phi(x,t)\in\partial E\right\}\qquad\text{and}\qquad t^-(x) = \inf\left\{ t>0 ~:~\Phi(x,-t)\in\partial E\right\} .\end{equation}

\smallskip

In all the sequel, we consider a PDMP $(X_t)_{t\geq0}$ evolving on an open subset $E$ of $\mathbb{R}^d$. In this context, we impose as usual \cite[(24.8) Standard conditions]{Dav} that,
$$\forall\,x\in E,~\exists\,\varepsilon>0,~\int_0^\varepsilon \lambda(\Phi(x,t))\,\ud t <+\infty ,$$
and
\begin{equation} \label{eq:conditionQ}
\forall\,x\in\overline{E},~\mathcal{Q}(x,E\setminus\{x\}) = 1.
\end{equation}
In addition, we restrict ourselves to the case where the transition kernel $\mathcal{Q}$ admits a density with respect to the Lebesgue measure,
\begin{equation}\label{eq:Qdensity}
\forall\,x\in\mathbb{R}^d,~\forall\,A\in\mathcal{B}(\mathbb{R}^d),~\mathcal{Q}(x,A) = \int_A\mathcal{Q}(x,y)\,\lambda_d(\ud y).\end{equation}
This assumption is natural when one considers multivariate real-valued PDMP's, and is satisfied in various problems arising in biology \cite{bertail2010}, population dynamics \cite{agTest} or in insurance \cite[(21.11) An insurance model]{Dav}.

\smallskip

Starting from any initial condition $X_0=x$, the motion of $(X_t)_{t\geq0}$ may be described as follows. The distribution of the first jump time $T_1$ is given by,
\begin{equation}\label{eq:T1}
\forall\,t\geq0,~\prob(T_1>t\,|\,X_0=x) =
\left\{
\begin{array}{cl}
\exp\left(-\int_0^t \lambda(\Phi(x,s))\,\ud s\right) & \text{if $t<t^+(x)$},\\
0 & \text{else.}
\end{array}
\right.
\end{equation}
In other words, the process jumps either when the flow hits the boundary of the state space at time $t^+(x)$ or in a Poisson-like fashion with rate $\lambda\circ\Phi$ before. Next the post-jump location $Z_1$ at time $T_1$ is defined through the transition kernel $\mathcal{Q}$: for any test function $\varphi$, we have
$$\mathbb{E}\left[ \varphi(Z_1)~\big|~T_1,\,X_0=x\right]  =  \int \varphi(u)\,\mathcal{Q}( \Phi(x,T_1) , \ud u) . $$
The path between $0$ and the first jump time $T_1$ is given by,
$$\forall\,0\leq t\leq T_1,~X_t=
\left\{
\begin{array}{cl}
\Phi(x,t)&\text{if $t<T_1$,}\\
Z_1 & \text{else.}
\end{array}
\right.
$$

\smallskip

Now starting from the post-jump location $X_{T_1}$, one chooses the next inter-jumping time $S_2=T_2-T_1$ and the future post-jump location $Z_2$ in a similar way as before, and so on. One obtains a strong Markov process with $(T_n)_{n\geq0}$ as the sequence of the jump times (where $T_0=0$ by convention). The inter-jumping times are defined by $S_0=0$ and, for any integer $n\geq1$, $S_n=T_n-T_{n-1}$. Finally $(Z_n)_{n\geq0}$ denotes the stochastic sequence of the post-jump locations of $(X_t)_{t\geq0}$, with for any $n$, $Z_n=X_{T_n}$.

\smallskip

All the randomness of the PDMP $(X_t)_{t\geq0}$ is contained in the stochastic sequence $(Z_n,S_{n+1})_{n\geq0}$ which is a Markov chain. In addition, the post-jump locations $(Z_n)_{n\geq0}$ also form a discrete-time Markov process on the state space $E$ because of the condition $(\ref{eq:conditionQ})$. In this paper, $\nu_n$ denotes the distribution of the $n^{\text{th}}$ post-jump location $Z_n$ for any integer $n\geq0$, while $\mathcal{P}$ denotes its Markov kernel,
\begin{eqnarray}
\forall\,x\in E,~\forall\,A\in\mathcal{B}(\mathbb{R}^d),~\mathcal{P}(x,A) &=& \prob(Z_{n+1}\in A\,|\,Z_n=x)  \nonumber \\
&=&  \int_{\mathbb{R}_+} \mathcal{S}(x,\ud t)\,\mathcal{Q}(\Phi(x,t) , A) , \label{eq:expr:Pint}
\end{eqnarray}
where $\mathcal{S}$ stands for the conditional distribution of $S_{n+1}$ given $Z_n$ for any $n$,
\begin{eqnarray}
\forall\,x\in E,~\forall\,t\geq0,~ \mathcal{S}(x,(t,+\infty)) &=& \prob(S_{n+1}>t ~|~Z_n=x)  \nonumber \\
&=&
\left\{
\begin{array}{cl}
\exp\left(-\int_0^t \lambda(\Phi(x,s))\ud s\right) &\text{if $t<t^+(x)$,}\\
0 &\text{else,}
\end{array}
\right. \label{eq:formulaS}
\end{eqnarray}
in light of $(\ref{eq:T1})$. We would like to highlight that the conditional distribution $\mathcal{S}(x,\cdot)$ is absolutely continuous with respect to the unidimensional Lebesgue measure on $(0,t^+(x))$ with sometimes a singular component at $t^+(x)$,
\begin{equation}\label{eq:S:densityFcirc}\forall\,A\in\mathcal{B}(\mathbb{R}_+),~\mathcal{S}(x,A) = \int_{A\cap(0,t^+(x))} f(x,t) \, \ud t ~ + ~ \mathcal{S}(x,A\cap\{t^+(x)\}),\end{equation}
where the conditional density $f$ may be obtained by deriving $(\ref{eq:formulaS})$,
\begin{equation}
\label{eq:deff}
\forall\,x\in E,~\forall\,0<t<t^+(x),~f(x,t) = \lambda(\Phi(x,t)) \exp\left(-\int_0^t \lambda(\Phi(x,s))\ud s\right) .
\end{equation}
In all the sequel $G$ stands for the conditional survival function associated with $f$, that is,
\begin{equation}\label{eq:defG}
\forall\,x\in E,~\forall\,0<t<t^+(x),~G(x,t) = \mathcal{S}(x,(t,+\infty)) ,
\end{equation}
where $\mathcal{S}$ is defined in $(\ref{eq:formulaS})$. As highlighted before, the process $(Z_n,S_{n+1})_{n\geq0}$ forms a Markov chain on the set $F$ defined by
\begin{equation}
\label{eq:defF}
F = \bigcup_{x\in E}{\{x\}\times[0,t^+(x)]} .
\end{equation}
$\mathcal{R}$ denotes the transition kernel of this process,
\begin{eqnarray}
\forall\,(x,t)\in F,~\forall\,A\times B\in\mathcal{B}(\mathbb{R}^d\times\mathbb{R}_+),~ \mathcal{R}( (x,t) , A\times B) &=& \prob(Z_{n+1}\in A,\,S_{n+2}\in B \,|\,Z_n=x,\,S_{n+1}=t) \nonumber \\
&=&  \int_{A} \mathcal{Q}( \Phi(x,t) ,\ud\xi)\,\mathcal{S}(\xi,B) , \label{eq:def:R}
\end{eqnarray}
and, for any $n$, $\mu_n$ denotes the distribution of the couple $(Z_n,S_{n+1})$.

\subsection{Assumptions}
\label{ss:ass}

The main assumption that we impose in the present paper is a condition of ergodicity on the Markov chain $(Z_n)_{n\geq0}$. This property is often a keystone in statistical inference for Markov processes and may be directly imposed \cite{AzaisESAIM14,AzaisSJOS} or established \cite{krell} from the primitive features of the data.

\smallskip

\begin{hyp}\label{hyp:ergonu}
There exists a distribution $\nu_\infty$ on $E$ such that, for any initial distribution $\nu_0=\delta_{\{x\}}$, $x\in E$,
$$\lim_{n\to+\infty} \|\nu_n - \nu_\infty\|_{TV} =0,$$
where $\|\cdot\|_{TV}$ stands for the total variation norm.
\end{hyp}

\smallskip

\noindent
This assumption may be checked directly on the Markov kernel $\mathcal{P}$ of $(Z_n)_{n\geq0}$ from the existence of a Foster-Lyapunov's function or Doeblin's condition for instance \cite[Theorem 16.0.2]{MandT}. In the following remark we establish a first property of the sequence $(\nu_n)_{n\geq0}$ and of its limit $\nu_\infty$.
\begin{rem}\label{rem:nu:density}
Since the transition kernel $\mathcal{Q}$ is assumed to be absolutely continuous with respect to the Lebesgue measure $(\ref{eq:Qdensity})$, the kernel $\mathcal{P}$ given by $(\ref{eq:expr:Pint})$ of the post-jump locations $(Z_n)_{n\geq0}$ also admits a density. As a consequence, for any integer $n$, the distribution $\nu_n$ of $Z_n$ and thus the invariant measure $\nu_\infty$ introduced in Assumption \ref{hyp:ergonu}  admit a density on the state space $E$. For the sake of clarity, we write $\nu_\infty(\ud x)= \nu_\infty(x)\lambda_d(\ud x)$ with a slight abuse of notation.
\end{rem}

\noindent
We add some regularity conditions on the main features of the process to show the convergence of the estimates in Theorem \ref{theo:3d}.

\begin{hyps}\label{hyps:first:regularity}~
\begin{itemize}
\item The sup-norms $\|\mathcal{Q}\|_{\infty}$ and $\|f\|_{\infty}$ are finite. These conditions are used in the proof of Theorem \ref{theo:3d} to find an upper bound of the non diagonal terms of some square variation process of interest.
\item The functions $\mathcal{Q}$ and $f$ are Lipschitz,
\begin{align*}
\forall\,x,\,y\in E,~\forall\,0<s<t<t^+(x)\wedge t^+(y),~&|f(x,t) - f(y,s)|~\leq~[f]_{Lip} (|t-s| + |y-x|) ,\\
\forall\,x\in\overline{E},~\forall\,y,\,z\in E,~&|\mathcal{Q}(x,y) - \mathcal{Q}(x,z)|~\leq~[\mathcal{Q}]_{Lip} |y-z|.
\end{align*}
These conditions are used in the proof of Theorem \ref{theo:3d} to control the diagonal terms of the same variation process and to study the convergence of some remainder terms.
\item The survival function $G$ is Lipschitz,
$$\forall\,x,\,y\in E,~\forall\,0<t<t^+(x)\wedge t^+(y),~|G(x,t) - G(y,t)| \leq [G]_{Lip}|x-y| .$$
This condition is used in the proof of Theorem \ref{theo:3d} to investigate the convergence of the remainder terms.
\item The deterministic exit time $t^+$ is continuous. This condition is used to find some admissible initial bandwidths $v_0$ and $w_0$.
\end{itemize}
\end{hyps}

\smallskip

\noindent
Finally, we consider an additional condition on both the transition kernels $\mathcal{P}$ and $\mathcal{Q}$ and the flow $\Phi$ in order to ensure the Lipschitz mixing property of the Markov chain $(Z_n)_{n\geq0}$. This will be sufficient to establish the almost sure convergence to $0$ of the remainder term with the adequate rate in the proof of Theorem \ref{theo:3d}.

\smallskip

\begin{hyps}\label{hyp:lipmix}
The transition kernel $\mathcal{P}$ of the Markov chain $(Z_n)_{n\geq0}$ satisfies, for some $a_1\geq1$ and $a_2<1$,
$$\forall\,(x,y)\in E^2,~\int_{\mathbb{R}^d\times\mathbb{R}^d} |u-v|^{a_1}\mathcal{P}(x,\ud u)\mathcal{P}(y,\ud v) \leq a_2 |x-y|^{a_1} .$$
In addition, the composed function $\mathcal{Q}(\Phi(\cdot,\cdot),\cdot)$ belongs to the regularity class $\text{\normalfont Li}(r_1,r_2)$ defined by
$$|\mathcal{Q}(\Phi(x_1,t_1),y_1)-\mathcal{Q}(\Phi(x_2,t_2),y_2)| = O\left(|(x_1,t_1,y_1)-(x_2,t_2,y_2)|^{r_2} (|(x_1,t_1,y_1)|^{r_1}+|(x_2,t_2,y_2)|^{r_1}+1)\right) ,$$
for some positive numbers $r_1$ and $r_2$ satisfying  $2(r_1+r_2)\leq a_1$.
\end{hyps}


\section{Estimation procedure}
\label{sec:3}

\subsection{Inference for the inter-jumping times}
\label{sec:31}

For any integer $n$, we introduce the $\sigma(Z_0,\,S_1,\,\dots,Z_{n-1},S_n)$-measurable functions $\widehat{\nu}_\infty^n$ defined on $E$, and $\widehat{\mathcal{F}}^n$ and $\widehat{\mathcal{G}}^n$ defined on the interior $\mathring{F}$ (let us recall that the set $F$ is given in $(\ref{eq:defF})$), by,
\begin{eqnarray}
\forall\,(x,t)\in\mathring{F},~\widehat{\mathcal{F}}^n(x,t) &=& \frac{1}{n}\sum_{i=0}^{n-1} \frac{1}{v_i^d w_i} \mathbb{K}_d\left(\frac{Z_i-x}{v_i}\right) \mathbb{K}_1\left(\frac{S_{i+1}-t}{w_i}\right), \nonumber \\
\forall\,(x,t)\in\mathring{F},~\widehat{\mathcal{G}}^n(x,t)&=& \frac{1}{n}\sum_{i=0}^{n-1}\frac{1}{v_i^d}\mathbb{K}_d\left(\frac{Z_i-x}{v_i}\right)\mathbb{1}_{\{S_{i+1}>t\}},\label{eq:calGhat}\\
\forall\,x\in E,~\widehat{\nu}^n_\infty(x) &=& \frac{1}{n}\sum_{i=0}^{n-1}\frac{1}{v_i^d}\mathbb{K}_d\left(\frac{Z_i-x}{v_i}\right), \label{eq:def:nuhatn}
\end{eqnarray}
where $\mathbb{K}_p$ denotes a kernel function on $\mathbb{R}^p$, $p\in\{1,d\}$, and the bandwidths are defined for any integer $k$ by $v_k=v_0 (k+1)^{-\alpha}$ and $w_k=w_0 (k+1)^{-\beta}$ for some $\alpha,\,\beta>0$ and initial positive values $v_0$ and $w_0$.

\smallskip

It should be already noted that these quantities are of a great interest in the statistical study of the inter-jumping times of the PDMP $(X_t)_{t\geq0}$. Indeed, we will see that:
\begin{itemize}
\item The ratio $\frac{\widehat{\mathcal{F}}^n(x,t)}{\widehat{\nu}_\infty^n(x)}$ estimates the conditional density $f(x,t)$ defined in $(\ref{eq:deff})$.
\item The ratio $\frac{\widehat{\mathcal{G}}^n(x,t)}{\widehat{\nu}_\infty^n(x)}$ estimates the conditional survival function $G(x,t)$ defined in $(\ref{eq:defG})$.
\item The ratio $\frac{\widehat{\mathcal{F}}^n(x,t)}{\widehat{\mathcal{G}}^n(x,t)}$ estimates the composed function $\lambda(\Phi(x,t))$.
\end{itemize}
In addition, as the name suggests, we will state that $\widehat{\nu}^n_\infty(x)$ is a good estimate of the density $\nu_\infty(x)$ of the unique invariant distribution of the post-jump locations, which is  relevant in the estimation problem for PDMP's but has already been investigated in \cite[Proposition A.11]{AzaisESAIM14}.

\smallskip

In all the sequel, we impose a few assumptions on both the kernel functions $\mathbb{K}_1$ and $\mathbb{K}_d$.

\smallskip

\begin{hyps}\label{hyps:kernel}For any $p\in\{1,d\}$, the kernel function $\mathbb{K}_p$ is assumed to be a nonnegative smooth function satisfying the following conditions:
\begin{itemize}
\item The sup-norm $\|K_p\|_{\infty}$ is finite.
\item $\int_{\mathbb{R}^p}\mathbb{K}_p\ud\lambda_p =1$. 
\item $\text{\normalfont supp}\,\mathbb{K}_p \subset B_p(0_p,\delta)$. Together with well-chosen initial bandwidths, this condition avoids to compute the kernel estimator from data located at the boundary of the state space $F$ (see also Remark \ref{rem:existv0w0}).
\item Only for $p=d$. The function $\mathbb{K}_d$ is Lipschitz,
$$\forall\,x,\,y\in\mathbb{R}^d,~|\mathbb{K}_d(x)-\mathbb{K}_d(y)| \leq [\mathbb{K}]_{Lip}|x-y| .$$
This condition is used to show eq. $(\ref{eq:proof:ise:2})$ in the proof of Proposition \ref{prop:estimise:kappa}.
\end{itemize}
\end{hyps}

\begin{rem}\label{rem:hyps:kernel:tau}
In particular, Assumptions \ref{hyps:kernel} ensure that, for any $p\in\{1,d\}$, $\tau_p^2=\int_{\mathbb{R}^p}\mathbb{K}_p^2\ud\lambda_p$ is finite. This is used to find an upper bound of the non diagonal terms of the variation process in the proof of Theorem \ref{theo:3d}. In addition, the integral $\int_{\mathbb{R}^p}|u|\mathbb{K}_p(u)\ud \lambda_p(u)$ is also finite, which is needed to establish the almost sure convergence of the remainder terms in the proof of the same result.
\end{rem}

\smallskip

\noindent
In the sequel the admissible set for the bandwidth parameters $\alpha$ and $\beta$ is given by
$$\mathcal{A}=\left\{(\alpha,\beta)\in\mathbb{R}^2~:~\alpha>0,~\beta>0,~\alpha d+\beta<1,~\alpha d+\beta+2\min(\alpha,\beta)>1\right\} .$$
In this part, our main result is obtained from the use of vector martingales and is stated in the following theorem.

\smallskip

\begin{theo}\label{theo:3d} For any couple $(x,t)\in\mathring{F}$ such that $\nu_\infty(x)f(x,t)>0$, for any $(\alpha,\beta)\in\mathcal{A}$ and $(v_0,w_0)$ such that
\begin{equation}\label{eq:condition:v0w0}t+w_0\delta\,<\,\inf_{\xi\in B_d(x,v_0\delta)} t^+(\xi) ,\end{equation}
where $\delta$ appears in the third item of Assumptions \ref{hyps:kernel}, we have the almost sure convergence,
$$\left[
\begin{array}{c}
\widehat{\mathcal{F}}^n(x,t)\\
\widehat{\mathcal{G}}^n(x,t)\\
\widehat{\nu}^n_\infty(x)
\end{array}
\right]
\stackrel{a.s}{\longrightarrow}
\left[
\begin{array}{c}
\nu_\infty(x)f(x,t)\\
\nu_\infty(x)G(x,t)\\
\nu_\infty(x)
\end{array}
\right]
$$
and the asymptotic normality,
$$n^{\frac{1-\alpha d-\beta}{2}}\left(
\left[
\begin{array}{c}
\widehat{\mathcal{F}}^n(x,t)\\
\widehat{\mathcal{G}}^n(x,t)\\
\widehat{\nu}^n_\infty(x)
\end{array}
\right]
-
\left[
\begin{array}{c}
\nu_\infty(x)f(x,t)\\
\nu_\infty(x)G(x,t)\\
\nu_\infty(x)
\end{array}
\right]
\right)
\stackrel{d}{\longrightarrow}
\mathcal{N}(0_3,\Sigma(x,t,\alpha,\beta)),
$$
where the variance-covariance matrix $\Sigma(x,t,\alpha,\beta)$ is degenerate with only one positive term at position $(1,1)$. $\Sigma(x,t,\alpha,\beta)$ is defined in $(\ref{eq:def:sigma})$.
\end{theo}
\begin{proof}
The proof is stated in Appendix \ref{sec:appendix:theo3d}.\hfill$\Box$
\end{proof}

\smallskip

\begin{rem}\label{rem:existv0w0}
The existence of a couple $(v_0,w_0)$ satisfying $(\ref{eq:condition:v0w0})$ is obvious whenever the exit time $t^+$ is continuous (see Assumptions \ref{hyps:first:regularity}). This condition ensures that all the inter-jumping times used in the calculus of $\widehat{\mathcal{F}}^n$ and $\widehat{\mathcal{G}}^n$ are not obtained from forced jumps when the process reaches the boundary of the state space. In the case where $t^+(x)=\infty$, it is obvious that no interarrival times are right-censored. The consistency and the asymptotic normality are therefore still accurate without any condition on $(v_0,w_0)$.
\end{rem}

\medskip

\begin{rem}\label{rem:uniform:v0w0}
In Theorem \ref{theo:3d}, the choice $(\ref{eq:condition:v0w0})$ of the initial bandwidths $v_0$ and $w_0$ is locally dependent on the point of interest. This may appear restrictive but may be avoided by considering the elements of
\begin{equation}\label{eq:definition:calC}
\mathcal{C}=\Big\{C\times[0,T]~:~\text{\normalfont$C$ is a compact subset of $E$ and~} T<\inf_{x\in C}t^+(x)\Big\}.
\end{equation}
Indeed, for any $C\times[0,T]\in\mathcal{C}$, there always exists a couple $(v_0,w_0)$ such that
$$T+w_0\delta~<~\inf_{x\in C}\,\inf_{\xi\in B_d(x,v_0\delta)}\,t^+(\xi) .$$
Thus, $(v_0,w_0)$ satisfies $(\ref{eq:condition:v0w0})$ for any point $(x,t)\in C\times[0,T]$.
\end{rem}

\medskip

\begin{rem}
\label{rem:2d}
The variance-covariance matrix appearing in the asymptotic normality presented in Theorem \ref{theo:3d} is degenerate with only the component $(1,1)$ positive. It means that the rate of the estimators $\widehat{\nu}^n_\infty$ and $\widehat{\mathcal{G}}^n$ is faster than the one of $\widehat{\mathcal{F}}^n$. This is straightforward because $\widehat{\mathcal{F}}^n$ is obtained by smoothing the empirical distribution of the data both in the spatial and temporal directions contrary to $\widehat{\nu}^n_\infty$ and $\widehat{\mathcal{G}}^n$. The proof of the previous result may be adapted to show the two-dimensional central limit theorem, with $\alpha$ such that $\alpha d<1$ and $\alpha(d+2)>1$:
$$
n^{\frac{1-\alpha d}{2}}\left(
\left[
\begin{array}{c}
\widehat{\mathcal{G}}^n(x,t)\\
\widehat{\nu}^n_\infty(x)
\end{array}
\right]
-
\left[
\begin{array}{c}
\nu_\infty(x)G(x,t)\\
\nu_\infty(x)
\end{array}
\right]
\right)
\stackrel{d}{\longrightarrow}
\mathcal{N}(0_2,\Sigma'(x,t,\alpha)),
$$
where $\Sigma'(x,t,\alpha)$ is a diagonal $2\times2$-matrix. The keystone to state this convergence is the behavior given in $(\ref{eq:crochetM:limite})$ of the hook of a vector martingale.
\end{rem}

\medskip

If we assume the geometric ergodicity of the Markov chain $(Z_n)_{n\geq0}$, one may also obtain the rate of convergence of the variances of the estimates $\widehat{\mathcal{F}}^n$, $\widehat{\mathcal{G}}^n$ and $\widehat{\nu}^n_{\infty}$ uniformly on any compact subset of $\mathring{F}$ such that the parameters $v_0$ and $w_0$ may be uniformly chosen (see Remark \ref{rem:uniform:v0w0}).

\smallskip

\begin{prop}\label{prop:variances}
Let us assume that there exists $b>1$ such that $\|\nu_n - \nu_\infty\|_{TV}=O(b^{-n})$. The geometric ergodicity is in particular ensured by Doeblin's condition (see \cite[Theorem 16.0.2]{MandT}). Then, for any set $C\times[0,T]\in\mathcal{C}$ and for any couple $(\alpha,\beta)$ such that $2(\alpha d+\beta) <1$, we have
\begin{eqnarray*}
\sup_{(x,t)\in C}\mathbb{V}\text{\normalfont ar}(\widehat{\mathcal{F}}^n(x,t)) &=& O(n^{2(\alpha d+\beta)-1}) ,\\
\sup_{(x,t)\in C}\left(\mathbb{V}\text{\normalfont ar}(\widehat{\mathcal{G}}^n(x,t))+\mathbb{V}\text{\normalfont ar}(\widehat{\nu}_\infty^n(x))\right)&=&O(n^{2\alpha d-1}).
\end{eqnarray*}
Let us recall that $\mathcal{C}$ has been defined in $(\ref{eq:definition:calC})$. The rate of convergence for $\widehat{\mathcal{F}}^n$ is faster than the one given in Theorem \ref{theo:3d} whenever $3(\alpha d+\beta)<1$.
\end{prop}
\begin{proof}
The proof is similar to the demonstrations of Proposition B.5 and Corollary B.6 of \cite{AzaisESAIM14} and relies on the control of the covariance process of functionals of a geometrically ergodic Markov chain (see \cite[Theorem 16.1.5]{MandT}).\hfill$\Box$
\end{proof}

\medskip

We present in the sequel some corollaries of Theorem \ref{theo:3d} that are of interest in the estimation problem for the inter-jumping times. First, we define the estimator $\widehat{f}^n(x,t)$ of $f(x,t)$ by,
\begin{equation*}\forall\,(x,t)\in\mathring{F},~\widehat{f}^n(x,t) = \frac{\widehat{\mathcal{F}}^n(x,t)}{\widehat{\nu}^n_\infty(x)} ,\end{equation*}
with the usual convention $0/0=0$. We have the following result of convergence.

\smallskip

\begin{cor}\label{cor:fxt}
For any couple $(x,t)\in\mathring{F}$ such that $\nu_\infty(x)f(x,t)>0$, for any $(\alpha,\beta)\in\mathcal{A}$ and $(v_0,w_0)$ satisfying $(\ref{eq:condition:v0w0})$, we have
$$\widehat{f}^n(x,t)\stackrel{a.s.}{\longrightarrow} f(x,t)\qquad\text{and}\qquad n^{\frac{1-\alpha d-\beta}{2}} \left(\widehat{f}^n(x,t) - f(x,t)\right) \stackrel{d}{\longrightarrow}\mathcal{N}\left(0 , \frac{\tau_1^2\,\tau_d^2\,f(x,t)}{(1+\alpha d +\beta)\nu_\infty(x)}\right).$$
\end{cor}
\begin{proof}
This result is a direct application of Theorem \ref{theo:3d} and Slutsky's lemma.\hfill$\Box$
\end{proof}

\smallskip

\noindent
Another feature of interest for the inter-jumping times is the survival function $G$. One may estimate this quantity by,
\begin{equation*}
\forall\,(x,t)\in\mathring{F},~\widehat{G}^n(x,t) = \frac{\widehat{\mathcal{G}}^n(x,t)}{\widehat{\nu}^n_\infty(x)},\end{equation*}
with the convention $0/0=0$. Some properties of convergence are stated in the following corollary.

\smallskip

\begin{cor}\label{cor:Gxt}
For any couple $(x,t)\in\mathring{F}$ such that $\nu_\infty(x)f(x,t)>0$, for any $\alpha$ such that $\alpha d<1$ and $\alpha(d+2)>1$, and any $v_0$, we have
$$\widehat{G}^n(x,t)\stackrel{a.s.}{\longrightarrow} G(x,t)\qquad\text{and}\qquad n^{\frac{1-\alpha d}{2}} \left(\widehat{G}^n(x,t) - G(x,t)\right) \stackrel{d}{\longrightarrow}\mathcal{N}\left(0 , \frac{\tau_d^2\,G(x,t)}{(1+\alpha d)\nu_\infty(x)}\right).$$
\end{cor}
\begin{proof}
The almost sure convergence is a direct application of Theorem \ref{theo:3d}. The central limit theorem is a consequence of the asymptotic normality established in Remark \ref{rem:2d} and Slutsky's lemma.\hfill$\Box$
\end{proof}

\smallskip


\subsection{Optimal estimation of the jump rate}
\label{sec:32}

We propose to estimate the composed function $\lambda\circ\Phi$ by the ratio $\widehat{\lambda\circ\Phi}^n$ defined by,
\begin{equation}\label{eq:def:lambdaPhi}
\forall\,(x,t)\in\mathring{F},~\widehat{\lambda\circ\Phi}^n(x,t) = \frac{\widehat{\mathcal{F}}^n(x,t)}{\widehat{\mathcal{G}}^n(x,t)} ,\end{equation}
again with the convention $0/0=0$. Pointwise convergence and asymptotic normality are again a consequence of Theorem \ref{theo:3d}.

\smallskip

\begin{cor}\label{cor:CV:lambdaPhi}
For any couple $(x,t)\in\mathring{F}$ such that $\nu_\infty(x)f(x,t)>0$, for any $(\alpha,\beta)\in\mathcal{A}$ and $(v_0,w_0)$ satisfying $(\ref{eq:condition:v0w0})$, we have
$$\widehat{\lambda\circ\Phi}^n(x,t)\stackrel{a.s.}{\longrightarrow} \lambda \circ \Phi(x,t) ,$$
and
\begin{equation*}\label{eq:tcl:lambdaPhi}
n^{\frac{1-\alpha d-\beta}{2}} \left(\widehat{\lambda\circ\Phi}^n(x,t) - \lambda\circ\Phi(x,t)\right) \stackrel{d}{\longrightarrow}\mathcal{N}\left(0 , \frac{\tau_1^2\,\tau_d^2\,\lambda\circ\Phi(x,t)}{(1+\alpha d +\beta)\nu_\infty(x) G(x,t)}\right).
\end{equation*}
\end{cor}
\begin{proof}
This result is a direct application of Theorem \ref{theo:3d} and Slutsky's lemma together with $(\ref{eq:formulaS})$, $(\ref{eq:deff})$ and $(\ref{eq:defG})$.\hfill$\Box$
\end{proof}

\smallskip

In all the sequel, we focus on the estimation of the rate $\lambda$ (and not on the composed function $\lambda\circ\Phi$) from the estimate $\widehat{\lambda\circ\Phi}^n$ defined in $(\ref{eq:def:lambdaPhi})$. In particular, for some fixed value $x\in E$, we introduce a class of estimators $\widehat{\lambda}^n_\xi(x)$ of $\lambda(x)$ indexed by the elements $\xi$ of the curve $\mathcal{C}_x$ described by the reverse flow $\Phi(x,-t)$, $t\geq0$, and we propose to choose an estimate from this class in an optimal way. Before proceeding further we define the notation $\mathcal{C}_x$ as
$$\mathcal{C}_x=\{\Phi(x,-t)~:~0\leq t< t^-(x)\}.$$
By definition $(\ref{eq:def:tplustmoins})$ of $t^-(x)$ we have $\mathcal{C}_x\subset E$. In addition, for any $\xi\in\mathcal{C}_x$, we define $\tau_x(\xi)$ as the unique time satisfying
\begin{equation*}
\Phi(\xi,\tau_x(\xi)) = x,~\xi\in\mathcal{C}_x.
\end{equation*}
Thus we have the following trivial result,
\begin{equation}\label{eq:lambdaPhilambda}
\forall\,\xi\in\mathcal{C}_x,~\lambda\circ\Phi(\xi,\tau_x(\xi)) = \lambda(x).
\end{equation}
As a consequence, we propose to define a class of estimators of $\lambda(x)$ by
\begin{equation}\label{eq:def:LAMBDAnx}
\Lambda^n(x) = \left\{ \widehat{\lambda}^n_\xi(x)= \widehat{\lambda\circ\Phi}^n(\xi,\tau_x(\xi))~:~\xi\in\mathcal{C}_x\right\},
\end{equation}
where $\widehat{\lambda\circ\Phi}^n$ estimates $\lambda\circ\Phi$ (see Corollary \ref{cor:CV:lambdaPhi}) and has already been defined in $(\ref{eq:def:lambdaPhi})$. By virtue of Corollary \ref{cor:CV:lambdaPhi} and only using $(\ref{eq:lambdaPhilambda})$, one has, as $n$ goes to infinity,
$$\forall\,\xi\in\mathcal{C}_x,\quad\widehat{\lambda}_\xi^n(x) \stackrel{a.s.}{\longrightarrow} \lambda(x)\quad\text{and}\quad n^{\frac{1-\alpha d-\beta}{2}}\left(\widehat{\lambda}_\xi^n(x)-\lambda(x)\right) \stackrel{d}{\longrightarrow}\mathcal{N}\left(0 , \sigma_{\xi,\alpha,\beta}^2(x)\right),$$
where the asymptotic variance is given by
\begin{equation}\label{eq:asymptoticvariance}
\sigma_{\xi,\alpha,\beta}^2(x) = \frac{\tau_1^2\,\tau_d^2\,\lambda(x)}{(1+\alpha d +\beta)\kappa_x(\xi)},\qquad\text{with}\quad\kappa_x(\xi)=\nu_\infty(\xi) G(\xi,\tau_x(\xi)) .
\end{equation}


\noindent
In this paper, we choose to approximate $\lambda(x)$ by the element $\widehat{\lambda}^n_\xi(x)\in\Lambda^n(x)$ minimizing the asymptotic variance $\sigma_{\xi,\alpha,\beta}^2(x)$. In other words, our optimal estimator of the jump rate $\lambda(x)$ is obtained as
\begin{equation*}\label{eq:def:lambdastar}
\widehat{\lambda}^n_\ast(x)=\widehat{\lambda}^n_{\xi_\ast}(x) ,\qquad\text{where}\quad\xi_\ast = \argmax_{\xi\in\mathcal{C}_x}~\kappa_x(\xi).
\end{equation*}

\smallskip

\begin{rem}\label{rem:choicecriterion}
A good criterion should be to maximize the invariant measure $\nu_\infty(\xi)$, $\xi\in\mathcal{C}_x$. Indeed, if $\nu_\infty(\xi)$ is large, a large frequency of post-jump locations around $\xi$ may be available in the dataset. Nevertheless, roughly speaking, the quantities of interest $f(\xi,\tau_x(\xi))$ and $G(\xi,\tau_x(\xi))$ are well estimated if a large number of post-jump locations are around $\xi$ together with inter-jumping times around $\tau_x(\xi)$. This naive criterion may be corrected by including the quality of the estimation at time $\tau_x(\xi)$, that is to say, by maximizing the product $\nu_\infty(\xi) G(\xi,\tau_x(\xi))$. We also refer the reader to the simulation study presented in Section \ref{sec:numericalillustration} and more precisely to Figure \ref{fig:boxplotresult}.
\end{rem}

\smallskip

The criterion $\kappa_x(\xi)$, $\xi\in\mathcal{C}_x$, is generally uncomputable from the known features of the PDMP and thus remains to be estimated. In light of Theorem \ref{theo:3d} and by definition $(\ref{eq:asymptoticvariance})$ of $\kappa_x(\xi)$, we naturally propose to approximate this quantity by,
\begin{equation}\label{eq:def:kappahatn}
\forall\,\xi\in\mathcal{C}_x,~\widehat{\kappa}_x^n(\xi)= \widehat{\mathcal{G}}^n(\xi,\tau_x(\xi)) .
\end{equation}
As a consequence, we propose to estimate the jump rate $\lambda(x)$ by its statistical approximation in $\Lambda^n(x)$ maximizing the estimated criterion $\widehat{\kappa}_x(\xi)$. More precisely,
\begin{equation}\label{eq:maxkappan}
\widehat{\widehat{\lambda}}^n_\ast(x)=\widehat{\lambda}^n_{\widehat{\xi}_\ast^n}(x) ,\qquad\text{where}\quad\widehat{\xi}^n_\ast = \argmax_{\xi\in\mathcal{C}_x}~\widehat{\kappa}^n_x(\xi).
\end{equation}
The high oscillations or alternatively the high smoothness of a kernel estimator with a ill-chosen bandwidth suggest that the choice of the parameter $\alpha$ appearing in $\widehat{\kappa}_x^n$ (see eq. $(\ref{eq:calGhat})$ and $(\ref{eq:def:kappahatn})$) is crucial at this maximization step.

\subsection{How to choose bandwidth parameters $\alpha$ and $\beta$?}
\label{sec:33}

This part is devoted to the choice of the bandwidth parameters $\alpha$ and $\beta$. The criteria that we introduce in this part to choose these features are defined as line integrals along the curve $\mathcal{C}_x$. As a consequence, we need to ensure that this kind of quantity is well defined in our setting.

\begin{hyp}
For any starting condition $\xi\in E$, the reverse flow $t\mapsto\Phi(\xi,-t)$ defines a change of variable, that is to say, is a diffeomorphic mapping.
\end{hyp}

It is common in the literature to minimize the Integrated Square Error (ISE) to choose the optimal bandwidths of a kernel estimator, in particular from dependent data: one may refer the reader to \cite{hart1990data,kim1997asymptotically,kim1996bandwidth}. Another classical solution is to investigate the behavior of the Mean Integrated Square Error (MISE). In the framework of Gaussian dependent data, the authors of \cite{claeskens2002effect} have shown that the optimal bandwidths obtained by minimizing the ISE and the MISE are very close if the dependence is of short range, that is to say, if the covariance function is integrable (see \cite[Theorem 2.2]{claeskens2002effect}). It should be noted that a geometric ergodic Markov chain satisfies this kind of condition (see \cite[Theorem 16.1.5]{MandT}).

\smallskip

Let us recall that $\alpha$ appears at first in the computation of the estimated criterion $\widehat{\kappa}^n_x$ which we need to maximize along the curve $\mathcal{C}_x$. Indeed $\widehat{\kappa}_x^n$ is computed $(\ref{eq:def:kappahatn})$ from the estimate $\widehat{\mathcal{G}}^n$ which implicitly depends on $\alpha$.
We propose to choose the bandwidth parameter $\alpha$ by minimizing the ISE associated with $\widehat{\kappa}_x$ and defined by
\begin{eqnarray}
\text{\normalfont ISE}^n_\kappa(\alpha) &=&\int_{\mathcal{C}_x} \left(\widehat{\kappa}_x^n(\xi) - \kappa_x(\xi)\right)^2\ud\xi \nonumber \\
&=& \int_{\mathcal{C}_x} \kappa_x(\xi)^2\ud\xi \,+\, \varepsilon_\kappa^n(\alpha) , \label{eq:isekappa:dev}
\end{eqnarray}
where the function $\varepsilon^n_\kappa$ is given by
\begin{equation}
\label{eq:def:varepsilon}
\varepsilon_\kappa^n(\alpha) = \int_{\mathcal{C}_x}\widehat{\mathcal{G}}^n(\xi,\tau_x(\xi))^2\ud\xi- 2\,\int_{\mathcal{C}_x} \widehat{\mathcal{G}}^n(\xi,\tau_x(\xi)) \, \kappa_x(\xi) \ud\xi.
\end{equation}

\noindent 
One may remark that here the ISE is unusually computed along a curve of interest. In $(\ref{eq:isekappa:dev})$ the dependency on $\alpha$ only holds through the function $\varepsilon_\kappa^n$. As a consequence the optimal parameter 
$\alpha$ minimizing this stochastic function $(\ref{eq:def:varepsilon})$ also minimizes the ISE.

\smallskip

The function $\varepsilon^n_\kappa$ is generally not computable since the unknown quantity $\kappa_x$ appears in its definition. As a consequence, we propose to estimate $\varepsilon^n_\kappa$ by cross-validation which is a popular technique for selecting the bandwidth that minimizes the ISE. The authors would like to highlight that cross-validation involves here two main difficulties. First the estimators are computed from dependent data which are not identically distributed. In addition, there is almost surely no data on the set of integration $\mathcal{C}_x$ whenever the dimension $d$ is larger than $2$. That is why we propose a specific procedure adapted to this framework.

\smallskip

Before defining our cross-validation estimate of $\varepsilon^n_\kappa$, we need to introduce the quantities $\mathbb{T}_{x,\rho}$ and $\theta_x$. First $\mathbb{H}_x$ denotes the hyperplane orthogonal to $\mathcal{C}_x$ at $x$, that is,
\begin{equation*}
\mathbb{H}_x = \left\{y\in\mathbb{R}^d ~:~y-x \perp \nabla_t\Phi(x,0)\right\}.
\end{equation*}
In addition, for any $\rho>0$, we introduce the notation $\mathbb{D}_{x,\rho}$ for
\begin{equation*}\label{eq:Dxrho}
\mathbb{D}_{x,\rho} = B_d(x,\rho)\cap\mathbb{H}_x .\end{equation*}
Furthermore $\mathbb{T}_{x,\rho}$ denotes the tube around $\mathcal{C}_x$ with radius $\rho$,
\begin{equation}\label{eq:Txrho}
\mathbb{T}_{x,\rho} = \bigcup_{y\in\mathbb{D}_{x,\rho}}\mathcal{C}_y .
\end{equation}
Finally, for any $\xi\in\mathbb{T}_{x,\rho}$, $\theta_x(\xi)$ denotes the unique time such that $\Phi(\xi,\theta_x(\xi))\in\mathbb{D}_{x,\rho}$. In particular,
\begin{equation}\label{eq:tautheta}
\forall\,\xi\in\mathcal{C}_x,~\theta_x(\xi) = \tau_x(\xi).
\end{equation}
It should be noted that, for $\rho$ small enough, $\mathbb{T}_{x,\rho}\subset E$.

\smallskip

We focus now on a cross-validation method for estimating the quantity $\varepsilon^n_\kappa$. The estimate $\widehat{\varepsilon}^{\,n,\widetilde{n},\rho}_\kappa$ of $\varepsilon^{\,n}_\kappa$ is defined from the observation of the embedded Markov chain $(\widetilde{Z}_k,\widetilde{S}_{k+1})_{k\geq0}$ of another PDMP $(\widetilde{X}_t)_{t\geq0}$ (independent on the first one $(X_t)_{t\geq0}$ and distributed according to the same parameters), by
\begin{equation}\label{eq:estim:ise:kappa:CV}
\widehat{\varepsilon}^{\,n,\widetilde{n},\rho}_\kappa(\alpha) = \int_{\mathcal{C}_x}\widehat{\kappa}_x^n(\xi)^2\ud\xi -  \frac{2\,\Gamma\left(\frac{d-1}{2}+1\right)}{\widetilde{n}\,\pi^{\frac{d-1}{2}}\rho^{d-1}}\sum_{k=0}^{\widetilde{n}-1} \widehat{\mathcal{G}}^{n}\!\!\left(\widetilde{Z}_k,\theta_x(\widetilde{Z}_k)\right)\,\mathbb{1}_{\mathbb{T}_{x,\rho}}(\widetilde{Z}_k)\,\mathbb{1}_{(\theta_x(\widetilde{Z}_k),+\infty)}(\widetilde{S}_{k+1}) ,\end{equation}
where $\Gamma$ denotes as usually the Euler function. Some regularity conditions are necessary to investigate the asymptotic behavior of the cross-validation estimate $\widehat{\varepsilon}^{\,n,\widetilde{n},\rho}_\kappa$ in Proposition \ref{prop:estimise:kappa}.

\smallskip

\begin{hyps}\label{hyps:ise:regularity}~
\begin{itemize}
\item The sup-norm $\|\nu_\infty\|_{\infty}$ is finite and $\nu_\infty$ is Lipschitz,
$$\forall\,x,\,y\in E,~|\nu_\infty(x)-\nu_\infty(y)|\leq[\nu_{\infty}]_{Lip}|x-y|.$$
\item The deterministic exit time for the reverse flow $t^-$ is Lipschitz,
$$\forall\,x,\,y\in E,~|t^-(x)-t^-(y)|\leq [t^-]_{Lip}|x-y|.$$
\item The flow $\Phi$ is Lipschitz,
$$\forall\,x,\,y\in E,~\forall\,t\in\mathbb{R},~|\Phi(x,t)-\Phi(y,t)|\leq[\Phi]_{Lip}|x-y|.$$
\item The sup-norm $\|\nabla_t\Phi\|_{\infty}$ is finite and $\nabla_t\Phi$ is Lipschitz,
$$\forall\,x,\,y\in E,~\forall\,t\in\mathbb{R},~|\nabla_t\Phi(x,t)-\nabla_t\Phi(y,t)|\leq[\nabla_t\Phi]_{Lip}|x-y|.$$
\end{itemize}
\end{hyps}

\smallskip

\begin{prop}\label{prop:estimise:kappa}
Conditionally to $\sigma(Z_0,\,S_1,\,\dots,Z_{n-1},\,S_n)$, we have
$$\lim\limits_{\substack{\widetilde{n}\to\infty \\ \rho\to0}} \widehat{\varepsilon}^{\,n,\widetilde{n},\rho}_\kappa(\alpha) ~ = ~ \varepsilon_\kappa^n(\alpha)\quad a.s.$$
\end{prop}
\begin{proof}The proof is stated in Appendix \ref{sec:app:proofprop}.\hfill$\Box$
\end{proof}

\smallskip

By virtue of Proposition \ref{prop:estimise:kappa}, one may obtain an estimate of the optimal bandwidth parameter $\alpha$ arising in $\widehat{\kappa}^n_x$ by minimizing the quantity $\widehat{\varepsilon}^{\,n,\widetilde{n},\rho}_\kappa$ for some small enough $\rho$ and large enough $\widetilde{n}$. In addition, by $(\ref{eq:def:lambdaPhi})$, $(\ref{eq:def:LAMBDAnx})$ and $(\ref{eq:def:kappahatn})$, the quantity $\widehat{\kappa}^n_x$ also appears in the calculus of the estimator $\widehat{\lambda}^n_\xi(x)$. In particular, the same choice of $\alpha$ may be done for computing the denominator $\widehat{\kappa}^n_x(\xi)$ of $\widehat{\lambda}^n_\xi(x)$. As a consequence it remains to choose in an optimal way the bandwidth parameters $\alpha$ and $\beta$ arising in the formula $(\ref{eq:def:lambdaPhi})$ of the numerator $\widehat{\mathcal{F}}^n(\xi,\tau_x(\xi))$ of $\widehat{\lambda}^n_\xi(x)$ $(\ref{eq:def:LAMBDAnx})$.

\smallskip

In a similar way as before, we propose to choose $\alpha$ and $\beta$ by minimizing the ISE associated with $\widehat{\mathcal{F}}^n(\cdot,\tau_x(\cdot))$ and computed along the curve $\mathcal{C}_x$,
\begin{eqnarray}
\text{ISE}_\mathcal{F}^n(\alpha,\beta) &=& \int_{\mathcal{C}_x} \left(\widehat{\mathcal{F}}^n(\xi,\tau_x(\xi)) - \mathcal{F}(\xi,\tau_x(\xi))\right)^2 \ud\xi \nonumber\\
&=&  \int_{\mathcal{C}_x} \mathcal{F}(\xi,\tau_x(\xi))^2\ud\xi +  \varepsilon_\mathcal{F}^n(\alpha,\beta) , \label{eq:ise:estim:F:CV}
\end{eqnarray}
where $\widehat{\mathcal{F}}^n$ implicitly depends on $\alpha$ and $\beta$, $\mathcal{F}(\xi,\tau_x(\xi))$ stands for $\nu_\infty(\xi)f(\xi,\tau_x(\xi))$ and $\varepsilon^n_\mathcal{F}$ is given by
\begin{equation*}\label{eq:def:varepsilonF}\varepsilon_\mathcal{F}^n(\alpha,\beta) = \int_{\mathcal{C}_x} \widehat{\mathcal{F}}^n(\xi,\tau_x(\xi))^2 \ud\xi  - 2  \int_{\mathcal{C}_x} \widehat{\mathcal{F}}^n(\xi,\tau_x(\xi))\mathcal{F}(\xi,\tau_x(\xi))\ud\xi.\end{equation*}
As in the previous part we propose to estimate $\varepsilon_\mathcal{F}^n$ by cross-validation from the observation of the embedded chain $(\widetilde{Z}_k,\widetilde{S}_{k+1})_{k\geq0}$ of another PDMP $(\widetilde{X}_t)_{t\geq0}$. We define our estimate $\widehat{\varepsilon}^{\,n,\widetilde{n},\rho_1,\rho_2}_\mathcal{F}$ of $\varepsilon_\mathcal{F}^n$ by
\begin{align}
\widehat{\varepsilon}^{\,n,\widetilde{n},\rho_1,\rho_2}_\mathcal{F}(\alpha,\beta) =&\phantom{-}\,\int_{\mathcal{C}_x}\widehat{\mathcal{F}}^n(\xi,\tau_x(\xi))^2\ud\xi \nonumber
\\&-\,\frac{2\,\Gamma\left(\frac{d-1}{2}+1\right)}{\widetilde{n}\,\rho_2\,\pi^{\frac{d-1}{2}}\rho_1^{d-1}}\sum_{k=0}^{\widetilde{n}-1} \widehat{\mathcal{F}}^{n}\!\!\left(\widetilde{Z}_k,\theta_x(\widetilde{Z}_k)\right)\,\mathbb{1}_{\mathbb{T}_{x,\rho_1}}(\widetilde{Z}_k)\,\mathbb{1}_{\left(\theta_x(\widetilde{Z}_k)-\frac{\rho_2}{2},\theta_x(\widetilde{Z}_k)+\frac{\rho_2}{2}\right)}(\widetilde{S}_{k+1}).\label{eq:vareps:F:CV}
\end{align}
The convergence of $\widehat{\varepsilon}^{\,n,\widetilde{n},\rho,\delta}_\mathcal{F}$ is investigated in Proposition \ref{prop:estimise:F}.

\smallskip

\begin{prop}\label{prop:estimise:F}
Conditionally to $\sigma(Z_0,\,S_1,\,\dots,Z_{n-1},\,S_n)$, we have
$$\lim\limits_{\substack{\widetilde{n}\to\infty \\ \rho_{1,2}\to0}} \widehat{\varepsilon}^{\,n,\widetilde{n},\rho_1,\rho_2}_\mathcal{F}(\alpha,\beta) ~ = ~ \varepsilon_\mathcal{F}^n(\alpha,\beta)\quad a.s.$$
\end{prop}
\begin{proof}The proof is similar to the demonstration of Proposition \ref{prop:estimise:kappa} stated in Appendix \ref{sec:app:proofprop}.\hfill$\Box$\end{proof}


\section{Simulation study}
\label{sec:numericalillustration}

In this section, we provide a self-contained presentation of the estimation procedure, as well as three application scenarios on both simulated and real datasets.

\subsection{Estimation algorithm}
\label{Algo}

The sequel is devoted to the self-contained presentation of the estimation procedure provided in this paper. Precisely, we are interested in the estimation of $\lambda(x)$ for some $x\in E\subset\mathbb{R}^d$.


\subsubsection{Preliminary computations}

\noindent
These preliminary computations only require to manipulate the flow $\Phi$ and the state space $E$.

\begin{itemize}
\item Compute the curve $\mathcal{C}_x = \{\Phi(x,-t)~:~t\geq0\}\cap E$.
\item Choose $\rho>0$ and compute
\begin{eqnarray*}
\mathbb{D}_{x,\rho} &=&B_d(x,\rho) \cap\{y\in\mathbb{R}^d~:~y-x\perp\nabla_t\Phi(x,0)\} ,\\
\mathbb{T}_{x,\rho}&=&\bigcup_{y\in\mathbb{D}_{x,\rho}} \mathcal{C}_y .
\end{eqnarray*}
\item
For any $\xi\in\mathbb{T}_{x,\rho}$ compute $\theta_x(\xi)$ as the unique solution of
$$\Phi(\xi,\theta_x(\xi)) \in \mathbb{D}_{x,\rho}.$$
The mapping $\tau_x$ is only defined on $\mathcal{C}_x\subset\mathbb{T}_{x,\rho}$ by $\tau_x(\xi)=\theta_x(\xi)$.
\end{itemize}

\subsubsection{Preliminary estimates}

\noindent
These preliminary computations require to choose kernel functions $\mathbb{K}_1$ and $\mathbb{K}_d$ on $\mathbb{R}$ and $\mathbb{R}^d$, as well as two positive values $v_0$ and $w_0$.
\begin{itemize}
\item For any couple $(\alpha,\beta)$, compute
$$\widehat{\mathcal{F}}^n_{\alpha,\beta}(\xi,t) = \frac{1}{nv_0^dw_0}\sum_{i=0}^{n-1} (i+1)^{\alpha d+\beta} \mathbb{K}_d\left(\frac{Z_i-\xi}{v_i}\right) \mathbb{K}_1\left(\frac{S_{i+1}-t}{w_i}\right).$$
\item Compute
$$\widehat{\mathcal{G}}^n_\alpha(\xi,t)= \frac{1}{nv_0^d}\sum_{i=0}^{n-1}(i+1)^{\alpha d}\mathbb{K}_d\left(\frac{Z_i-\xi}{v_i}\right)\mathbb{1}_{\{S_{i+1}>t\}}.$$
\end{itemize}

\subsubsection{Choice of bandwidth parameters by cross-validation}

\noindent
From the observation of two independent embedded chains $(Z_i,S_{i+1})_{0\leq i\leq n-1}$ and $(\widetilde{Z}_i,\widetilde{S}_{i+1})_{0\leq i\leq\widetilde{n}-1}$, with $\widetilde{n}\ll n$, one determines the optimal bandwidth parameters appearing in the preceding estimates. In practice, from only one trajectory of the underlying PDMP, one may divide the data into two categories: the largest one is used for the estimation, while the cross-validation step relies on the other one.
\begin{itemize}
\item Compute
$$\alpha^\mathcal{G} = \argmax_{\alpha>0}\int_{\mathcal{C}_x}\widehat{\mathcal{G}}_\alpha^n(\xi,\tau_x(\xi))^2\ud\xi -  \frac{2\,\Gamma\left(\frac{d-1}{2}+1\right)}{\widetilde{n}\pi^{\frac{d-1}{2}}\rho^{d-1}}\sum_{k=0}^{\widetilde{n}-1} \widehat{\mathcal{G}}_\alpha^{n}\!\!\left(\widetilde{Z}_k,\theta_x(\widetilde{Z}_k)\right)\mathbb{1}_{\mathbb{T}_{x,\rho}}(\widetilde{Z}_k)\,\mathbb{1}_{(\theta_x(\widetilde{Z}_k),+\infty)}(\widetilde{S}_{k+1}).$$

\item Choose $\rho_2>0$ and compute
\begin{align*}
(\alpha^\mathcal{F},\beta^\mathcal{F})=\argmax_{\alpha,\beta>0}&\phantom{-}\!\!\int_{\mathcal{C}_x}\widehat{\mathcal{F}}_{\alpha,\beta}^n(\xi,\tau_x(\xi))^2\ud\xi \nonumber
\\&\!\!\!\!\!\!-\,\,\frac{2\,\Gamma\left(\frac{d-1}{2}+1\right)}{\widetilde{n}\,\rho_2\,\pi^{\frac{d-1}{2}}\rho^{d-1}}\sum_{k=0}^{\widetilde{n}-1} \widehat{\mathcal{F}}_{\alpha,\beta}^{n}\!\!\left(\widetilde{Z}_k,\theta_x(\widetilde{Z}_k)\right)\mathbb{1}_{\mathbb{T}_{x,\rho}}(\widetilde{Z}_k)\mathbb{1}_{\left(\theta_x(\widetilde{Z}_k)-\frac{\rho_2}{2},\theta_x(\widetilde{Z}_k)+\frac{\rho_2}{2}\right)}(\widetilde{S}_{k+1}).
\end{align*}
\end{itemize}

\subsubsection{Estimation of $\lambda(x)$}

\noindent
Finally, we compute the best estimate of $\lambda(x)$ from all the preceding computations.

\begin{itemize}
\item Compute
$$\widehat{\xi}^n_* = \argmax_{\xi\in\mathcal{C}_x} \widehat{\mathcal{G}}^n_{\alpha^\mathcal{G}}(\xi,\tau_x(\xi)).$$
\item From $(\alpha^\mathcal{F},\beta^\mathcal{F})$, $\alpha^\mathcal{G}$ and $\widehat{\xi}^n_*$, compute
$$\widehat{\widehat{\lambda}}^n_\ast(x) = \frac{\widehat{\mathcal{F}}^n_{\alpha^\mathcal{F},\beta^\mathcal{F}} \left( \widehat{\xi}^n_* , \tau_x(\widehat{\xi}^n_*)\right)}{\widehat{\mathcal{G}}^n_{\alpha^\mathcal{G}} \left( \widehat{\xi}^n_* , \tau_x(\widehat{\xi}^n_*)\right)}.$$
\end{itemize}

\begin{rem}
The time complexity of this algorithm depends on several parameters, namely the number $n$ of observed jumps, the number $\widetilde{n}$ of observed data for the cross-validation steps, and also the numbers $N_\xi$, $N_\alpha$ and $N_\beta$ for discretizing the state spaces of $\xi$, $\alpha$ and $\beta$ at each maximization procedure. It is easy to see that the cross-validation step is the most complex but remains polynomial, precisely in $O(n N_\alpha N_\beta (N_\xi+\widetilde{n}))$.
\end{rem}

\subsection{TCP-like process}
\label{ss:tcp}

The application on which we focus in this part is a variant of the famous TCP window size process appearing in the modeling of the Transmission Control Protocol used for data transmission over the Internet and presented in \cite{Ma}. This protocol has been designed to adapt to the traffic conditions of the network: for a connection, the maximum number of packets that can be sent is given by a random variable called the congestion window size. At each time step, if all the packets are successfully transmitted, then one tries to transmit one more packet until a congestion appears.

\smallskip

The model presented in this part is two dimensional. For the sake of clarity, we will use the following notation: for any $x\in\mathbb{R}^2$, $x_1$ and $x_2$ denote the components of $x$. We consider a PDMP $(X_t)_{t\geq0}$ evolving on the state space $E=(0,1)^2$. The deterministic part of the model is defined from the flow $\Phi$ given by,
$$\forall\,x\in\mathbb{R}^2,~\forall\,t\in\mathbb{R},~\Phi(x,t) = (x_1+t , x_2).$$
The jump rate $\lambda$ is defined by,
$$\forall\,x\in\mathbb{R}^2,~\lambda(x)=x_1+x_2.$$
The transition kernel $\mathcal{Q}$ is defined by,
$$\forall\,x\in E,~\forall\,A\in\mathcal{B}(E),~\mathcal{Q}(x,A) \propto \int_A u(1-u)^{2/x_1-1}\,v(1-v)\,\ud u\,\ud v .$$
Starting from $(x_1,x_2)$, the process evolves in the unit square, always to the right, until a jump appears either when the motion hits the boundary $\{(\xi_1,\xi_2)\,:\,\xi_1=1\}$ or with the non homogeneous rate $x_1+x_2+t$ before, that is according to a Weibull distribution. The two components of the post-jump location are independent and both governed by a Beta-distribution, in such a way that the process tends with a high probability to jump to the left of the location just before the jump. One obtains a TCP-like process (see Figure \ref{fig:simulations}) for which the second dimension models the quality of the network (upper the second component is, higher the probability of a congestion is).

\begin{center}[ Figure \ref{fig:simulations} approximately here ]\end{center}

The asymptotic behavior of the process may be represented by the invariant distribution $\nu_\infty$ of the post-jump locations. Since this quantity is unknown, we propose to show in Figure \ref{fig:nuhat} its estimate $\widehat{\nu}^n_\infty$ defined by $(\ref{eq:def:nuhatn})$ and computed from $n=20\,000$ observed jumps.

\begin{center}[ Figure \ref{fig:nuhat} approximately here ]\end{center}

We present here all the procedure for estimating the jump rate $\lambda$ at the location $x=(0.75,0.5)$ for which the quality of the network is average. In this context, the class of estimators of $\lambda(x)$ is indexed by the elements $\xi\in\mathcal C_x=(0,0.75]\times\{0.5\}$. It should be noted that the invariant distribution at $x$ is quite low (see Figure \ref{fig:nuhat}). Nevertheless our method is expected to work pretty well even in this unfavorable framework.

\smallskip

In this simulation study, we assume that we observe the embedded Markov chain $(Z_n,S_{n+1})_{n\geq0}$ until the $10\,000^{\text{th}}$ jump. The cross-validation procedure is computed from an additional chain, independent on the first one, and observed until the $1\,000^\text{th}$ jump. When boxplots are presented, they have been computed over $100$ replicates.

\smallskip

We begin with the choice of the bandwidth parameter $\alpha$ appearing in $\widehat{\kappa}^n_x(\xi)=\widehat{\mathcal G}^n(\xi, \tau_x(\xi))$. The cross-validation procedure relies on the minimization of the estimate $\widehat{\varepsilon}^{\,n,\widetilde{n},\rho}_\kappa(\alpha)$ which depends on the positive parameter $\rho$. We present in Figure \ref{fig:alphaChoice} this quantity as a function of $\alpha$ and from different values of $\rho$. Fortunately, this new parameter seems to have little influence over the behavior of the estimation of the ISE along $\mathcal{C}_x$.

\begin{center}[ Figure \ref{fig:alphaChoice} approximately here ]\end{center}

Now, from this $\alpha$ (denoted in the sequel $\alpha^{\mathcal G}$), we maximize the estimated criterion $\widehat{\kappa}^n_x(\xi)$  along the curve $\mathcal{C}_x$ (see Figure \ref{fig:optimalpoint}): we obtain the optimal point $\widehat{\xi}^n_\ast$ at which we will compute our estimator of $\lambda(x)$. The crucial role of $\alpha$ at this maximization step is illustrated in Figure \ref{fig:badalpha}.

\begin{center}[ Figure \ref{fig:badalpha} approximately here ]\end{center}

We continue with the choice of the couple $(\alpha,\beta)$  implicitly appearing in the estimator $\widehat{\mathcal F}^n(\xi, \tau_x(\xi))$. The optimal parameters (denoted in the sequel $(\alpha^{\mathcal F},\beta^{\mathcal F})$) are obtained by minimizing the estimate $\widehat{\varepsilon}^{\,n,\widetilde{n},\rho_1,\rho_2}_{\mathcal F}(\alpha,\beta) $ of the related ISE (see Figure \ref{fig:alphabetaChoice}). 

\begin{center}[ Figure \ref{fig:alphabetaChoice} approximately here ]\end{center}

We compute the estimators $\widehat{\lambda}^n_{\xi}(x)$ for different values of $\xi\in\mathcal{C}_x$ and with the optimal bandwidths $\alpha^{\mathcal G}$ and $(\alpha^{\mathcal F},\beta^{\mathcal F})$. The related boxplots are presented in Figure \ref{fig:boxplotresult}. The procedure makes us able to choose the best index $\widehat{\xi}^n_\ast$ which most of the time corresponds with the estimate with least bias and variance. This proves the strong interest of the estimation algorithm developed in this paper.

\begin{center}[ Figure \ref{fig:boxplotresult} approximately here ]\end{center}

\subsection{Bacterial motility}
\label{ss:bac}

We present here a model of bacterial motility. Most motile bacteria move by the use of a flagellum or several flagella. The bacteria moves in the direction of the flagellum. A flagellum behaves like a rotary motor. Periodically, the flagellum changes its direction and results in reorientation of the bacteria. This allows bacteria to change direction. Bacteria can sense nutrients and move towards them. Additionally, they can move in response to temperature, light, etc. If the bacteria is in a favorable environment, the frequency of changes in direction is low. This intelligent behavior is allowed by the fact that the jump rate of the direction depends on the environment. For example, \cite{norris2013exploring,othmer2013excitation,tindall2008overview} propose models for the trajectory of the bacteria {\it E. Coli}. In particular, the author of \cite{norris2013exploring} uses PDMP's to describe the movement of bacteria under the influence of an external attractive chemical signal. 

\smallskip

We present here a variant of the model presented in \cite{norris2013exploring}. The path of the bacteria is described by a PDMP $(X_t)_{t\geq 0}$ evolving in the three-dimensional space state $E=D\times [0,2\pi)$ with $D$ the unit disk of $\R^2$. For the sake of clarity, we will use the following notation: for any $x\in\mathbb{R}^3$, $x_1,x_2$ and $x_3$ denote the components of $x$. In our case, $(x_1,x_2)$ is the position of the bacteria in the unit disk, and $x_3$ is the direction of its flagellum. 

\smallskip

The bacteria moves in the direction of its flagellum. In other words, the flow is given by,
$$\forall\,x\in \mathbb R^3,~\forall\,t\in \mathbb R,~ \Phi(x,t) =(x_1+t\cos x_3, x_2+t\sin x_3, x_3). $$
The bacteria changes its direction according to a jump rate function $\lambda(x)=\lambda(x_1,x_2)$ which depends on the environment only through the position $(x_1,x_2)$ and not on the current angle. The jump rate describes the interaction between the bacteria and its environment.

\smallskip

When the bacteria changes its direction, the new direction is  chosen preferentially in the direction of a more favorable environment. The transition kernel $\mathcal Q(x,\cdot)$ models this change of direction. In our case, we suppose  the bacteria has no information {\it a priori} on the quality of the environment around itself. Thus, $\mathcal Q$ can be defined by,
$$\forall\,x\in E,~ \forall A\in \mathcal B([0,2\pi[), ~\mathcal Q(x,(x_1,x_2)\times A)=\frac{1}{2\pi}\int_A \ud u.$$
Starting from the position $(x_1,x_2)$ in $D$, the bacteria evolves in the direction $x_3\in[0,2\pi)$, until a jump appears either when the bacteria hits the boundary $\partial D$ of its environment or with the rate $\lambda\left(\Phi(x,t)\right)=\lambda(x_1+t\cos x_3,x_2+t\sin x_3)$ before. The post-jump direction is next randomly chosen in $[0,2\pi)$, and the bacteria continues its path in this new direction, and so on. A possible path of this model is presented in Figure \ref{fig:pathBacteria}.

\begin{center}[ Figure \ref{fig:pathBacteria} approximately here ]\end{center}

In this simulation study, the jump rate is taken constant equal to $\lambda=1$. In other words, there is no interaction between the bacteria and its environment. Of course, this assumption is not taken into account in the estimation. Our goal is to estimate the jump rate $\lambda$ in different points of $E$ in order to check if it depends on the position. Actually, this provides an estimate of the influence of a likely external attractive signal on the bacteria.

\smallskip

The estimation algorithm presented in this paper yields an estimate of the jump rate at the three-dimensional state $(x_1,x_2,\theta)$. Nevertheless, one should take into account that the jump rate do not depend on the angle $\theta$. We explain here the procedure to reduce the dimension and estimate the jump rate at the location $(x_1,x_2)\in D$.
\begin{itemize}
\item Let $\theta\in[0,2\pi)$ be a fixed angle. Estimate the best index $\widehat \xi^n_*(\theta)$ and the jump rate $\widehat{ \lambda}_{\widehat \xi^n_*(\theta)}^n(x_\theta)$ at the location $x_\theta=(x_1,x_2,\theta)$.



\item For any $\theta$, $\widehat{ \lambda}_{\widehat \xi^n_*(\theta)}^n(x_\theta)$ estimates $\lambda(x_1,x_2)$. We reproduce the preceding step for any $\theta\in [0,2\pi)$. We obtain both a connected trajectory around $(x_1,x_2)$
\begin{equation}\label{eq:fleur}
\frak{C}_{x_1,x_2}=\big\{\widehat \xi^n_*(\theta)~:~\theta\in[0,2\pi)\big\},
\end{equation}
and a class of optimal estimators 
$$\frak{L}_{x_1,x_2}=\big\{\widehat{ \lambda}_{\widehat \xi^n_*(\theta)}^n(x_\theta)~:~\theta \in [0,2\pi)\big\} =\big\{\widehat\lambda_\xi^n~:~\xi\in\frak{C}_{x_1,x_2}\big\}$$ indexed by elements of $\frak{C}_{x_1,x_2}$. 
\item As the jump rate does not depend on the direction $\theta$, we aggregate the estimators in $\frak{L}_{x_1,x_2}$ to obtain the following jump rate estimate
\begin{equation}
\label{eq:lambdathetaint}
\widehat \lambda(x_1,x_2)=\displaystyle\frac{1}{2\pi}\int_0^{2\pi}\widehat{\lambda}_{\widehat \xi_*^n(\theta)}^n(x_\theta)\,\ud\theta .
\end{equation}
\end{itemize}
We investigate the estimation of the jump rate at the following 9 target points,
 $$(x_1,x_2)\in \{ (0,0),(-0.5,0),(-0.5,0.5),(-0.5,-0.5),(0,0.5),(0,-0.5),(0.5,0),(0.5,0.5),(0.5,-0.5)\},$$
from a trajectory of $n=100\,000$ changes of direction. In the sequel, the aggregated estimator $(\ref{eq:lambdathetaint})$ is approximated from a discretization of the interval $[0,2\pi)$ with step $\pi/8$. For each of the 9 target points, we present in Figure \ref{fig:estimations_100000} the boxplot of the  estimates $\widehat{\lambda}_{\widehat \xi_*^n(\theta)}^n(x_\theta)$ for $\theta$ in the discretization grid.

\begin{center}[ Figure \ref{fig:estimations_100000} approximately here ]\end{center}

As shown in Figure \ref{fig:fleurs_100000}, the trajectories $\frak{C}_{x_1,x_2}$ defined in $(\ref{eq:fleur})$ are very similar and close to the boundary $\partial D$. We explain this fact as follows. For each target state $x=(x_1,x_2,\theta)$, the optimal point $\widehat \xi^n_*(\theta)$ maximizes $\widehat{\kappa}_x^n(\xi)= \widehat{\mathcal{G}}^n(\xi,\tau_x(\xi))$ along
$$\mathcal C_{x}=\{(x_1-t\cos\theta,x_2-t\sin\theta,\theta)~:~t\geq0\}\cap D,$$
where $\tau_x(\xi)$ satisfies $x_1-\tau_x(\xi)\cos\theta=\xi_1$ and $x_2-\tau_x(\xi)\sin\theta=\xi_2$. Recall that $\widehat{\mathcal{G}}^n(\xi,\tau_x(\xi))$ is an estimator of $\kappa_x(\xi)=\nu_\infty(\xi) G(\xi,\tau_x(\xi))$. On the first hand, $G(\xi,t)=\exp(-t)$ for $t$ small enough, then the only point that maximizes $G(\xi,\tau_x(\xi))$ is $\xi=x$ itself. On the other hand, there is an important number of data that are very close to the boundary $\partial D$ (due to the fact that the bacteria jumps when it hits the boundary $\partial D$). Then the maximization of $\widehat{\mathcal{G}}^n(\xi,\tau_x(\xi))$ is a compromise between $\xi=x$ and $\xi$ near the boundary.

\smallskip

Finally, the estimates $\widehat \lambda(x_1,x_2)$ at the 9 target points are close together, and close to the true value $\lambda=1$ with a small variance, which is a good indicator of a constant jump rate (see Figure \ref{fig:fleurs_100000}). 
 
\begin{center}[ Figure \ref{fig:fleurs_100000} approximately here ]\end{center}


Now, we investigate the effect of the size $n$ of the dataset on the quality of the estimation. We compare results of simulations from $n=100\,000$, $n=50\,000$ and $n=20\,000$ observed data. When $n=20\,000$, the trajectories $\frak{C}_{x_1,x_2}$ are not all the time close to the boundary. It is due to the fact that, in this case, some areas of the state space are not  well-covered by the observed post-jump locations. Then, if the target point is $x=(x_1,x_2,\theta)$, the optimal point $\widehat \xi^n_*(\theta)\in\mathcal C_{x}$ is obtained at the location $x$ itself (see Figure \ref{fig:fleursComparN20000}). In addition, the estimated curves $\frak{C}_{x_1,x_2}$ obtained from $n=50\,000$ data are very close to the ones computed from $n=100\,000$ observed jumps.

\begin{center}[ Figure \ref{fig:fleursComparN20000} approximately here ]\end{center}

Of course, the larger the number of data $n$ is, the better the estimations (see Figures \ref{fig:BacteriaComparN}  and  \ref{fig:BoxplotComparN}). In particular, it is difficult to conclude that the jump rate is constant from only $n=20\,000$ observed jumps because the variance of the estimates is too large. Nevertheless, a dataset of size $50\,000$ may be sufficient to conclude that the jump rate does not depend on the location.

\begin{center}[ Figures \ref{fig:BacteriaComparN}  and  \ref{fig:BoxplotComparN} approximately here ]\end{center}

\subsection{Fatigue crack propagation}
\label{ss:fcg}

Fatigue crack propagation is a stochastic phenomenon due to the inherent uncertainties originating from material properties and cyclic mechanical loads.
As a consequence, stochastic processes offer an appropriate framework for modeling crack propagation. Fatigue life may be divided into crack initiation and two crack growth periods, namely the linear or stable regime described by Paris' equation
$$\frac{\ud a}{\ud N} = C \left(\Delta K(a)\right)^m,$$
and the acceleration or unstable regime modeled by Forman's equation
$$\frac{\ud a}{\ud N} = \frac{C\left(\Delta K(a)\right)^m}{(1-R)K_c - \Delta K(a)},$$
where $a$ is the crack length, $N$ is the time measured in number of loading cycles, $R$ is the stress ratio, $K_c$ represents the maximal value of the stress intensify factor required to induce failure and $\Delta K(a)$ is the range of the stress intensity factor given by
$$\Delta K(a) =\Delta\sigma\cos\left(\frac{\pi a}{\omega}\right)^{-1/2}\sqrt{\pi a},$$
where $\omega$ is the size of the test specimen and $\Delta\sigma$ is the stress range. In addition, $m$ and $C$ are two unknown material parameters. In this context, regime-switching models have been proposed to analyze fatigue crack growth data \cite{fcg,Chiquet}. The authors of \cite{fcg} propose to estimate the transition time between Paris' and Forman's regimes assuming that the crack propagation follows the trajectory of some piecewise-deterministic Markov process. The estimation procedure is performed on Virkler's dataset \cite{Virkler} in which 62 identical centre-cracked aluminium alloy specimens ($\omega=152\,\text{mm}$) were tested
under constant amplitude loading $\Delta\sigma=48.28\,\text{MPa}$ at a stress ratio $R = 0.2$. The number of loading cycles for the crack tip to advance a predetermined increment $\Delta a$ was recorded from an initial crack length of $a_0=9\,\text{mm}$ to a final length of 49.8 mm. 68 crack growth histories were obtained from these tests (see Figure \ref{fig:virkler}).

\begin{center}[ Figure \ref{fig:virkler} approximately here ]\end{center}

In this paper, we assume that the crack growth propagation follows Paris' equation with random parameters $m$ and $C$ before switching to Forman's equation with another set of parameters $m$ and $C$ distributed according to some transition measure. The transition occurs at a random time $T_1$ given by its survival function
$$\prob(T_1>t~|~a_0=9\,\text{mm},~m,~C) = \exp\left(-\int_0^t \lambda\left(\Phi_{m,C}^P(a_0,s)\right)\ud s\right),$$
where $\Phi_{m,C}^P$ is the deterministic flow of Paris' equation with parameters $m$ and $C$. In the sequel, $\Phi_{m,C}^P$ will be computed from the Runge-Kutta method because there is no explicit solution to this differential equation. We propose to estimate the jump rate of this stochastic model from real crack growth data. It should be noted that the jump times are not directly observed in Virkler's dataset but have been estimated in \cite{fcg}. The jump rate is of significant importance in understanding the transition between the stable and unstable regimes of crack propagation, and has never been estimated in the framework of PDMP's. We emphasize that Virkler's dataset is composed of $68$ independent experiments and thus does not directly follow the theoretical framework developed in this manuscript. Nevertheless, our estimation procedure is expected to perform pretty well in this more favorable context (independent curves instead of only one Markov path).

\smallskip

Material parameters $m$ and $C$ are unobserved but have also been estimated from Virkler's experiments in \cite{fcg}. It is well-known in the literature that there is a strong linear relationship between these features. This characteristic has again been highlighted in \cite{fcg} and is shown in Figure \ref{fig:mlogC}. In order to reduce the dimension and thus simplify the model, we propose to parametrize Paris' equation with only one parameter, say $m$, while $\log(C)$ is obtained from this linear relationship.

\begin{center}[ Figure \ref{fig:mlogC} approximately here ]\end{center}

We focus on the estimation of $\lambda(a)$ for some crack length $a$. For each parameter $m$ of Paris' equation there exists a unique deterministic time $\tau_m(a)$ such that Paris' flow reaches $a$ at time $\tau_m(a)$ starting from $a_0=9\,\text{mm}$. Criterion $\kappa_a(m)$ and its maximum have been estimated for different values of $a$ (see Figure \ref{fig:critere}). In particular, one may observe that the larger the target length $a$ is, the larger the optimal parameter $m$.

\begin{center}[ Figure \ref{fig:critere} approximately here ]\end{center}

For each crack length $a$, we estimate the jump rate $\lambda(a)$ from the parameter $m$ maximizing the estimated criterion $\kappa_a(m)$. We obtain the estimated function displayed in Figure \ref{fig:fcglambda}. This curve is increasing as expected and describes the transition rate from the stable region of propagation towards the acceleration regime that leads to fracture. This makes us able to detect conditions of crack growth instability and could be used to predict the critical length in fatigue crack propagation with a given level of confidence.

\begin{center}[ Figure \ref{fig:fcglambda} approximately here ]\end{center}


\appendix

\section{Ergodicity and invariant measures}
\label{sec:app:1}

In this section we present some preliminary and technical results about the invariant distributions of the underlying Markov chains. We begin with some properties of the Markov chain $(Z_n)_{n\geq0}$ of the post-jump locations of the PDMP of interest. In particular, under Assumption \ref{hyp:ergonu}, one may state the following result.

\begin{prop}\label{prop:nu:propr}We have the following statements:
\begin{itemize}
\item $(Z_n)_{n\geq0}$ is $\nu_\infty$-irreducible, positive Harris-recurrent and aperiodic.
\item $\nu_\infty$ is the unique invariant distribution of $(Z_n)_{n\geq0}$.
\end{itemize}
\end{prop}
\begin{proof}The proof is similar to the demonstration of Proposition 4.2 of \cite{azaisIHP}.\hfill$\Box$\end{proof}

\medskip


\noindent
These properties make us able to apply the law of large numbers to the Markov chain $(Z_n)_{n\geq0}$ (see \cite[Theorem 17.1.7]{MandT}). Now we propose to focus on the sequence $(Z_n,S_{n+1})_{n\geq0}$ which also forms a Markov process whose transition kernel is given by $\mathcal{R}$ in $(\ref{eq:def:R})$. Let us recall that $\mu_n$ denotes the distribution on the state space $F$ defined by $(\ref{eq:defF})$ of the couple $(Z_n,S_{n+1})$, $n\geq0$. We define the measure $\mu_\infty$ by,
\begin{equation}\label{eq:munuinfty}\forall\,A\in\mathcal{B}(\mathbb{R}^d),~\forall\,t\geq0,~\mu_\infty(A\times (t,+\infty)) = \int_A \nu_\infty(\ud x)\,\mathcal{S}(x,(t,+\infty)),\end{equation}
where the conditional distribution $\mathcal{S}$ is given by $(\ref{eq:formulaS})$. In particular, $\mu_\infty$ is the unique invariant distribution of $(Z_n,S_{n+1})_{n\geq0}$.

\begin{prop}\label{ergo:ZS}
We have the following statements:
\begin{itemize}
\item For any initial distribution $\mu_0=\delta_{\{x,t\}}$, $(x,t)\in F$, we have
$$\lim_{n\to\infty}\|\mu_n - \mu_\infty\|_{TV} = 0.$$
\item $(Z_n,S_{n+1})_{n\geq0}$ is $\mu_\infty$-irreducible, positive Harris-recurrent and aperiodic.
\item $\mu_\infty$ is the unique invariant distribution of $(Z_n,S_{n+1})_{n\geq0}$.
\end{itemize}
\end{prop}
\begin{proof}The proof is similar to the demonstrations of Lemma 4.4 and Proposition 4.5 of \cite{azaisIHP}.\hfill$\Box$\end{proof}

\medskip

\noindent
We are able again to apply the law of large numbers to the Markov chain $(Z_n,S_{n+1})_{n\geq0}$ by virtue of this result. It should be noted that the measures $\mu_n$, $n\geq0$, and $\mu_\infty$ do not share with $\nu_n$, $n\geq0$, and $\nu_\infty$ the property of absolute continuity with respect to the Lebesgue measure presented in Remark \ref{rem:nu:density}.

\begin{rem}
For any $n$, $\mu_n$ and thus the limit $\mu_\infty$ admit a density only on the interior $\mathring{F}$ of the state space $F$ because of the expression $(\ref{eq:S:densityFcirc})$ of $\mathcal{S}$ and thanks to Remark \ref{rem:nu:density}.
\end{rem}

\medskip
\noindent
Finally, we would like to highlight that the link between the measures $\nu_\infty$ and $\mu_\infty$ may be expressed in another way than $(\ref{eq:munuinfty})$. Indeed, for any $x\in E$, we have
\begin{equation}\label{eq:munuinfty:other} \int_F \mathcal{Q}(\Phi(y,s),x)\mu_\infty(\ud y\times\ud s) = \nu_\infty(x),\end{equation}
by the expression $(\ref{eq:expr:Pint})$ of the transition kernel $\mathcal{P}$ of $(Z_n)_{n\geq0}$ and because $\nu_\infty\mathcal{P}=\nu_\infty$ by virtue of Proposition \ref{prop:nu:propr}. The formula $(\ref{eq:munuinfty:other})$ of $\nu_\infty$ will be useful in our investigations.

\section{Proof of Theorem \ref{theo:3d}}
\label{sec:appendix:theo3d}

In all this section we use the notation $\mathbb{F}_n$ for the $\sigma$-algebra $\sigma(Z_0,\,S_1,\,\dots\,,Z_{n},S_{n+1})$, $n\geq0$. In addition, the classical symbols $\sim$, $o$ and $O$ must be understood to hold here in the almost sure sense.

\subsection{Sketch of the proof}

The proof of Theorem \ref{theo:3d} relies on the following decomposition. For any integer $n\geq1$, we have
\begin{equation}\label{eq:formula:diffMn}
\left[
\begin{array}{c}
\widehat{\mathcal{F}}^n(x,t)\\
\widehat{\mathcal{G}}^n(x,t)\\
\widehat{\nu}^n_\infty(x)
\end{array}
\right]
\,-\,
\left[
\begin{array}{c}
\nu_\infty(x)f(x,t)\\
\nu_\infty(x)G(x,t)\\
\nu_\infty(x)
\end{array}
\right]
\,=\,\frac{\mathcal{M}_{n-1}}{n} + \mathcal{R}_n ,
\end{equation}
where the sequence $(\mathcal{M}_n)_{n\geq0}$ is a $(\mathbb{F}_n)_{n\geq0}$-vector martingale defined by $(\ref{eq:Mn:component})$ and studied in Appendix \ref{ss:martingaleterm}, and $\mathcal{R}_n$ is a remainder term defined by $(\ref{eq:def:Rni})$ and studied in Appendix \ref{ss:remainderterms}.

\smallskip

In Appendix \ref{sss:Rn:as}, we establish that the remainder term almost surely goes to $0$ when $n$ tends to infinity. This is a first step to show the almost sure convergence presented in Theorem \ref{theo:3d}.
In addition we investigate the rate of convergence of the remainder term in Appendix \ref{sss:Rn:rate} under an additional Lipschitz mixing condition stated in Assumption \ref{hyp:lipmix}. This is enough to prove the asymptotic normality given in Theorem \ref{theo:3d}.

\smallskip

The rest of the proof deals with the study of the martingale term. We prove in Appendix \ref{sss:martingale?} that the process $(\mathcal{M}_n)_{n\geq0}$ is a vector martingale. We investigate its asymptotic behavior by studying its square variation process $(\langle\mathcal{M}\rangle_n)_{n\geq0}$ in Appendix \ref{ss:sqvar}. Thanks to these results we state the law of large numbers and the central limit theorem for $(\mathcal{M}_n)_{n\geq0}$ in Appendix \ref{ss:limitmg}.

\smallskip

Finally the almost sure convergence presented in Theorem \ref{theo:3d} is a direct application of $(\ref{eq:formula:diffMn})$ and $(\ref{eq:restgoesto0})$ together with $(\ref{eq:lln:Mn})$, while the asymptotic normality is obtained from $(\ref{eq:formula:diffMn})$, $(\ref{eq:rest:rateto0})$ and $(\ref{eq:tcl:Mn})$.

\subsection{Remainder term}
\label{ss:remainderterms}

This part of the paper is devoted to the asymptotic study of the remainder term sequence $(\mathcal{R}_n)_{n\geq1}$ appearing in $(\ref{eq:formula:diffMn})$.

\hiddensubsubsection{Definition of the remainder term}
\label{sss:rem:def}

For any $n$, one may write $\mathcal R_n =\left[R_n^{(k)}\right]_{1\leq k\leq3}$, where each of the components is defined by
\begin{equation} \label{eq:def:Rni}
R_n^{(1)} =\sum_{j=0}^4 R_n^{(1,j)},\qquad R_n^{(2)} =\sum_{j=0}^3 R_n^{(2,j)} \qquad\text{and}\qquad R_n^{(3)} =\sum_{j=0}^2 R_n^{(3,j)}.
\end{equation}
The five terms that define $R_n^{(1)}$ are given by
\begin{eqnarray*}
R_n^{(1, 0)}&=&\displaystyle\frac{1}{n v_0^dw_0}\mathbb{K}_d\left(\frac{Z_0-{x}}{v_0}\right)\mathbb{K}_1\left(\frac{S_1-t}{w_0}\right) ,\\
R_n^{(1, 1)}&=&\displaystyle \frac1n\sum_{j=1}^{n-1}\int_{\mathbb R^d\times \mathbb R}\!\!\!\!\! f(x+u v_j, t+vw_j)\mathbb{K}_d(u)\mathbb{K}_1(v) \left(\mathcal{Q}\left(\Phi\left(Z_{j-1}, S_j\right), x+u v_j\right)\!-\!\mathcal Q\left(\Phi\left(Z_{j-1}, S_j\right), x \right) \right)\ud u\,\ud v , \\
R_n^{(1,2)}&=&\displaystyle \frac1n\sum_{j=1}^{n-1}\mathcal Q\left(\Phi\left(Z_{j-1}, S_j\right), {x} \right) \int_{\mathbb R^d\times \mathbb R} \left[ f({x}+uv_j, t+vw_j)- f({x}, t+vw_j)\right]\mathbb{K}_d(u)\mathbb{K}_1(v)\,\ud u\,\ud v , \\
R_n^{(1,3)}&=&\displaystyle \frac1n\sum_{j=1}^{n-1}\mathcal Q\left(\Phi\left(Z_{j-1}, S_j\right), {x} \right) \int_{\mathbb R^d\times \mathbb R} \left[f({x}, t+vw_j)-f({x}, t)\right]\mathbb{K}_d(u)\mathbb{K}_1(v)\,\ud u\,\ud v , \\
R_n^{(1,4)}&=&\displaystyle\left[\frac1n\sum_{j=1}^{n-1}\mathcal Q\left(\Phi\left(Z_{j-1}, S_j\right), {x} \right)  - \nu_\infty(x)\right]f({x}, t).
\end{eqnarray*}
The four terms that define $R_n^{(2)}$ are given by
\begin{eqnarray*}
R^{(2,0)}_n&=&\frac1n\frac{1}{v_0^d}\mathbb{K}_d\left(\frac{Z_0-x}{v_0}\right)\mathbb{1}_{\{S_1>t\}} ,\\
R^{(2,1)}_n&=&\displaystyle \frac1n\sum_{j=1}^{n-1}\int_{\mathbb R^d} \left[\mathcal Q\left(\Phi\left(Z_{j-1}, S_j\right), x+u v_j\right) -\mathcal Q\left(\Phi\left(Z_{j-1}, S_j\right), {x} \right) \right]G({x}+u v_j, t)\mathbb{K}_d(u)\,\ud u ,\\
R^{(2,2)}_n&=&\displaystyle \frac1n\sum_{j=1}^{n-1}\mathcal Q\left(\Phi\left(Z_{j-1}, S_j\right), {x} \right) \int_{\mathbb R^d} \left[ G(x+uv_j, t)- G(x, t)\right]\mathbb{K}_d(u)\,\ud u , \\
R^{(2,3)}_n&=&\displaystyle\left[ \frac1n\sum_{j=1}^{n-1}\mathcal Q\left(\Phi\left(Z_{j-1}, S_j\right), {x} \right)  - \nu_\infty({x})\right]G(x, t) .
\end{eqnarray*}
Finally, the three terms that define $R_n^{(3)}$ are given by
\begin{eqnarray*}
R^{(3,0)}_n&=&\frac1n\frac{1}{v_0^d}\mathbb{K}_d\left(\frac{Z_0-{x}}{v_0}\right) ,\\
R^{(3,1)}_n&=&\displaystyle \frac1n\sum_{j=1}^{n-1}\int_{\mathbb R^d} \left[\mathcal Q\left(\Phi\left(Z_{j-1}, S_j\right), {x} +u v_j\right) -\mathcal Q\left(\Phi\left(Z_{j-1}, S_j\right), {x} \right) \right]\mathbb{K}_d(u)\,\ud u ,\\
R^{(3,2)}_n&=&\displaystyle \frac1n\sum_{j=1}^{n-1}\mathcal Q\left(\Phi\left(Z_{j-1}, S_j\right), {x} \right)  - \nu_\infty({x}) .
\end{eqnarray*}

\hiddensubsubsection{Almost sure convergence}\label{sss:Rn:as}

We only investigate the first component $(R_n^{(1)})_{n\geq1}$ since the other terms may be treated with similar arguments. First, it is obvious that the sequence $(R_n^{(1,0)})_{n\geq1}$ almost surely tends to $0$. In addition, the term $(R_n^{(1,4)})_{n\geq1}$ converges to $0$ by virtue of the ergodic theorem applied to the Markov chain $(Z_n,S_{n+1})_{n\geq0}$ (thanks to Proposition \ref{ergo:ZS}) together with $(\ref{eq:munuinfty:other})$. Because both the functions $\mathcal{Q}$ and $f$ are Lipschitz and bounded (see Assumptions \ref{hyps:first:regularity}), we have
\begin{eqnarray*}
\left| R_n^{(1,1)}\right|&\leq& \frac{[ \mathcal Q]_{Lip}\|f\|_\infty}{n}\sum_{j=1}^{n-1}v_j\int_{\mathbb R^d}|u| \mathbb{K}_d(u)\ud u ,\\
\left| R_n^{(1,2)}\right|&\leq& \frac{\| \mathcal Q\|_{\infty}[f]_{Lip}}{n}\sum_{j=1}^{n-1}v_j\int_{\mathbb R^d}|u|\mathbb{K}_d(u)\ud u,\\
\left| R_n^{(1,3)}\right|&\leq& \frac{\| \mathcal Q\|_{\infty}[f]_{Lip}}{n}\sum_{j=1}^{n-1}w_j\int_{\mathbb R}|v|\mathbb{K}_1(v) \ud v,
\end{eqnarray*}
which are some Ces\`aro means because both the sequences $(v_n)_{n\geq0}$ and $(w_n)_{n\geq0}$ tend to $0$. Thus,
\begin{equation}\label{eq:restgoesto0}
\mathcal{R}_n \stackrel{a.s.}{\longrightarrow} 0_3 .
\end{equation}

\hiddensubsubsection{Rate of convergence} \label{sss:Rn:rate}

Under the first item of Assumptions \ref{hyp:lipmix}, the Markov chain $(Z_n)_{n\geq0}$ and thus the two-dimensional process $(Z_n,S_{n+1})_{n\geq0}$ satisfy the contraction assumption given in Theorem 6.3.17 of \cite{MD}. By applying this theorem to the function $\varphi=\mathcal{Q}(\Phi(\cdot,\cdot),\cdot)\in\text{Li}(r_1,r_2)$ (see second item of Assumptions \ref{hyp:lipmix}), we obtain that, for any $\gamma\in(0,1)$ and $x\in E$,
$$n^{-\frac{1+\gamma}{2}}\sum_{j=1}^{n-1}\left[\mathcal Q\left(\Phi\left(Z_{j-1}, S_j\right), x \right)  - \nu_\infty(x)\right]\stackrel{a.s}{\longrightarrow}0 ,$$
because $2(r_1+r_2)\leq a_1$. Therefore, we have for any couple $(k,j)$ in the set $\{(1,4), (2,3), (3,2)\}$,
\begin{equation*}
n^{\frac{1-\alpha d -\beta}{2}}R_n^{(k,j)} \stackrel{a.s.}{\longrightarrow}0.
\end{equation*}
The same result is obvious for the other couples $(k,j)$ with the condition $\alpha d +\beta +2\min(\alpha, \beta)>1$. Finally we obtain, when $n$ goes to infinity,
\begin{equation}\label{eq:rest:rateto0}
n^{\frac{1-\alpha d -\beta}{2}} \mathcal{R}_n \stackrel{a.s.}{\longrightarrow} 0_3 .
\end{equation}

\subsection{Martingale term}
\label{ss:martingaleterm}

\hiddensubsubsection{$(\mathcal{M}_n)_{n\geq0}$ is a vector martingale}
\label{sss:martingale?}

Let $(\mathcal{M}_n)_{n\geq0}$ be the three-dimensional process defined for any $n$ by $\mathcal{M}_n=\left[M_n^{(k)}\right]_{1\leq k\leq3}$, where each component is defined as
\begin{equation}\label{eq:Mn:component}
M_n^{(k)}= \sum_{j=1}^n A_j^{(k)} - B_j^{(k)},~k\in\{1,2,3\},
\end{equation}
and the terms $A_j^{(k)}$ and $B_j^{(k)}$ are given by
\begin{eqnarray*}
A_j^{(1)} &=& \frac{1}{v_j^d}\mathbb{K}_d\left(\frac{Z_j-x}{v_j}\right) \frac{1}{w_j}\mathbb{K}_1\left(\frac{S_{j+1}-t}{w_j}\right) ,\\
B_j^{(1)} &=& \int_{\mathbb{R}^d\times\mathbb{R}} \mathcal{Q}(\Phi(Z_{j-1},S_j) , x+u v_j) f(x+uv_j , t+v w_j)\,\mathbb{K}_d(u) \mathbb{K}_1(v)\,\lambda_d(\ud u)\,\ud v ,\\
A_j^{(2)} &=& \frac{1}{v_j^d}\mathbb{K}_d\left(\frac{Z_j-x}{v_j}\right) \mathbb{1}_{\{S_{j+1}>t\}} ,\\
B_j^{(2)} &=& \int_{\mathbb{R}^d} \mathcal{Q}(\Phi(Z_{j-1},S_j) , x+u v_j) G(x+uv_j , t)\,\mathbb{K}_d(u)\,\lambda_d(\ud u) ,\\
A_j^{(3)} &=& \frac{1}{v_j^d}\mathbb{K}_d\left(\frac{Z_j-x}{v_j}\right),\\
B_j^{(3)} &=& \int_{\mathbb{R}^d} \mathcal{Q}(\Phi(Z_{j-1},S_j) , x+u v_j)\,\mathbb{K}_d(u)\,\lambda_d(\ud u).
\end{eqnarray*}
We claim that the process $(\mathcal{M}_n)_{n\geq0}$ is a $(\mathbb{F}_n)_{n\geq0}$-martingale. The keystone is to show that, for any $n$ and $k$, we have
$$\mathbb{E}\left[A_{n}^{(k)} ~\Big|~ \mathbb{F}_{n-1}\right] = B_{n-1}^{(k)} .$$
The proof presents no particular difficulty except for the first component $k=1$ for which we provide some details. Let us recall that $\mathcal{R}$ denotes the transition kernel of the Markov chain $(Z_n,S_{n+1})_{n\geq0}$ (see Subsection \ref{ss:def:not}). We have
\begin{eqnarray}
\mathbb{E}\left[A_{n}^{(1)}~\Big|~\mathbb{F}_{n-1}\right]	&=& \frac{1}{v_n^d w_n}\int_{\mathbb{R}^d\times\mathbb{R}_+}\mathbb{K}_d\left(\frac{\xi-x}{v_n}\right)\mathbb{K}_1\left(\frac{s-t}{w_n}\right) \mathcal{R}\left( (Z_{n-1},S_n) , \ud\xi\times\ud s\right) \nonumber \\
&=&\frac{1}{v_n^d w_n}\int_{\mathbb{R}^d\times\mathbb{R}_+}\mathbb{K}_d\left(\frac{\xi-x}{v_n}\right)\mathbb{K}_1\left(\frac{s-t}{w_n}\right) \mathcal{Q}\left(\Phi(Z_{n-1},S_n) ,\ud\xi\right)\mathcal{S}(\xi,\ud s), \label{eq:proof:mg:int}
\end{eqnarray}
with $(\ref{eq:def:R})$. Let us recall that the bandwidth sequence $(v_n)_{n\geq0}$ is decreasing. Together with the third item of Assumptions \ref{hyps:kernel}, we obtain
$$\text{supp}\,\mathbb{K}_d\left(\frac{\cdot - x}{v_n}\right)~\subset~\text{supp}\,\mathbb{K}_d\left(\frac{\cdot-x}{v_0}\right)~\subset~B_d( x , v_0\delta) .$$
With similar arguments, we obtain
\begin{equation}\label{eq:inclusions:K}\text{supp}\,\mathbb{K}_1\left(\frac{\cdot - t}{w_n}\right)~\subset~\text{supp}\,\mathbb{K}_1\left(\frac{\cdot-t}{w_0}\right)~\subset~B_1( t , w_0\delta)~\subset~\left(-\infty,\inf_{\xi\in B_d(x,v_0\delta)}t^+(\xi)\right),
\end{equation}
with the condition $(\ref{eq:condition:v0w0})$ on the couple $(v_0,w_0)$. By $(\ref{eq:proof:mg:int})$, $(\ref{eq:inclusions:K})$ together with the expression $(\ref{eq:S:densityFcirc})$ of $\mathcal{S}$ we obtain
\begin{equation*}
\mathbb{E}\left[A_{n}^{(1)}~\Big|~\mathbb{F}_{n-1}\right] = \frac{1}{v_n^d w_n}\int_{\mathbb{R}^d\times\mathbb{R}_+}\mathbb{K}_d\left(\frac{\xi-x}{v_n}\right)\mathbb{K}_1\left(\frac{s-t}{w_n}\right) \mathcal{Q}\left(\Phi(Z_{n-1},S_n) ,\ud\xi\right)f(\xi,s)\,\ud s .
\end{equation*}
One may conclude by a change of variable and $(\ref{eq:Qdensity})$.

\hiddensubsubsection{Predictable square variation process}
\label{ss:sqvar}

The asymptotic behavior of the martingale $(\mathcal{M}_n)_{n\geq0}$ may be investigated by studying its predictable square variation process that we denote as usual $(\langle\mathcal{M}\rangle_n)_{n\geq0}$. At any time $n$, $\langle\mathcal{M}\rangle_n$ is the $3\times3$ symmetric matrix defined by
\begin{equation*}
\langle\mathcal{M}\rangle_n = \sum_{k=1}^n
\left[ \mathbb{E}\left[ M_k^{(i)}M_k^{(j)} \,\Big|\,\mathbb{F}_{k-1}\right] - M_{k-1}^{(i)}M_{k-1}^{(j)} \right]_{1\leq i,j\leq3}.
\end{equation*}

\noindent
Now, we calculate each coefficient of this matrix by beginning with the diagonal terms and the first of these $\langle\mathcal{M}\rangle_n^{(1,1)}$. One has
\begin{equation*}
\langle\mathcal{M}\rangle_n^{(1,1)} = n \, T_n^{(1)} - T_n^{(2)} ,
\end{equation*}
where the terms $T_n^{(1)}$ and $T_n^{(2)}$ are defined by
\begin{eqnarray*}
T_n^{(1)} &=& \frac{1}{n}\sum_{j=1}^n \frac{1}{v_j^d w_j}\int_{\mathbb{R}^d\times\mathbb{R}} \mathcal{Q}(\Phi(Z_{j-1},S_j) , {x}+u v_j) f({x}+uv_j , t+v w_j) \mathbb{K}_d^2(u) \mathbb{K}_1^2(v) \ud u\,\ud v ,\\
T_n^{(2)} &=& \sum_{j=1}^n \left(\int_{\mathbb{R}^d\times\mathbb{R}} \mathcal{Q}(\Phi(Z_{j-1},S_j) , {x}+u v_j) f({x}+uv_j , t+v w_j) \mathbb{K}_d(u) \mathbb{K}_1(v) \ud u\,\ud v\right)^2.
\end{eqnarray*}
Since the functions $\mathcal{Q}$ and $f$ are bounded (see Assumptions \ref{hyps:first:regularity}), we easily obtain that
\begin{equation}\label{eq:Tn2}T_n^{(2)} = O(n) .\end{equation}
Let us introduce the additional notation,
\begin{eqnarray*}
\widetilde{T}_n^{(1)} &=& \frac{1}{n}\sum_{j=1}^n\int_{\mathbb{R}^d\times\mathbb{R}} \mathcal{Q}(\Phi(Z_{j-1},S_j) , {x}+u v_j) f({x}+uv_j , t+v w_j) \mathbb{K}_d^2(u) \mathbb{K}_1^2(v) \ud u\,\ud v ,\\
\check{T}_n^{(1)} &=& \frac{1}{n}\sum_{j=1}^n\mathcal{Q}(\Phi(Z_{j-1},S_j) , {x}) f({x} , t) \int_{\mathbb{R}^d\times\mathbb{R}}\mathbb{K}_d^2(u) \mathbb{K}_1^2(v) \ud u\,\ud v \\
&=& \frac{\tau_d^2\,\tau_1^2}{n}\sum_{j=1}^n\mathcal{Q}(\Phi(Z_{j-1},S_j) , {x}) f({x} , t) ,
\end{eqnarray*}
by Remark \ref{rem:hyps:kernel:tau} on the kernel functions $\mathbb{K}_1$ and $\mathbb{K}_d$. Since both $\mathcal{Q}$ and $f$ are Lipschitz and bounded (see Assumptions \ref{hyps:first:regularity}), together with $\text{supp}\,\mathbb{K}_p \subset B_p(0_p,\delta)$ (see Assumptions \ref{hyps:kernel}), the stochastic sequences $(\widetilde{T}_n^{(1)})_{n\geq0}$ and $(\check{T}_n^{(1)})_{n\geq0}$ have the same limit thanks to
\begin{equation}\label{eq:TtildeTcheck}
\left|\widetilde{T}_n^{(1)} - \check{T}_n^{(1)}\right|\leq \frac{\left(\|\mathcal{Q}\|_{\infty}[f]_{Lip} + [\mathcal{Q}]_{Lip}\|f\|_{\infty}\right) \delta \tau_d^2 \tau_1^2}{n}\sum_{j=1}^n (v_j+w_j) ,
\end{equation}
which is a Ces\`aro mean. In addition, by the almost sure ergodic theorem (see Proposition \ref{ergo:ZS}) together with $(\ref{eq:munuinfty:other})$, we obtain
\begin{equation}\label{eq:ergodic:Tcheck}
\check{T}_n^{(1)} \stackrel{a.s.}{\longrightarrow} \tau_d^2 \tau_1^2 \nu_\infty(x) f(x,t) . 
\end{equation}
By $(\ref{eq:TtildeTcheck})$ and $(\ref{eq:ergodic:Tcheck})$ and by virtue of Lemma 7.1.5 of \cite{MD}, we obtain
$$\frac{T_n^{(1)}}{n^{{\alpha d+\beta}}} \stackrel{a.s.}{\longrightarrow}   \frac{\tau_d^2 \tau_1^2 \nu_\infty(x) f(x,t)}{1+\alpha d+\beta} .$$
Together with $(\ref{eq:Tn2})$, as $n$ goes to infinity,
\begin{equation} \label{eq:asymptotic:crochet1}
\frac{\langle\mathcal{M}\rangle_n^{(1,1)}}{n^{{1+\alpha d+\beta}}} \stackrel{a.s.}{\longrightarrow} \frac{\tau_d^2 \tau_1^2 \nu_\infty(x) f(x,t)}{1+\alpha d+\beta} \,=\,\triangle_1.
\end{equation}

\noindent
The second diagonal term of the predictable variation process may be studied in a similar way as before. We obtain
\begin{eqnarray*}
\langle\mathcal{M}\rangle_n^{(2,2)} &=& \sum_{j=1}^n \Bigg[\frac{1}{v_j^d} \int_{\mathbb{R}^d} \mathcal{Q}(\Phi(Z_{j-1},S_j) , {x}+u v_j) G({x}+uv_j , t) \mathbb{K}_d^2(u) du \\
&~&\qquad- \left(\int_{\mathbb{R}^d}\mathcal{Q}(\Phi(Z_{j-1},S_j) , {x}+u v_j) G({x}+uv_j , t) \mathbb{K}_d(u)  du\right)^2\Bigg] ,
\end{eqnarray*}
and, as $n$ goes to infinity,
\begin{equation} \label{eq:asymptotic:crochet2}
\frac{\langle\mathcal{M}\rangle_n^{(2,2)}}{n^{1+\alpha d}} \stackrel{a.s.}{\longrightarrow} \frac{\tau_d^2 \nu_\infty(x) G(x,t)}{1+\alpha d} \,=\,\triangle_2.
\end{equation}

\noindent
The third and last diagonal term may also be investigated in the same way. We have
\begin{eqnarray*}
\langle\mathcal{M}\rangle_n^{(3,3)} &=& \sum_{j=1}^n \Bigg[\frac{1}{v_j^d} \int_{\mathbb{R}^d} \mathcal{Q}(\Phi(Z_{j-1},S_j) , {x}+u v_j) \mathbb{K}_d^2(u) du \\
&~&\qquad- \left(\int_{\mathbb{R}^d} \mathcal{Q}(\Phi(Z_{j-1},S_j) , {x}+u v_j) \mathbb{K}_d(u)  du\right)^2\Bigg],
\end{eqnarray*}
and, when $n$ goes to infinity,
\begin{equation} \label{eq:asymptotic:crochet3}
\frac{\langle\mathcal{M}\rangle_n^{(3,3)}}{n^{1+\alpha d}} \stackrel{a.s.}{\longrightarrow} \frac{\tau_d^2 \nu_\infty(x)}{1+\alpha d} \,=\,\triangle_3.
\end{equation}

\noindent
Now we focus on the non diagonal terms. For any integer $n$, we have
\begin{eqnarray}
\langle\mathcal{M}\rangle_n^{(1,2)} &=& \sum_{j=1}^n \Bigg[         \int_{\mathbb{R}^d\times\mathbb{R}}\mathcal{Q}(\Phi(Z_{j-1},S_j) , {x}+u v_j)^2 f({x}+uv_j , t+v w_j) G({x}+u v_j,t) \mathbb{K}_d^2(u) \mathbb{K}_1(v) \ud u\,\ud v   \nonumber      \\
&~&\qquad- \left( \int_{\mathbb{R}^d\times\mathbb{R}} \mathcal{Q}(\Phi(Z_{j-1},S_j) , {x}+u v_j) f({x}+uv_j , t+v w_j) \mathbb{K}_d(u) \mathbb{K}_1(v)\ud u\,\ud v\right) \nonumber \\
&~&\qquad\times \left(  \int_{\mathbb{R}^d} \mathcal{Q}(\Phi(Z_{j-1},S_j) , {x}+u v_j) G({x}+uv_j , t) \mathbb{K}_d(u) \ud u  \right)\Bigg] , \label{eq:maj:crochet12} \\
\langle\mathcal{M}\rangle_n^{(1,3)} &=& \sum_{j=1}^n \Bigg[         \int_{\mathbb{R}^d\times\mathbb{R}} \mathcal{Q}(\Phi(Z_{j-1},S_j) , {x}+u v_j)^2 f({x}+uv_j , t+v w_j) \mathbb{K}_d^2(u) \mathbb{K}_1(v)\ud u\,\ud v    \nonumber     \\
&~&\qquad- \left( \int_{\mathbb{R}^d\times\mathbb{R}} \mathcal{Q}(\Phi(Z_{j-1},S_j) , {x}+u v_j) f({x}+uv_j , t+v w_j) \mathbb{K}_d(u) \mathbb{K}_1(v) \ud u\,\ud v\right) \nonumber \\
&~& \qquad\times \left(   \int_{\mathbb{R}^d} \mathcal{Q}(\Phi(Z_{j-1},S_j) , {x}+u v_j) \mathbb{K}_d(u)  \ud u  \right)\Bigg] , \label{eq:maj:crochet13}\\
\langle\mathcal{M}\rangle_n^{(2,3)} &=& \sum_{j=1}^n \Bigg[         \int_{\mathbb{R}^d} \mathcal{Q}(\Phi(Z_{j-1},S_j) , {x}+u v_j)^2 G({x}+uv_j,t) \mathbb{K}_d^2(u)\ud u  \nonumber      \\
&~&\qquad- \left( \int_{\mathbb{R}^d} \mathcal{Q}(\Phi(Z_{j-1},S_j) , {x}+u v_j) G({x}+uv_j , t) \mathbb{K}_d(u) \ud u\right)  \nonumber \\
&~& \qquad \times \left(   \int_{\mathbb{R}^d} \mathcal{Q}(\Phi(Z_{j-1},S_j) , {x}+u v_j) \mathbb{K}_d(u)\ud u  \right)\Bigg] . \label{eq:maj:crochet23}
\end{eqnarray}
From $(\ref{eq:maj:crochet12})$, $(\ref{eq:maj:crochet13})$ and $(\ref{eq:maj:crochet23})$, and because $\mathcal{Q}$ and $f$ are bounded (see Assumptions \ref{hyps:first:regularity}) together with the fact that the integral $\int\mathbb{K}_p^2\ud\lambda_p$ is finite (see Remark \ref{rem:hyps:kernel:tau}), we easily obtain that, for any $i\neq j$,
\begin{equation}\label{eq:asymptotic:nondiago}
\langle\mathcal{M}\rangle_n^{(i,j)} = O(n).
\end{equation}

\noindent
As a conclusion, by $(\ref{eq:asymptotic:crochet1})$, $(\ref{eq:asymptotic:crochet2})$, $(\ref{eq:asymptotic:crochet3})$ and $(\ref{eq:asymptotic:nondiago})$, one may sum up the asymptotic behavior of the predictable variation process $(\langle\mathcal{M}\rangle_n)_{n\geq0}$ by the following formula,
\begin{equation} \label{eq:crochetM:limite}
\frac{\langle\mathcal{M}\rangle_n}{n^{1+\alpha d}} \sim \left[
\begin{array}{ccc}
n^{{\beta}} \triangle_1 & 0 & 0 \\
0 &\triangle_2 & 0\\
0 & 0 & \triangle_3
\end{array}\right].
\end{equation}
It should be noted that the coefficients $\triangle_i$, $1\leq i\leq3$, are positive because we assume that $\nu_\infty(x)>0$ and $f(x,t)>0$ in the statement of Theorem \ref{theo:3d}.

\hiddensubsubsection{Limit theorems for the vector martingale}
\label{ss:limitmg}

\paragraph{Law of large numbers.} We propose to apply the law of large numbers for vector martingales (see \cite[Theorem 4.3.15]{MD}) to the process of interest $(\mathcal{M}_n)_{n\geq0}$. In the sequel, for any $n$, $\mathcal{T}_n$ denotes the trace of the matrix $\langle\mathcal{M}\rangle_n$, while $\mathcal{E}_{n}$ stands for its minimum eigenvalue. First, in light of $(\ref{eq:crochetM:limite})$, the trace $(\mathcal{T}_n)_{n\geq0}$ almost surely tends to infinity. Thus we are able to apply the third item of Theorem 4.3.15 of \cite{MD} with any function $a(t)=t^{1+\eta}$, $\eta>0$, and we obtain,
$$\left\|\langle\mathcal{M}\rangle_n^{-1/2}\mathcal{M}_n\right\|^2 = o\left(\frac{\log(\mathcal{T}_n)^{1+\eta}}{\mathcal{E}_n}\right).$$
By $(\ref{eq:crochetM:limite})$ again, we have, when $n$ goes to infinity, 
$$\mathcal{T}_n = O\left(n^{1+\alpha d+\beta}\right) \qquad\text{and}\qquad\mathcal{E}_n\sim \min(\triangle_2,\triangle_3) \, n^{1+\alpha d} .$$
As a consequence, using that $n^{1+\alpha d+\beta}=o(n^2)$ whenever $\alpha d+\beta<1$ together with $(\ref{eq:crochetM:limite})$, we have the law of large numbers, as $n$ tends to infinity,
\begin{equation}\label{eq:lln:Mn}
\frac{\|\mathcal{M}_n\|}{n} \stackrel{a.s.}{\longrightarrow} 0 .
\end{equation}

\paragraph{Central limit theorem.} We now investigate the asymptotic normality of the vector martingale $(\mathcal{M}_n)_{n\geq0}$. We apply Corollary 2.1.10 of \cite{MD} with the sequence $(a_n)_{n\geq0}$ defined by $a_n=n^{1+\alpha d+\beta}$. The first assumption of this result is obviously satisfied: the sequence $(a_n^{-1}\langle\mathcal{M}\rangle_n)_{n\geq0}$ almost surely converges to some positive semi-definite matrix. Indeed, by $(\ref{eq:crochetM:limite})$,
\begin{equation}\label{eq:def:sigma}
n^{-1-\alpha d-\beta}\langle\mathcal{M}\rangle_n \stackrel{a.s.}{\longrightarrow}  \left[
\begin{array}{ccc}
 \frac{\tau_d^2 \tau_1^2 \nu_\infty(x) f({x},t)}{1+\alpha d+\beta} & 0 & 0 \\
0 &0 & 0\\
0 & 0 & 0
\end{array}\right] \,=\,\Sigma({x},t,\alpha,\beta),
\end{equation}
where $\Sigma(x,t,\alpha,\beta)$ is a degenerate variance-covariance matrix. 
As a consequence, we only have to check Lindeberg's condition in order to establish the central limit theorem for $(\mathcal{M}_n)_{n\geq0}$.
%
%
In other words, we have to prove that, for any $\epsilon>0$,
\begin{equation}\label{Cond_b}
\frac{1}{ n^{1+\alpha d+\beta}}\sum_{j=1}^n \mathbb E\left[\left|M_j- M_{j-1}\right|^2\,\mathbb{1}_{\left\{|M_j- M_{j-1}|\geq \epsilon n^{\frac{1+\alpha d+\beta}{2}}\right\}}\,\Bigg|\,\mathbb{F}_{j-1}\right] \stackrel{\mathbb P}{\longrightarrow}0.
\end{equation}
For any $1\leq j\leq n$ and $1\leq k\leq3$, we study the three components $M_j^{(k)}- M_{j-1}^{(k)}$. We have
\begin{eqnarray*}
M_j^{(1)}- M_{j-1}^{(1)}&=&\frac{(j+1)^{\alpha d +\beta}}{v_0^dw_0} \mathbb K_d\left(\frac{Z_j-{x}}{v_j}\right)     \mathbb K_1\left(\frac{S_{j+1}-t}{w_j}\right)   \\
&& -\int_{\mathbb{R}^d\times\mathbb{R}} \mathcal{Q}(\Phi(Z_{j-1},S_j) , {x}+u v_j) f({x}+uv_j , t+v w_j)\,\mathbb{K}_d(u) \mathbb{K}_1(v)\,\lambda_d(\ud u)\,\ud v,\\
M_j^{(2)}- M_{j-1}^{(2)}&=&\frac{(j+1)^{\alpha d}}{v_0^d} \mathbb K_d\left(\frac{Z_j-{x}}{v_j}\right)     \mathbb{1}_{\{S_{j+1}>t\}} \\
&& -\int_{\mathbb{R}^d} \mathcal{Q}(\Phi(Z_{j-1},S_j) , {x}+u v_j) G({x}+uv_j , t)\,\mathbb{K}_d(u)\,\lambda_d(\ud u) , \\
M_j^{(3)}- M_{j-1}^{(3)}&=&\frac{(j+1)^{\alpha d}}{v_0^d} \mathbb K_d\left(\frac{Z_j-{x}}{v_j}\right)   -\int_{\mathbb{R}^d} \mathcal{Q}(\Phi(Z_{j-1},S_j) , {x}+u v_j)\,\mathbb{K}_d(u)\,\lambda_d(\ud u).
\end{eqnarray*}
Thus, we obtain
\begin{eqnarray*}
| M_j^{(1)}- M_{j-1}^{(1)}| &\leq& \frac{(n+1)^{\alpha d+\beta}}{v_0^dw_0}\|\mathbb K_d\|_\infty\|\mathbb K_1\|_\infty+\| Q\|_\infty\| f\|_\infty ,\\
| M_j^{(2)}- M_{j-1}^{(2)}| &\leq& \frac{(n+1)^{\alpha d}}{v_0^d}\|\mathbb K_d\|_\infty+\| Q\|_\infty ,\\
| M_j^{(3)}- M_{j-1}^{(3)}| &\leq& \frac{(n+1)^{\alpha d}}{v_0^d}\|\mathbb K_d\|_\infty+\| Q\|_\infty .\\
\end{eqnarray*}
Together with the condition $\alpha d+\beta <1$, $|M_j- M_{j-1}|=o\left(n^{\frac{1+\alpha d +\beta}{2}}\right)$. As a consequence, there exists an integer $N_\epsilon$ such that for any $n\geq N_\epsilon$, the event $\left\{|M_j- M_{j-1}|\geq \epsilon n^{\frac{1+\alpha d+\beta}{2}}\right\}$ is almost surely empty. This shows Lindeberg's condition $(\ref{Cond_b})$. Finally, by $(\ref{eq:def:sigma})$ and $(\ref{Cond_b})$, we obtain, as $n$ tends to infinity,
\begin{equation}\label{eq:tcl:Mn}
n^{-\frac{{1+\alpha d+\beta}}{2}} \mathcal{M}_n \stackrel{d}{\longrightarrow} \mathcal{N}\left(0_3 , \Sigma({x},t,\alpha,\beta)\right).
\end{equation}





\section{Proof of Proposition \ref{prop:estimise:kappa}}
\label{sec:app:proofprop}

By virtue of the ergodic theorem (see Proposition \ref{ergo:ZS}) applied to the Markov chain $(\widetilde{Z}_n,\widetilde{S}_{n+1})_{n\geq0}$, we have, as $\tilde n$ tends to infinity,
$$\widehat{\varepsilon}^{\,n,\widetilde{n},\rho}_\kappa(\alpha) \stackrel{a.s.}{\longrightarrow} \int_{\mathcal{C}_x}\widehat{\kappa}_x^n(\xi)^2\ud\xi  \,-\, \frac{2\,\Gamma\left(\frac{d-1}{2}+1\right)}{\pi^{\frac{d-1}{2}}\rho^{d-1}} \int_{\mathbb{T}_{x,\rho}\times\mathbb{R}_+}\widehat{\mathcal{G}}^n(\xi,\theta_x(\xi))\,\mathbb{1}_{(\theta_x(\xi),+\infty)}(t)\mu_\infty(\ud\xi\times\ud t) ,$$
conditionally to $\sigma(Z_0,\,S_1,\,\dots,Z_{n-1},\,S_n)$. Together with the definition of $\mu_\infty$ $(\ref{eq:munuinfty})$, the expression $(\ref{eq:defG})$ of $\mathcal{S}(x,\cdot)$ and Remark \ref{rem:nu:density}, we obtain
\begin{equation}\label{eq:proof:ise:1}
\widehat{\varepsilon}^{\,n,\widetilde{n},\rho}_\kappa(\alpha) \stackrel{a.s.}{\longrightarrow} \int_{\mathcal{C}_x}\widehat{\kappa}_x^n(\xi)^2\ud\xi  \,-\, \frac{2\,\Gamma\left(\frac{d-1}{2}+1\right)}{\pi^{\frac{d-1}{2}}\rho^{d-1}} \int_{\mathbb{T}_{x,\rho}}\widehat{\mathcal{G}}^n(\xi,\theta_x(\xi))\,G(\xi,\theta_x(\xi))\nu_\infty(\xi)\lambda_d(\ud\xi) ,
\end{equation}
conditionally to $\sigma(Z_0,\,S_1,\,\dots,Z_{n-1},\,S_n)$. By definition $(\ref{eq:Txrho})$ of $\mathbb{T}_{x,\rho}$ we have
\begin{equation}\label{eq:proof:ise:3}\int_{\mathbb{T}_{x,\rho}}\widehat{\mathcal{G}}^n(\xi,\theta_x(\xi))\,G(\xi,\theta_x(\xi))\nu_\infty(\xi)\lambda_d(\ud\xi) = \int_{\mathbb{D}_{x,\rho}}\!\!\!\left(\int_{\mathcal{C}_y} \widehat{\mathcal{G}}^n(\xi,\theta_y(\xi))\,G(\xi,\theta_y(\xi))\nu_\infty(\xi)\ud\xi\right)\lambda_{d-1}(\ud y) .\end{equation}
In addition, one has
$$\int_{\mathcal{C}_y} \widehat{\mathcal{G}}^n(\xi,\theta_y(\xi))\,G(\xi,\theta_y(\xi))\nu_\infty(\xi)\ud\xi = \int_0^{t^-(y)}\widehat{\mathcal{G}}^n(\Phi(y,-t),t)G(\Phi(y,-t),t)\nu_\infty(\Phi(y,-t))\left|\nabla_t\Phi(y,-t)\right|\ud t .$$
As a consequence, under Assumptions \ref{hyps:ise:regularity}, a conscientious calculus together with eq. $(\ref{eq:tautheta})$ shows that
\begin{equation}
\label{eq:proof:ise:2}
\left|\int_{\mathcal{C}_y} \widehat{\mathcal{G}}^n(\xi,\theta_y(\xi))\,G(\xi,\theta_y(\xi))\nu_\infty(\xi)\ud\xi - \int_{\mathcal{C}_x} \widehat{\mathcal{G}}^n(\xi,\tau_x(\xi))\,G(\xi,\tau_x(\xi))\nu_\infty(\xi)\ud\xi\right| = O(|x-y|).
\end{equation}
We obtain the expected result from $(\ref{eq:proof:ise:1})$, $(\ref{eq:proof:ise:3})$ and $(\ref{eq:proof:ise:2})$ and by remarking that the normalizing constant is $\lambda_{d-1}(\mathbb{D}_{x,\rho})$.

\begin{rem}
In order to prove $(\ref{eq:proof:ise:2})$, one may split the integral of interest into two terms: an integral between $0$ and $t^-(x)\wedge t^-(y)$ and a remainder term integrated between $t^-(x)\wedge t^-(y)$ and $t^-(x)\vee t^-(y)$. The first one is clearly upper bounded by the integral of a Lipschitz function 
between $0$ and $t^-(x) + |x-y|[t^-]_{Lip}$, thus bounded by an integral between $0$ and $t^-(x) + \rho[t^-]_{Lip}$. The second integral is obviously bounded by $\|\widehat{\mathcal G}^n\|_\infty\|\nu_\infty\|_\infty \|\nabla_t\Phi\|_\infty [t^-]_{Lip}|x-y|$.
\end{rem}

\bibliographystyle{plain}
\bibliography{tex_main}

\newpage


\begin{figure}[p]
\centering\includegraphics[width=6cm,height=5cm]{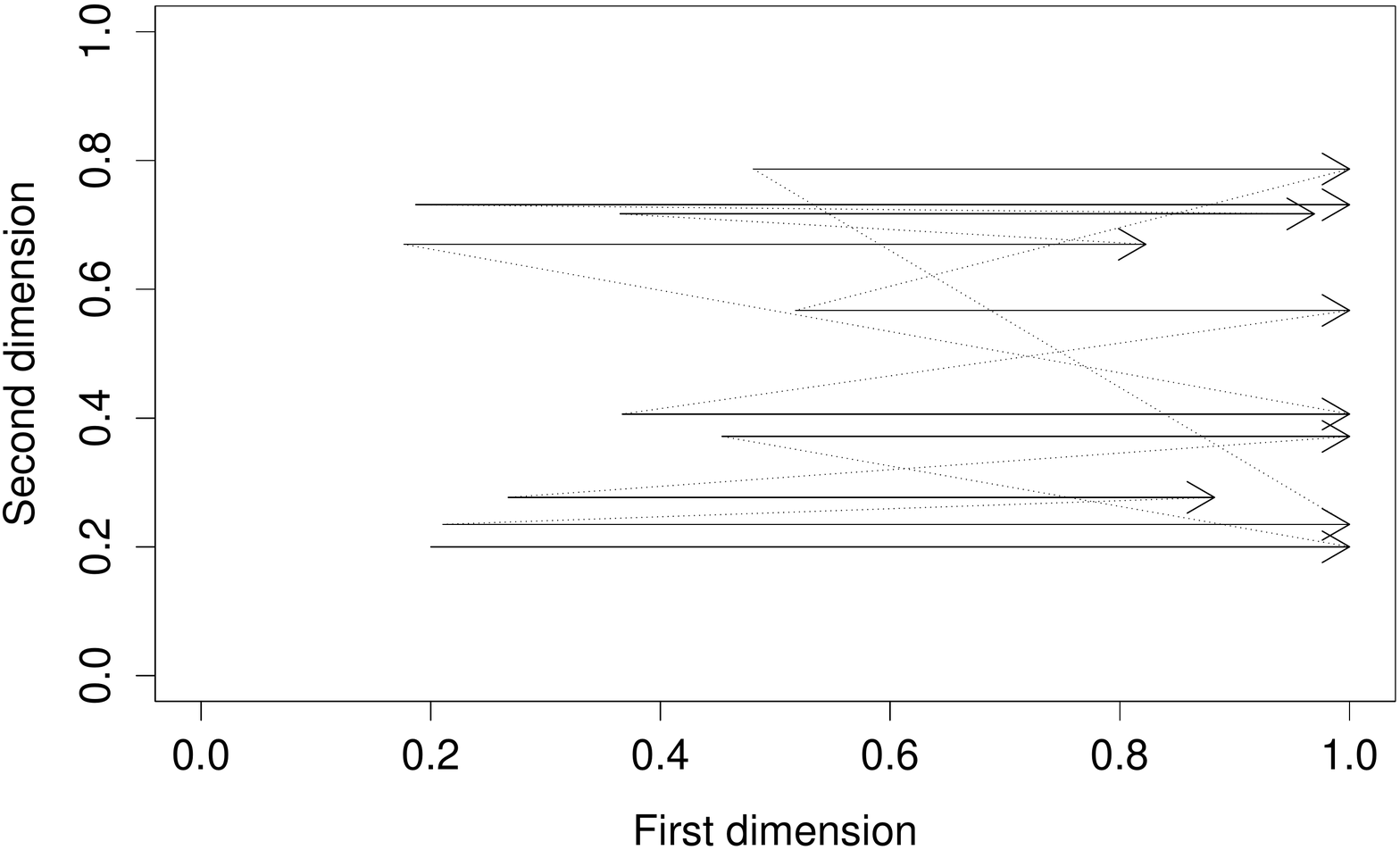}\qquad\includegraphics[width=6cm,height=5cm]{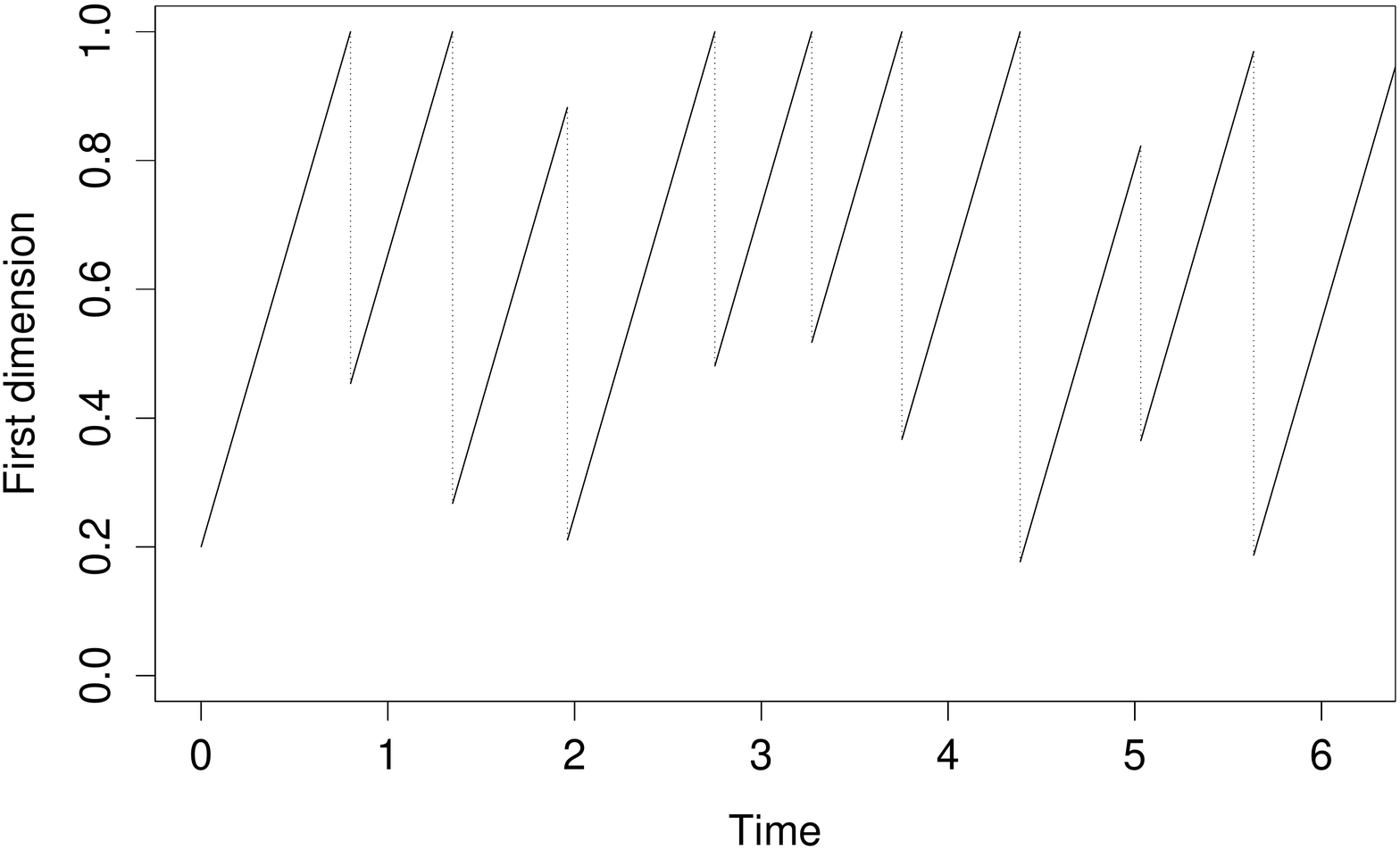}
\caption{Two representations of the same simulated path of the TCP-like model of interest until the 10$^{\text{th}}$ jump. A vector field graph is given on the left, while we observe the trajectory of the first component versus time on the right.}
\label{fig:simulations}
\end{figure}

\begin{figure}[p]
\centering\includegraphics[width=5cm]{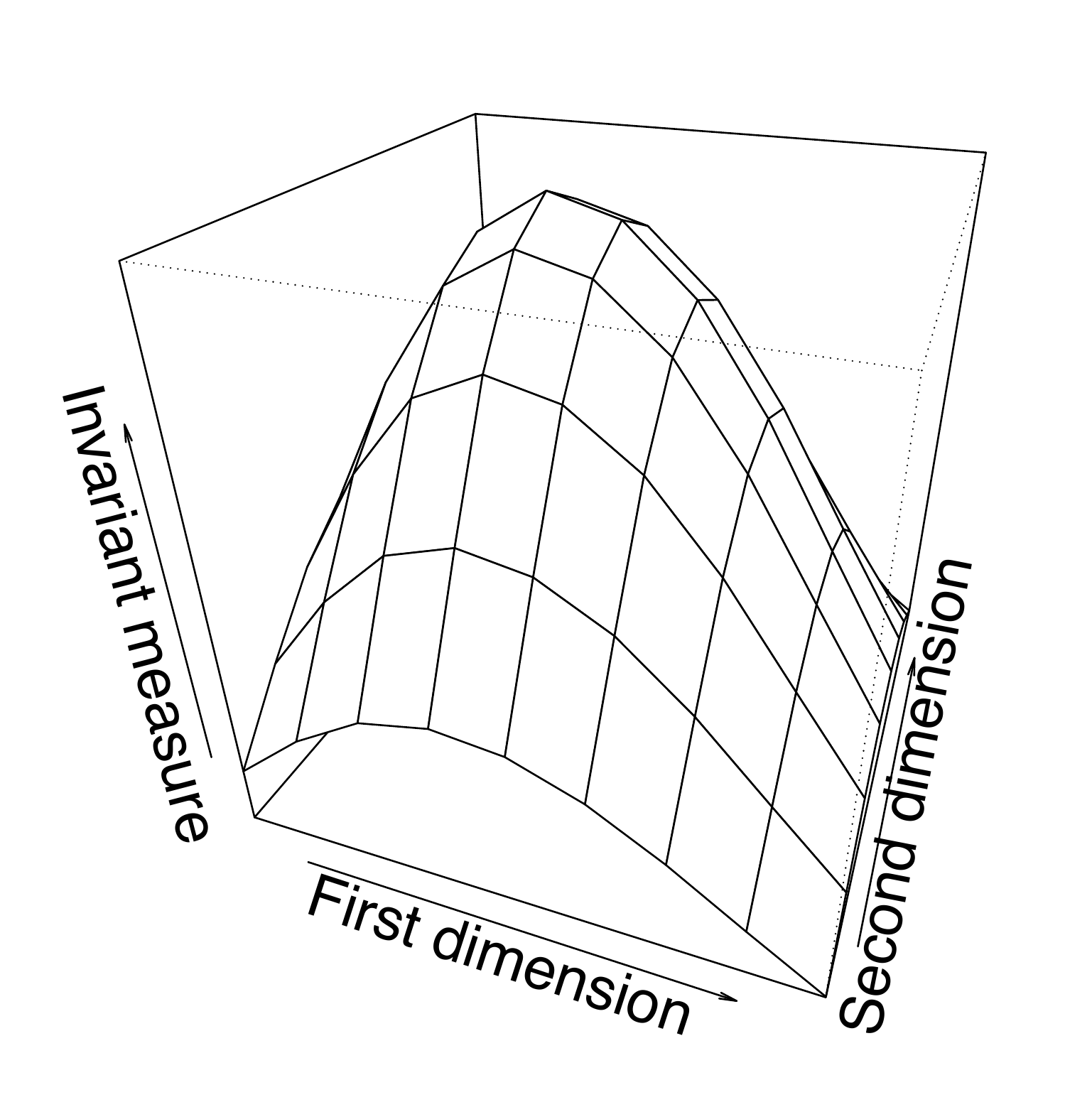}\qquad\includegraphics[width=8cm]{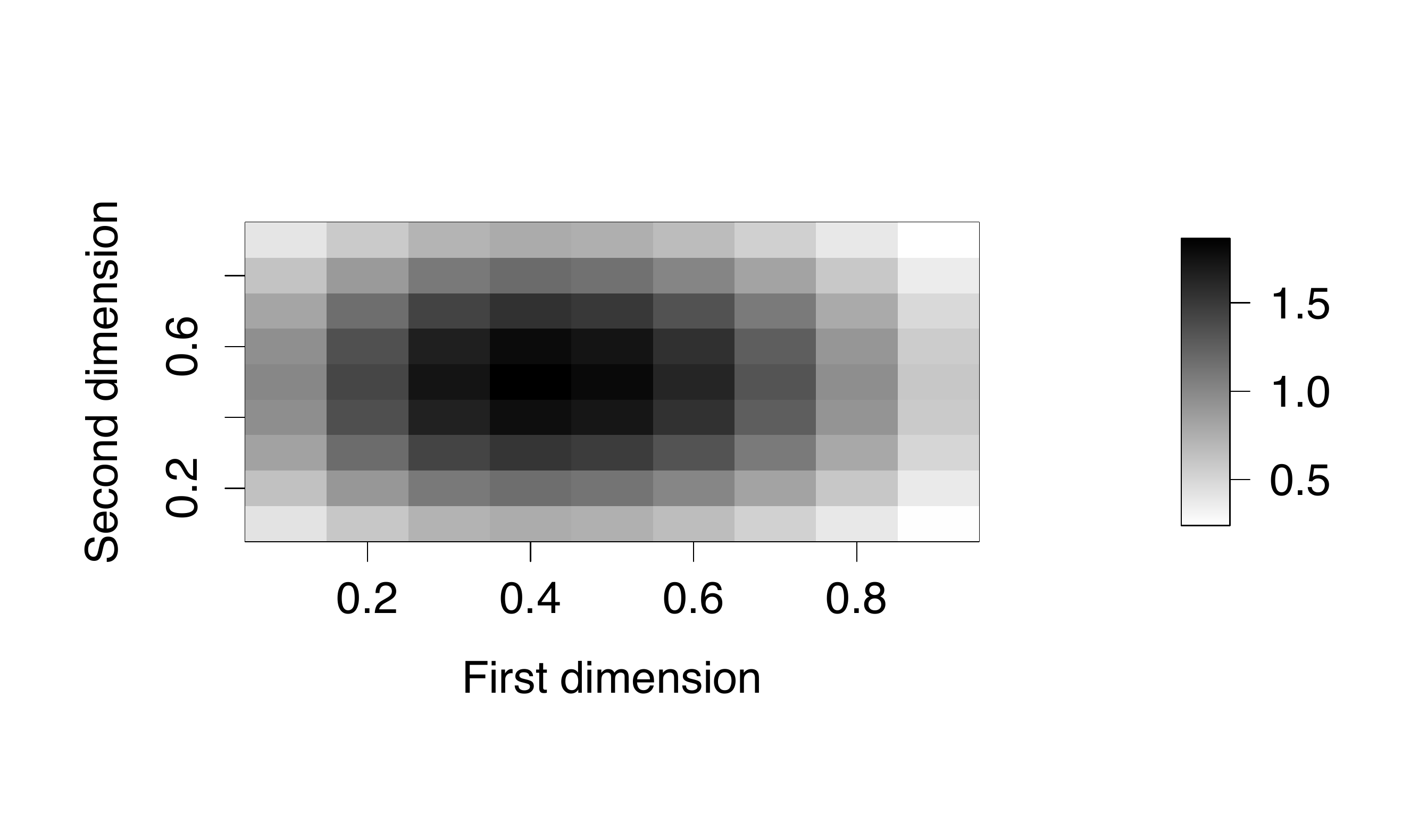}
\caption{Estimation of the invariant distribution of the post-jump locations computed from the $20\,000$ first jumps of the TCP-like model.}
\label{fig:nuhat}
\end{figure}

\begin{figure}[p]
\centering\includegraphics[width=5cm]{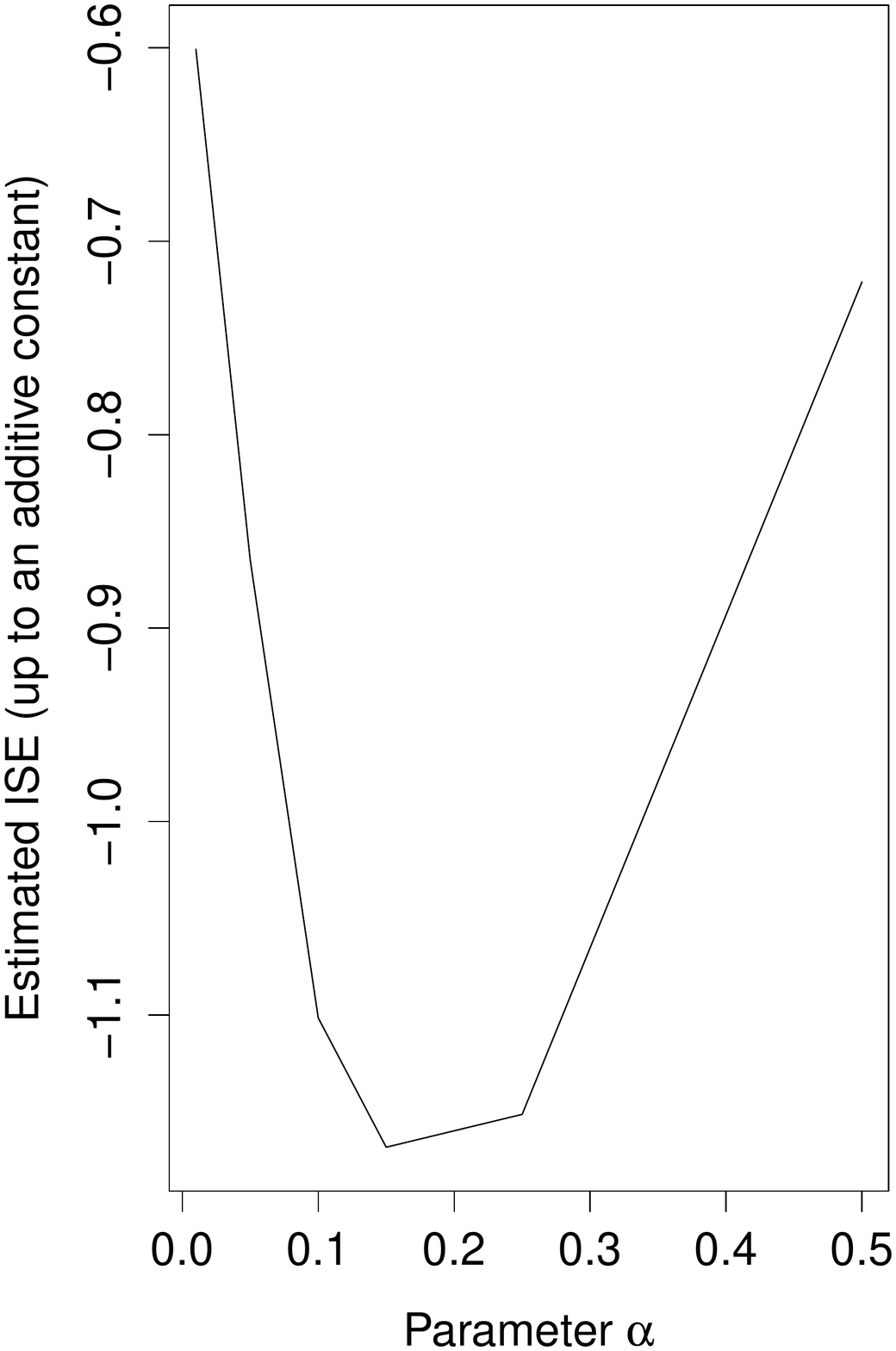}~\includegraphics[width=5cm]{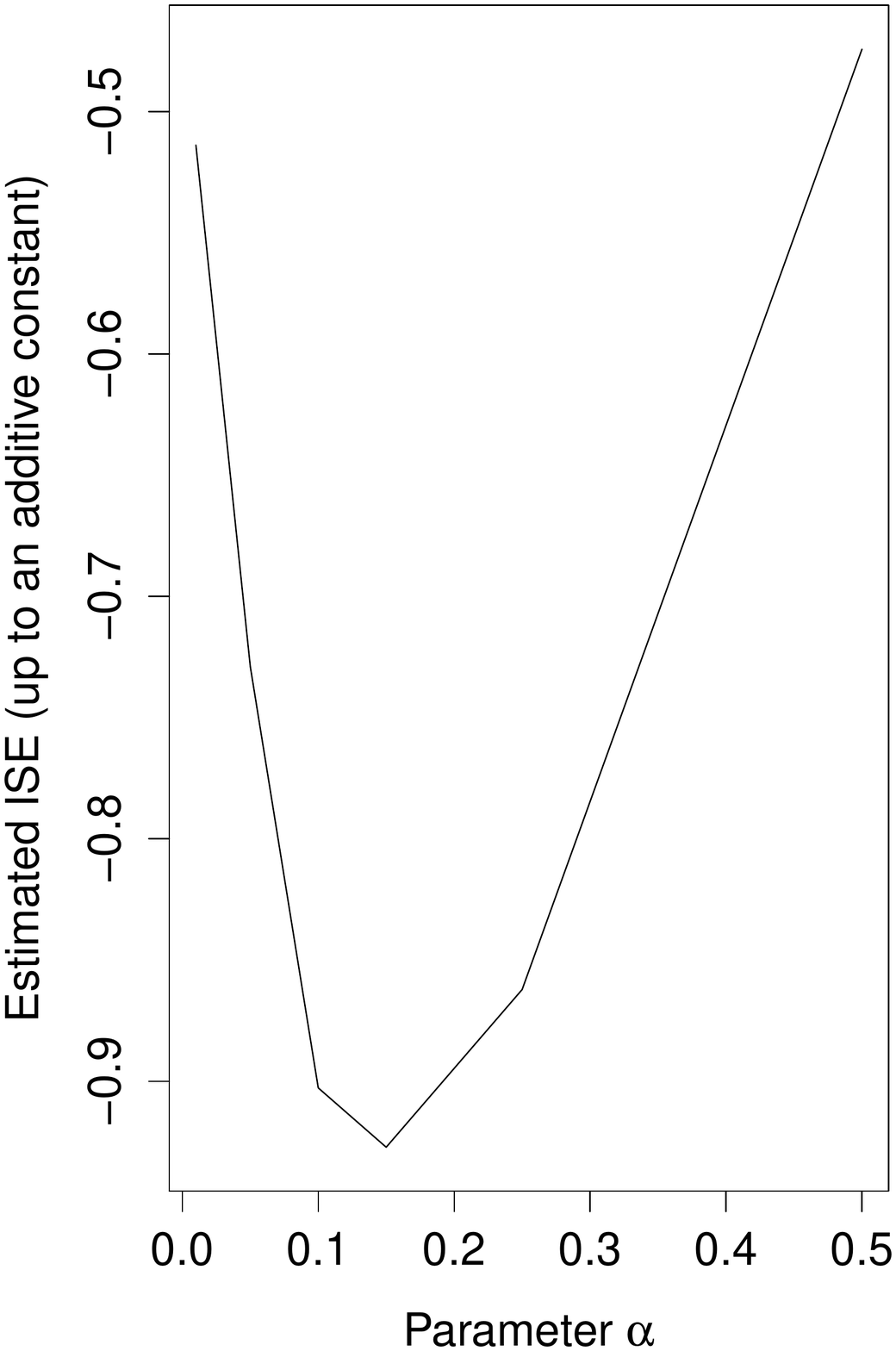}~\includegraphics[width=5cm]{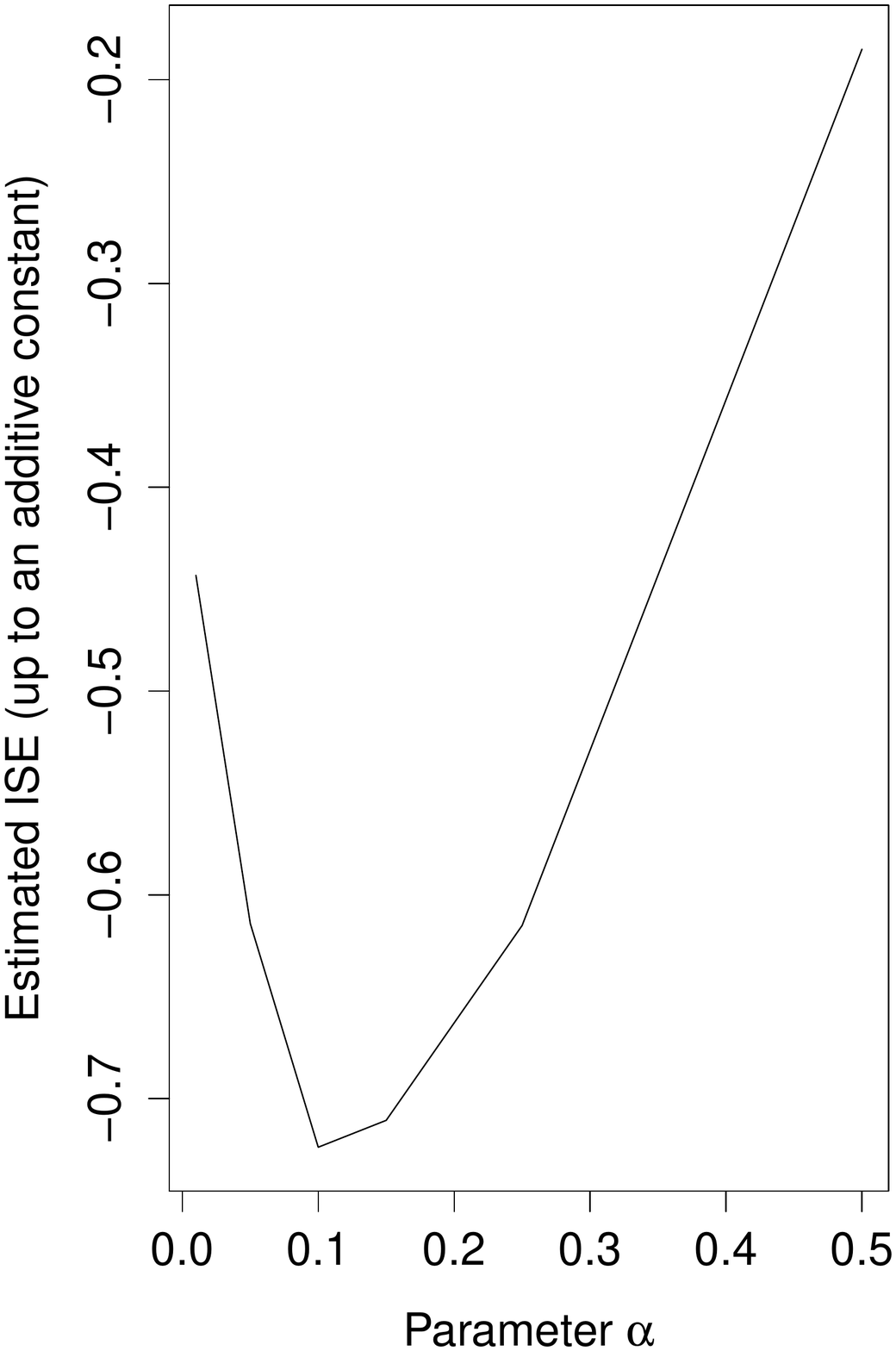}
\caption{The choice of the bandwidth parameter $\alpha$ appearing in the criterion $\widehat{\kappa}^n_x$ is obtained by minimizing the cross-validation estimate $ \widehat{\varepsilon}^{\,n,\widetilde{n},\rho}_\kappa(\alpha)$ of the related ISE. The parameter $\rho$ seems to have only a small influence over the minimization: the estimated error is computed from $\rho=0.005$, $\rho=0.01$ and $\rho=0.02$ from left to right.}
\label{fig:alphaChoice}
\end{figure}

\begin{figure}[p]
\centering\includegraphics[width=8cm]{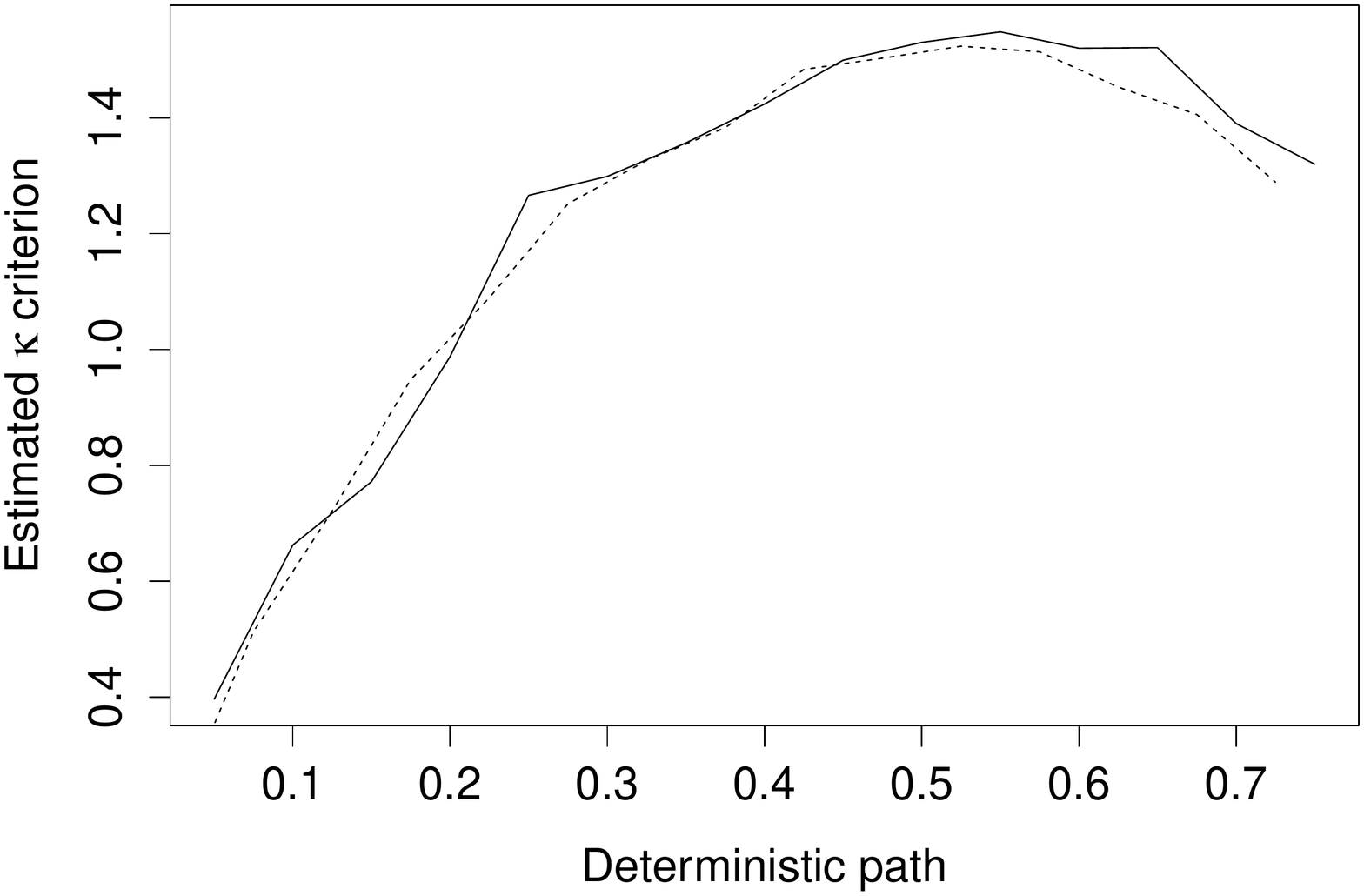}\quad\includegraphics[width=8cm]{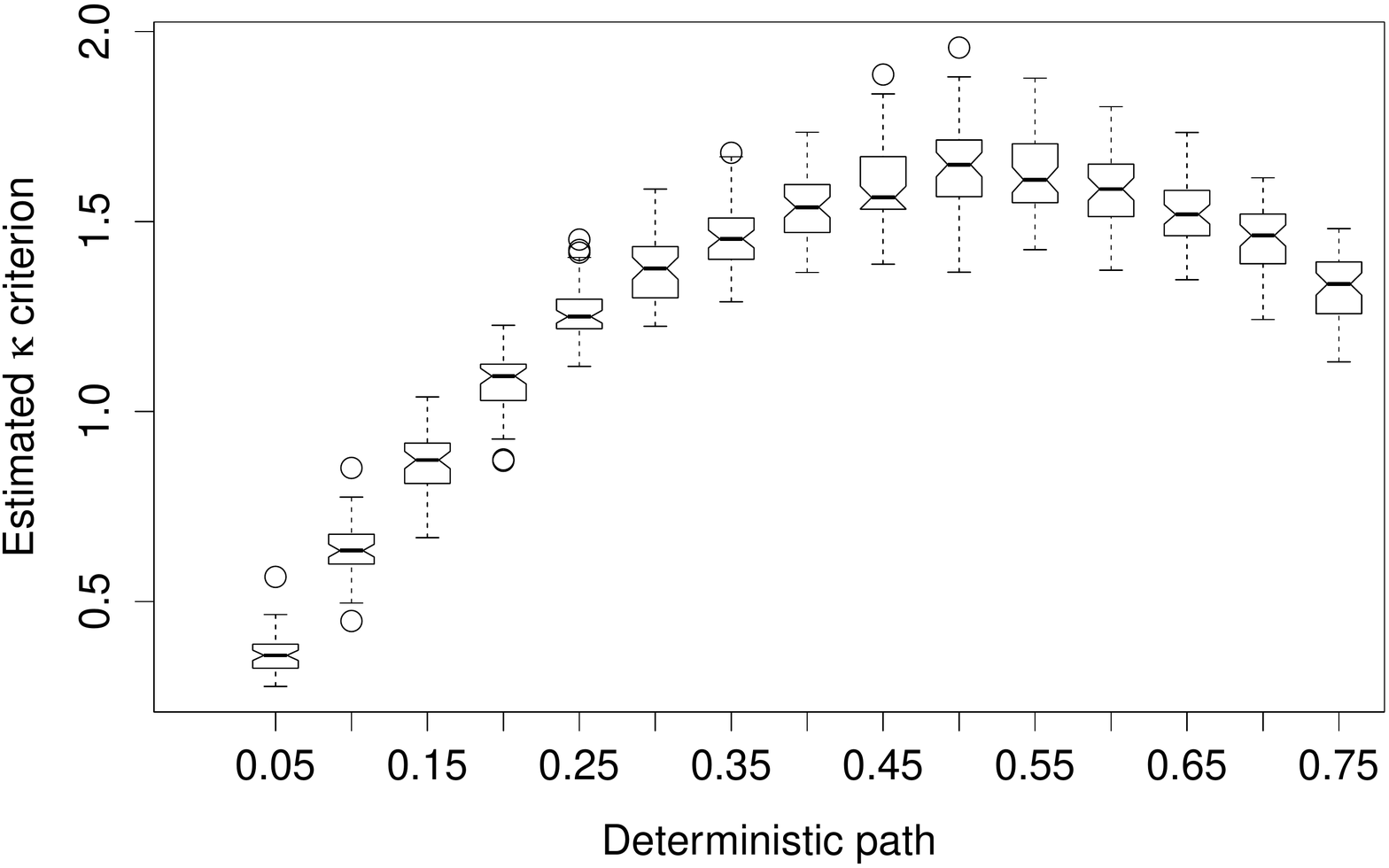}
\caption{The optimal index $\widehat{\xi}^n_*$ is calculated by maximizing the criterion $\widehat{\kappa}^n_x(\xi)$, $\xi\in\mathcal{C}_x$, computed with the optimal parameter $\alpha$. On the left side of the figure, we compare the estimate $\widehat{\kappa}^n_x(\xi)$ (full line) with $\widehat{\nu}_\infty^n(\xi)G(\xi,\tau_x(\xi))$ (dashed line) both computed from $\alpha=0.1$, which confirms that the estimation performs pretty well. This quantity seems to admit one and only one absolute maximum.}
\label{fig:optimalpoint}
\end{figure}

\begin{figure}[p]
\includegraphics[width=7cm]{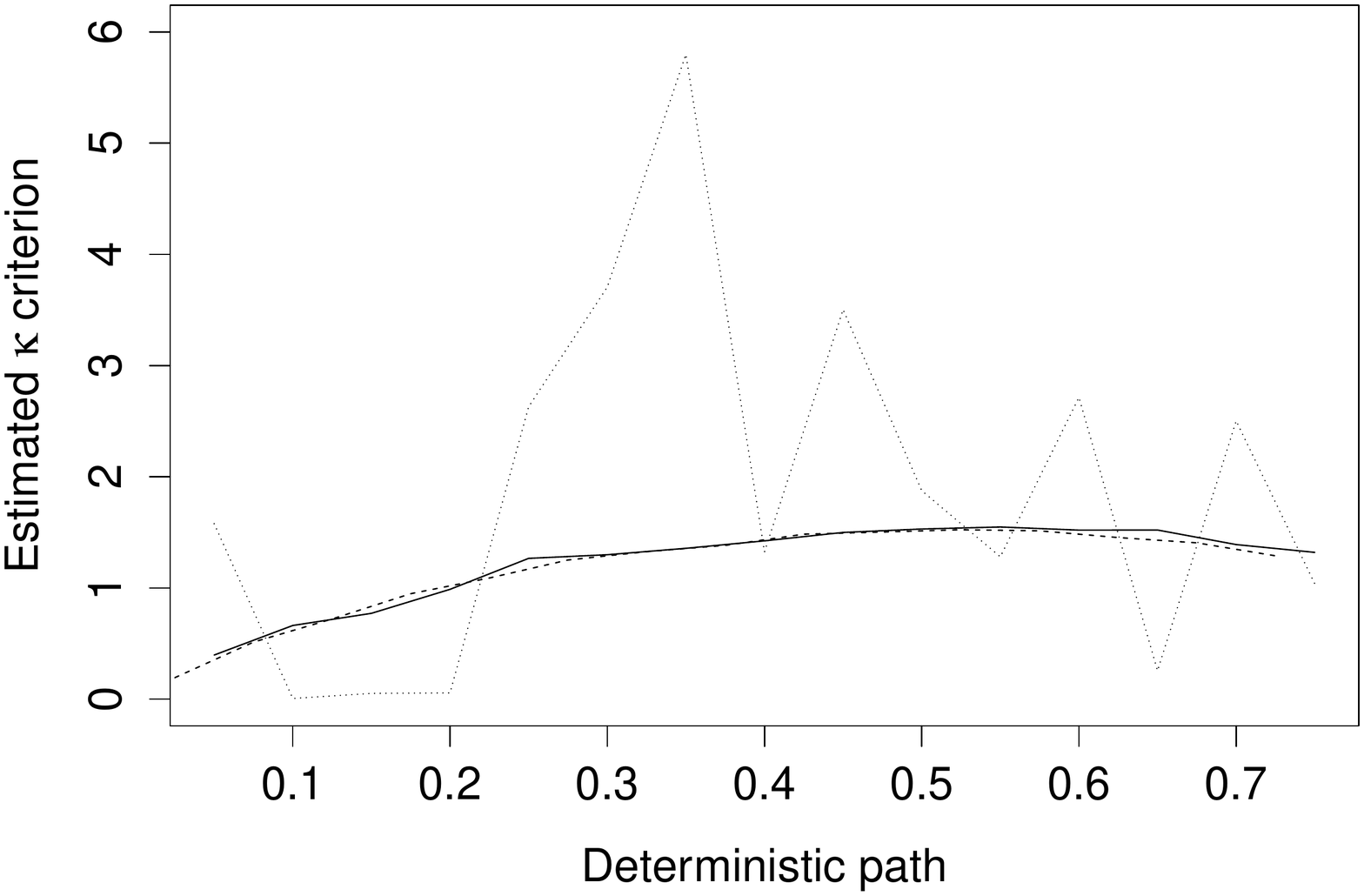}
\caption{The parameter $\alpha$ plays a crucial role in the estimation of $\widehat{\kappa}^n_x$. We compare here the curves computed from $\alpha=0.1$ and already presented in Figure \ref{fig:optimalpoint} (full and dashed lines) with the too oscillating estimate obtained from $\alpha=0.4$ (dotted line).}
\label{fig:badalpha}
\end{figure}

\begin{figure}[p]
\centering\includegraphics[width=5cm]{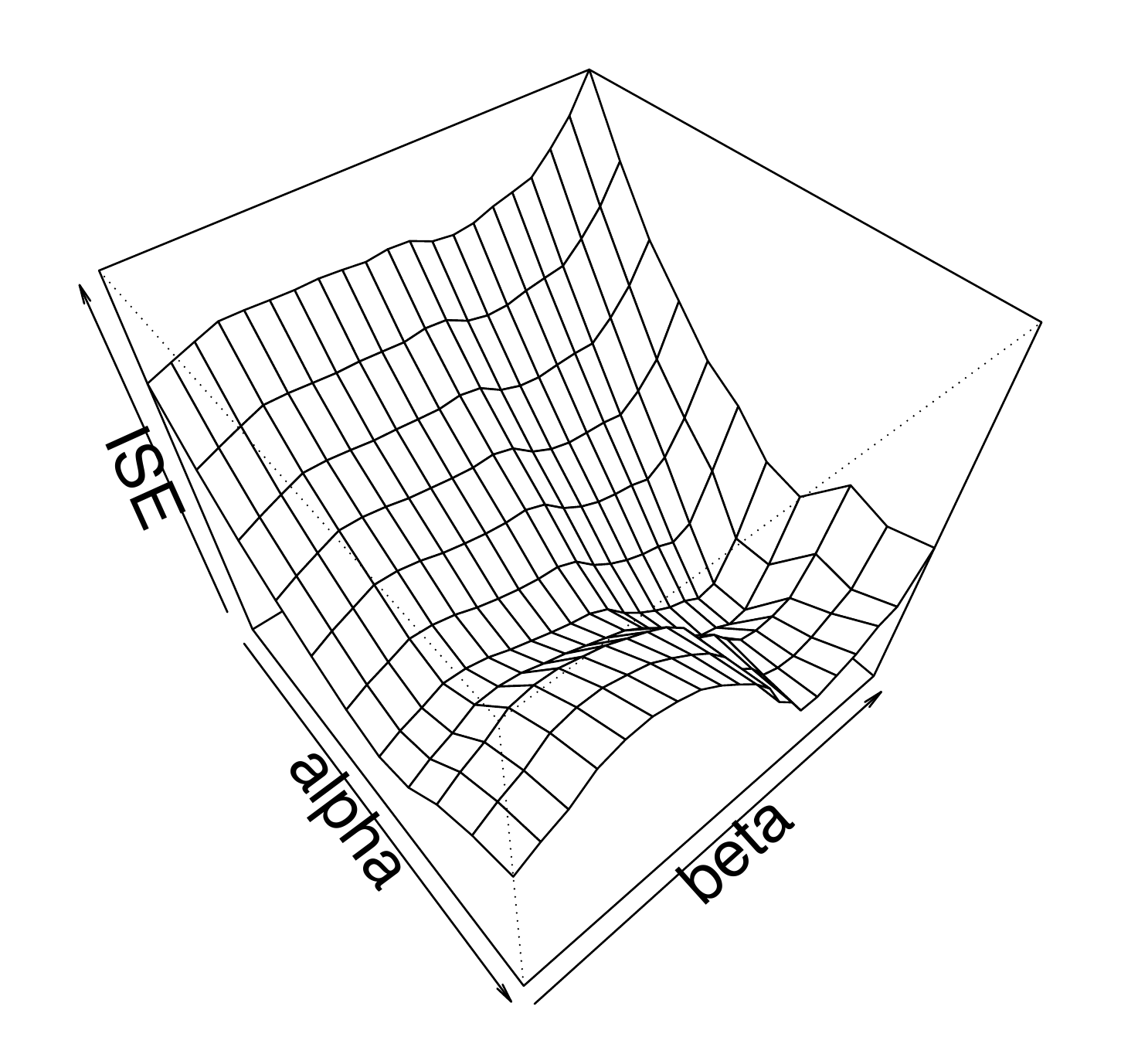}\qquad\includegraphics[width=8cm]{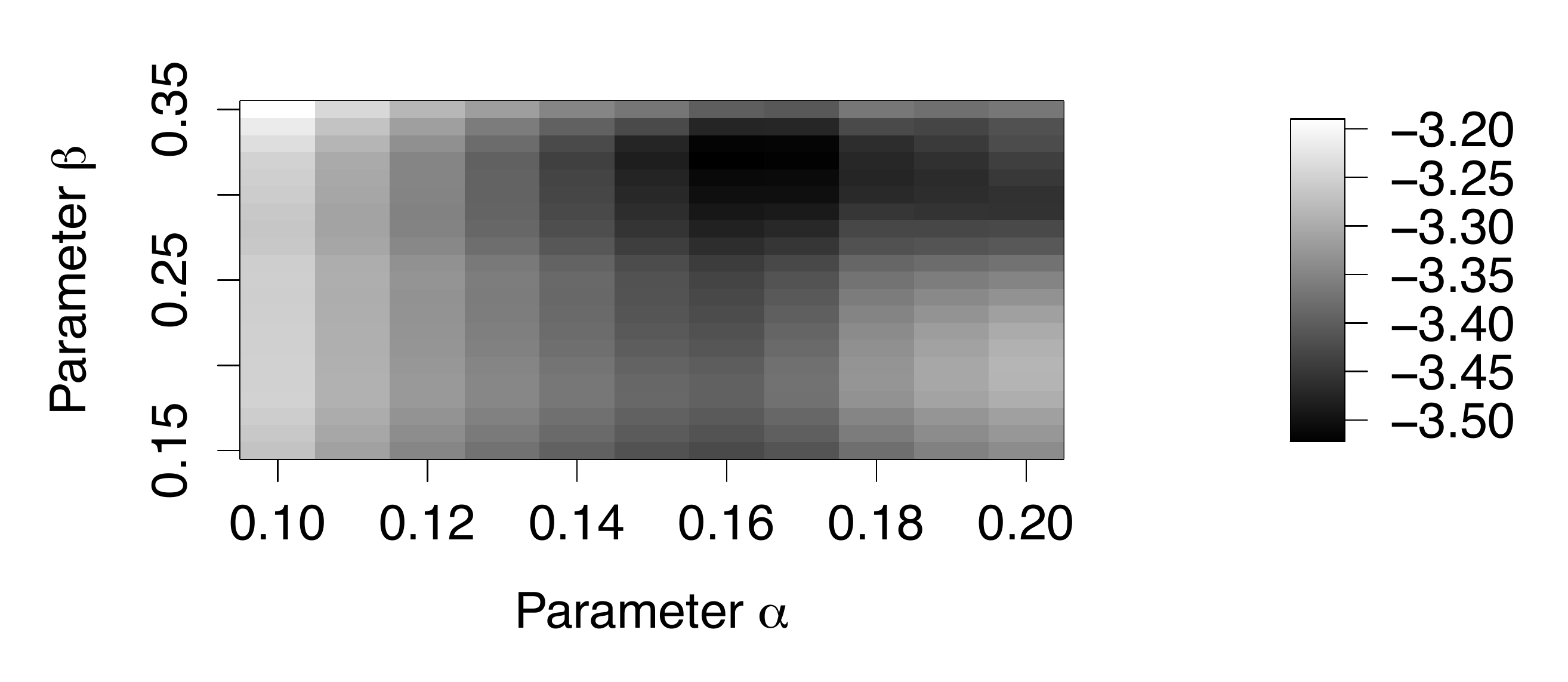}
\caption{The optimal bandwidth parameters $\alpha$ and $\beta$ implicitly appearing in $\widehat{\mathcal F}^n(\xi, \tau_x(\xi))$, $\xi\in\mathcal{C}_x$, are obtained by minimizing the cross-validation estimate $\widehat{\varepsilon}^{\,n,\widetilde{n},\rho_1,\rho_2}_{\mathcal F}(\alpha,\beta)$, computed from $\rho_1=\rho_2=0.1$ here, of the related ISE. The parameter $\beta$ seems to have only little influence over the estimation error in comparison to $\alpha$.}
\label{fig:alphabetaChoice}
\end{figure}

\begin{figure}[p]
\centering\includegraphics[width=9cm]{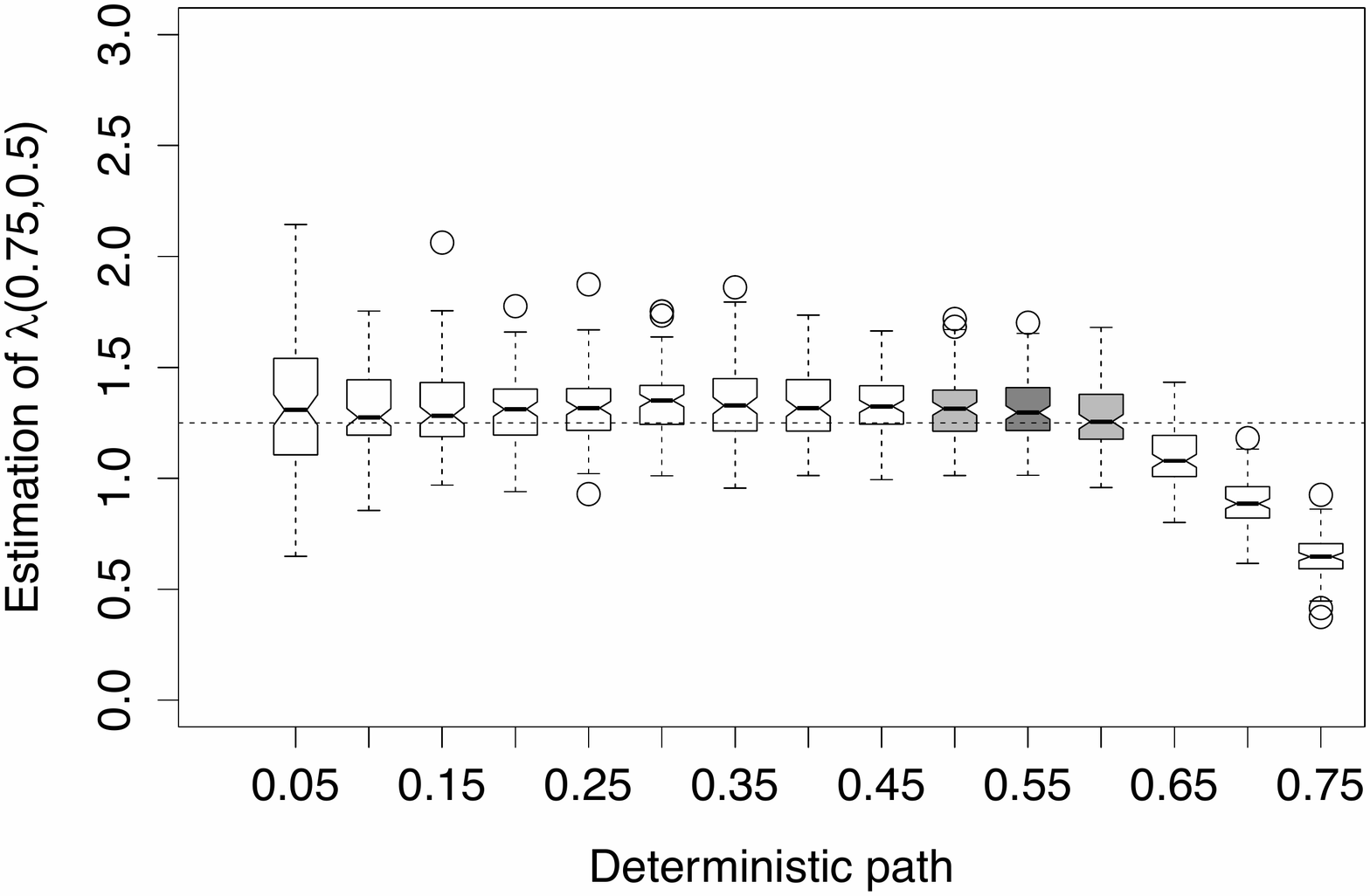}
\caption{The estimator $ \widehat{\lambda}^n_{\xi}(x)$ is computed from different values of the index $\xi\in\mathcal{C}_x$, and from the optimal bandwidth parameters $\alpha^{\mathcal G}$ and $(\alpha^{\mathcal F},\beta^{\mathcal F})$. Over the $100$ replicates, the optimal points $\widehat{\xi}^n_\ast$ are most of the time located at $0.55$ and more than $90$ times between $0.5$ and $0.6$ (enhanced by gray colors), which seems to correspond with the estimators with least bias and variance. In particular, we obtain a better result than by maximizing the estimated invariant measure $\widehat{\nu}^n_{\infty}$ (index around $0.35$), see Remark \ref{rem:choicecriterion}.}
\label{fig:boxplotresult}
\end{figure}




\begin{figure}[p]
\centering\includegraphics[width=8cm,height=5.5cm]{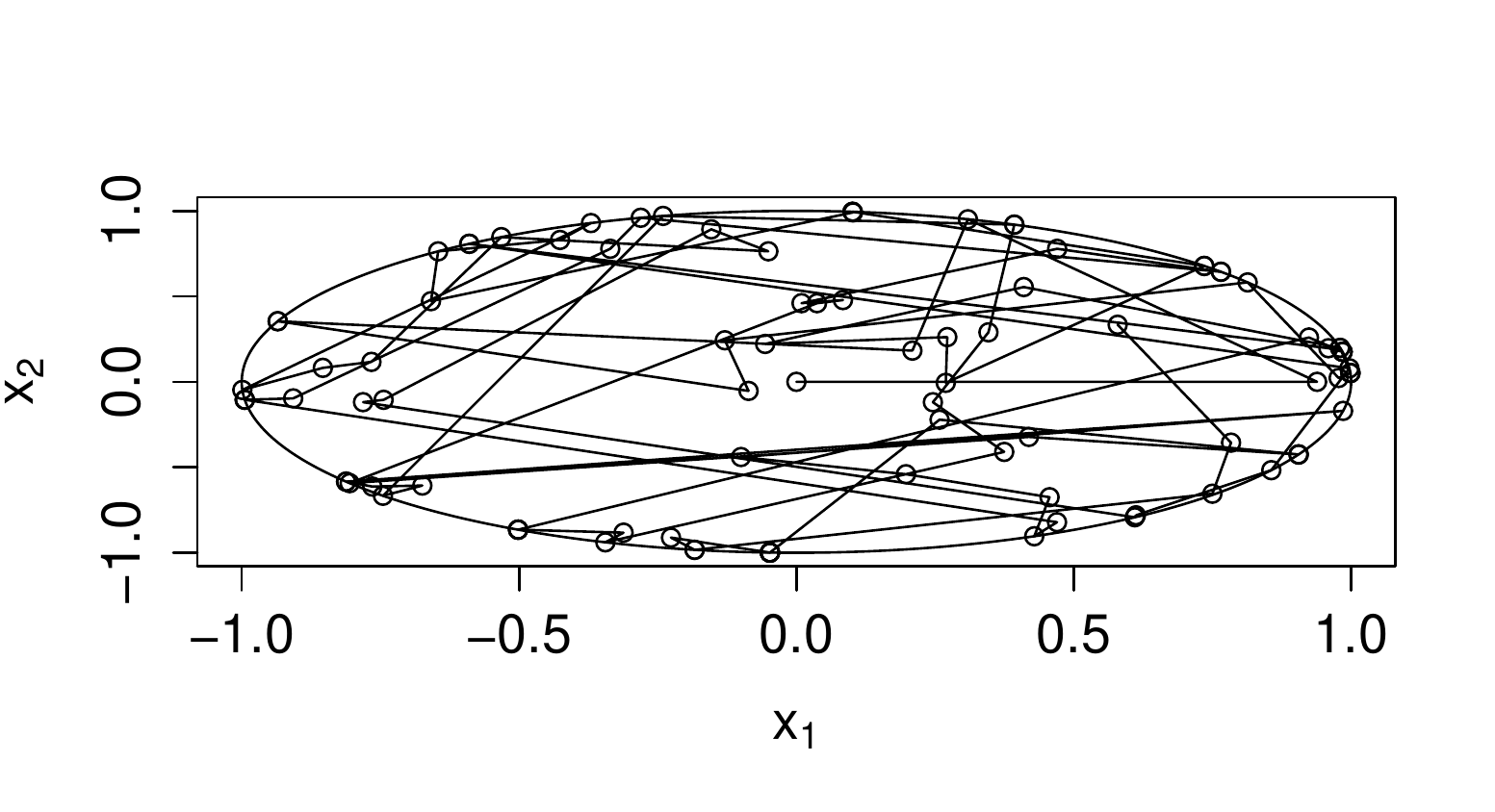}
\caption{Simulated trajectory of the bacteria in the unit disk until the $100^{\text{th}}$ jump.}
\label{fig:pathBacteria}
\end{figure}

\begin{figure}[p]
\centering\includegraphics[width=9cm]{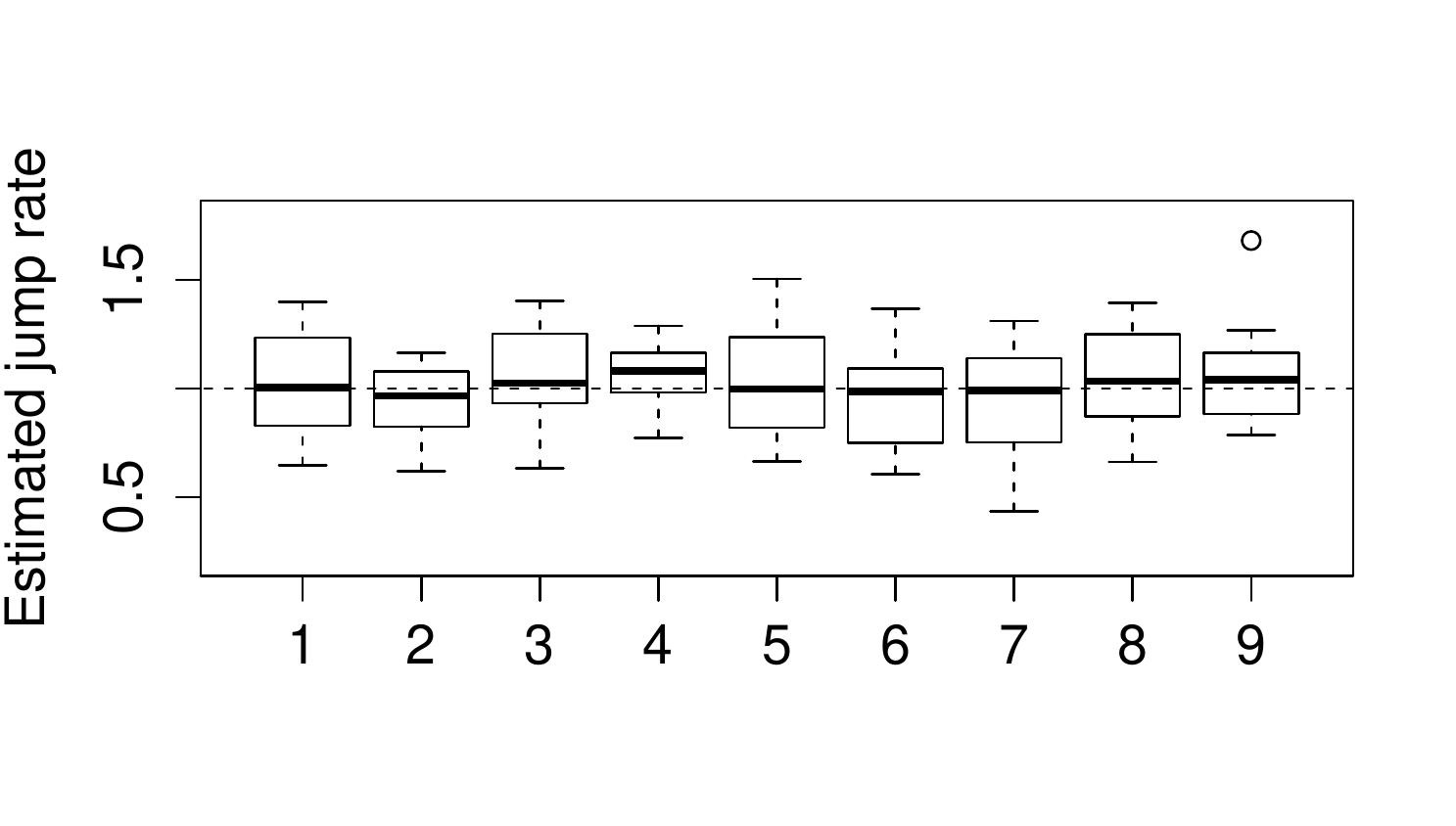}
\caption{For each of the 9 target points $(x_1,x_2)$, the boxplot of the optimal estimators $\widehat{ \lambda}_{\widehat \xi^n_*(\theta)}^n(x_\theta)$ for $\theta$ taken in a uniform discretized grid of $[0,2\pi]$ with  step $\pi/8$ is presented. The black thick lines in the boxplots correspond to the aggregated estimate $\widehat \lambda(x_1,x_2)$ defined in $(\ref{eq:lambdathetaint})$, while the dashed line represents the true jump rate.}
\label{fig:estimations_100000}
\end{figure}

\begin{figure}[p]
\centering\includegraphics[width=8cm,height=5.5cm]{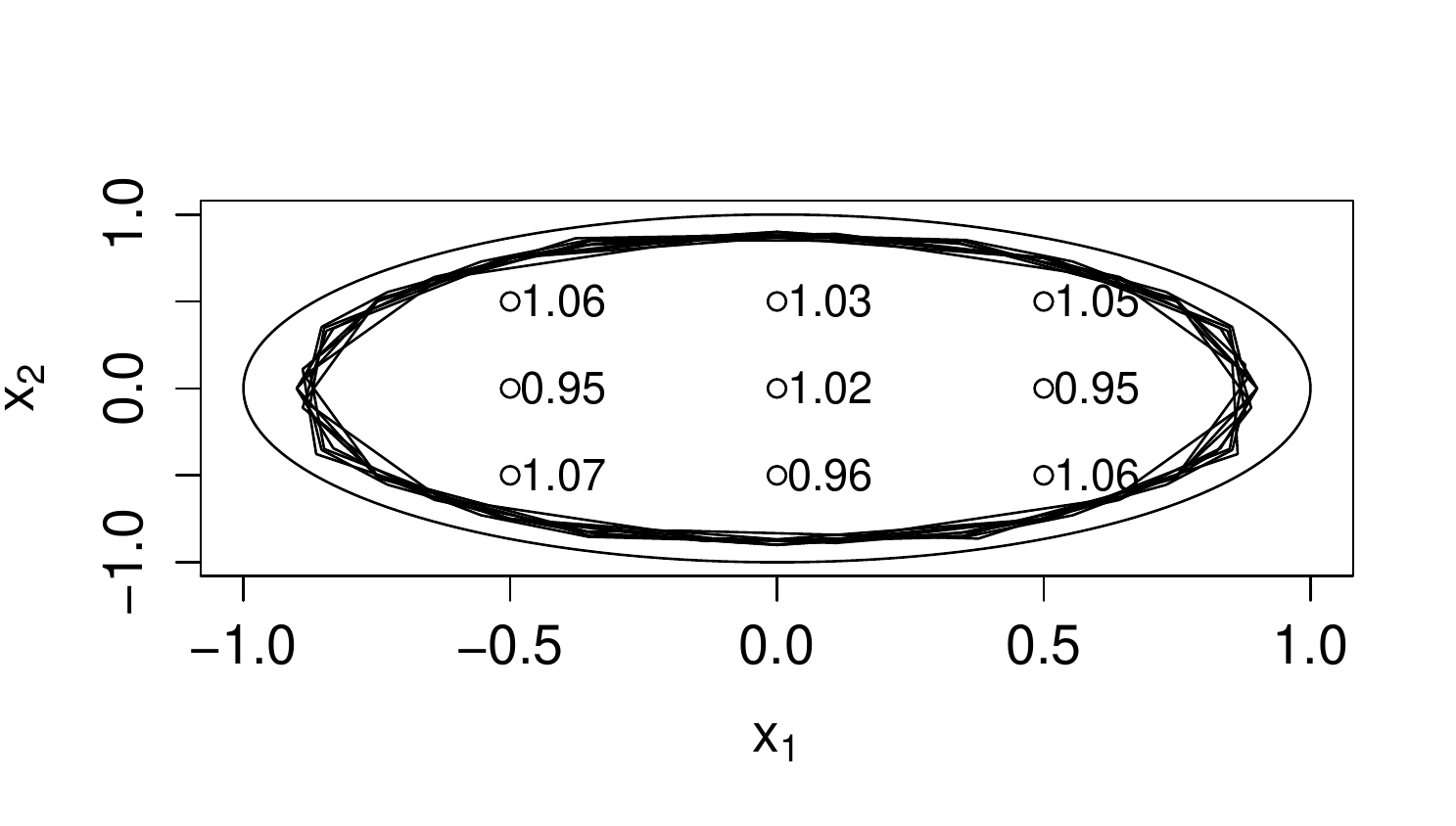}\qquad\includegraphics[width=6cm]{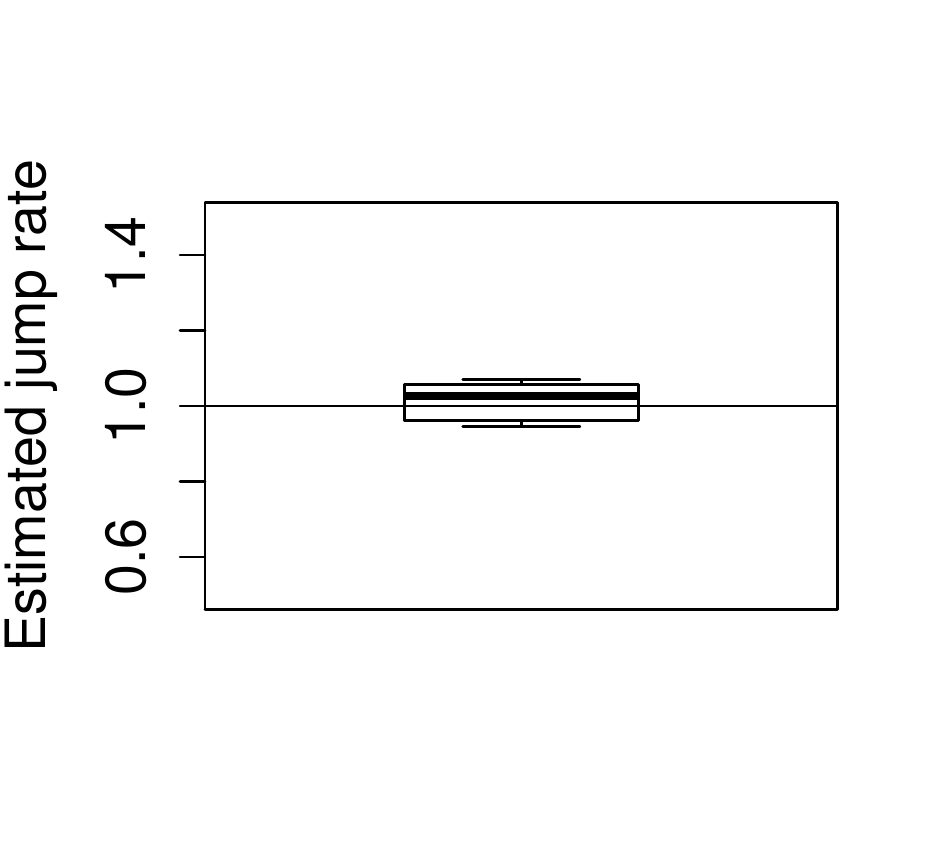}
\caption{For each of the 9 target points, the curve $\frak{C}_{x_1,x_2}$ as well as the estimate $\widehat{\lambda}(x_1,x_2)$ have been computed from $n=100\,000$ observed data and are provided (left). The boxplot of these 9 estimates is also given (right) and presents no bias and a small dispersion.}
\label{fig:fleurs_100000}
\end{figure}

\begin{figure}[p]
\centering\includegraphics[width=8cm,height=5.5cm]{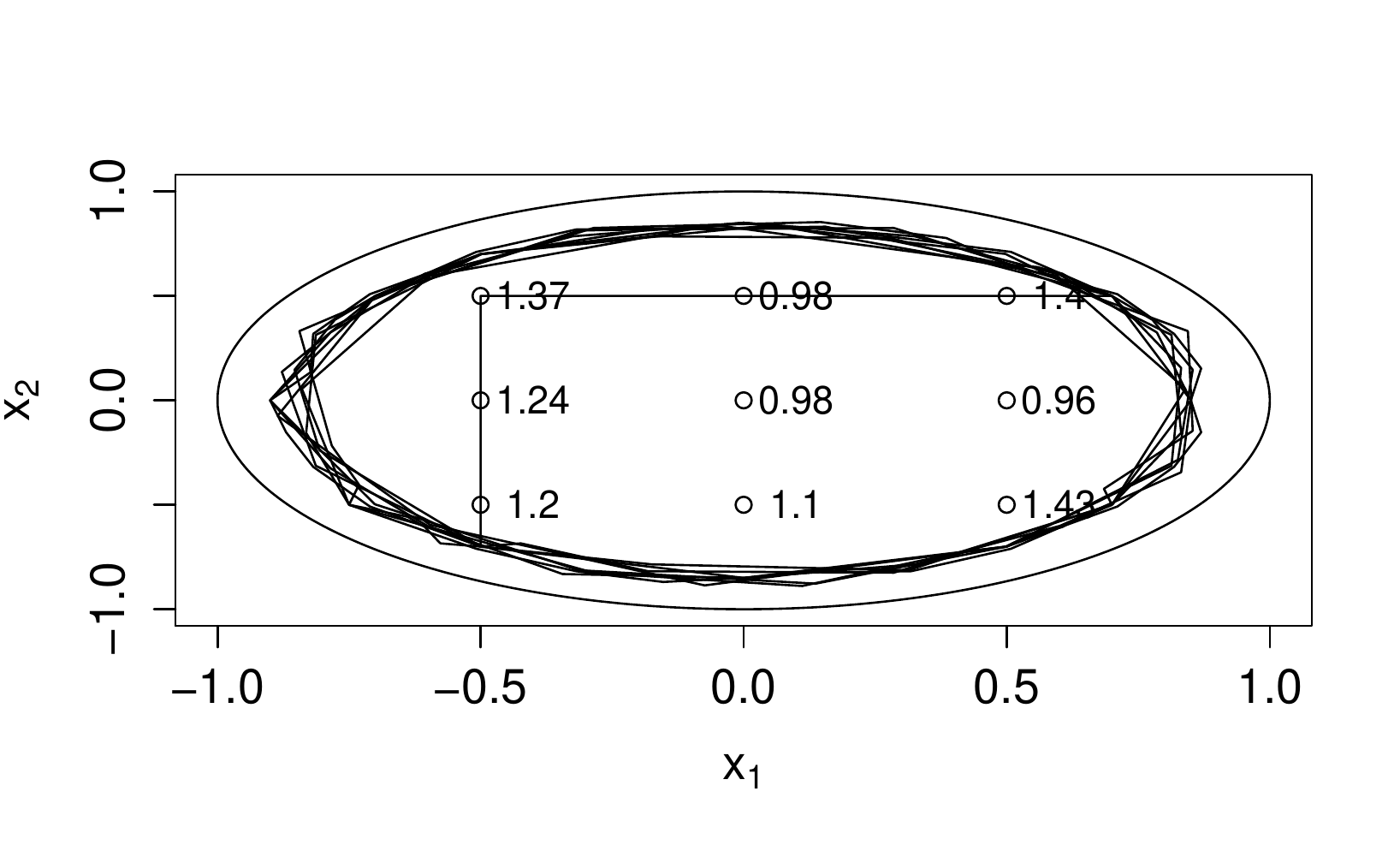}\qquad\includegraphics[width=6cm]{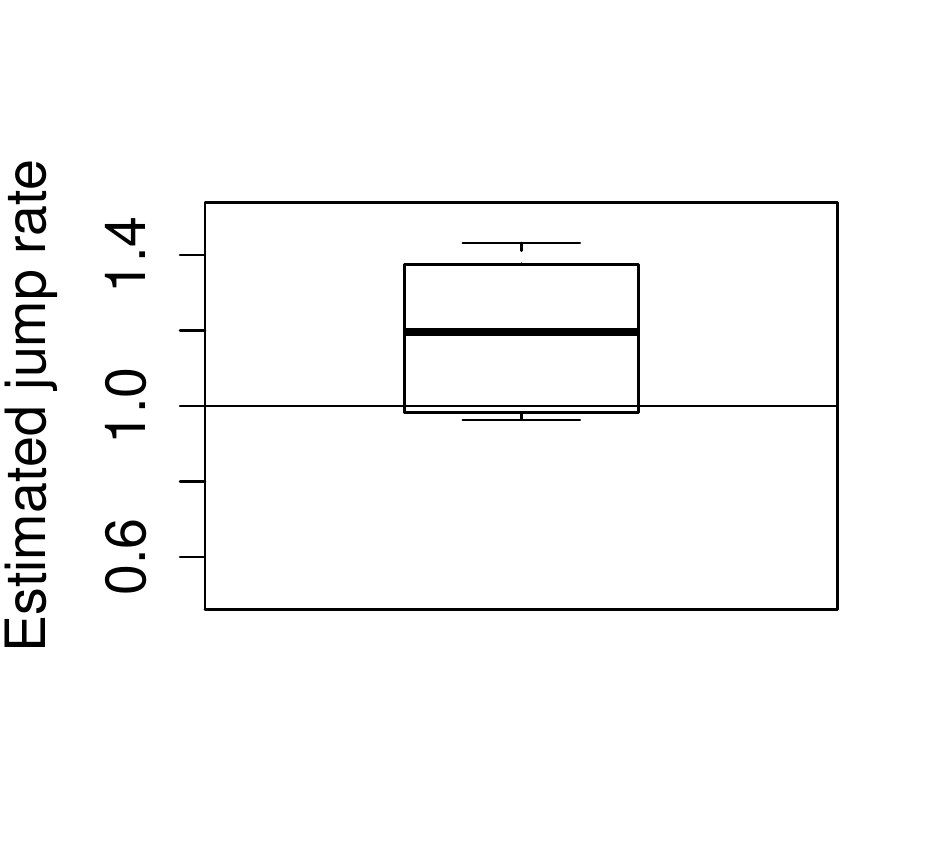}
\centering\includegraphics[width=8cm,height=5.5cm]{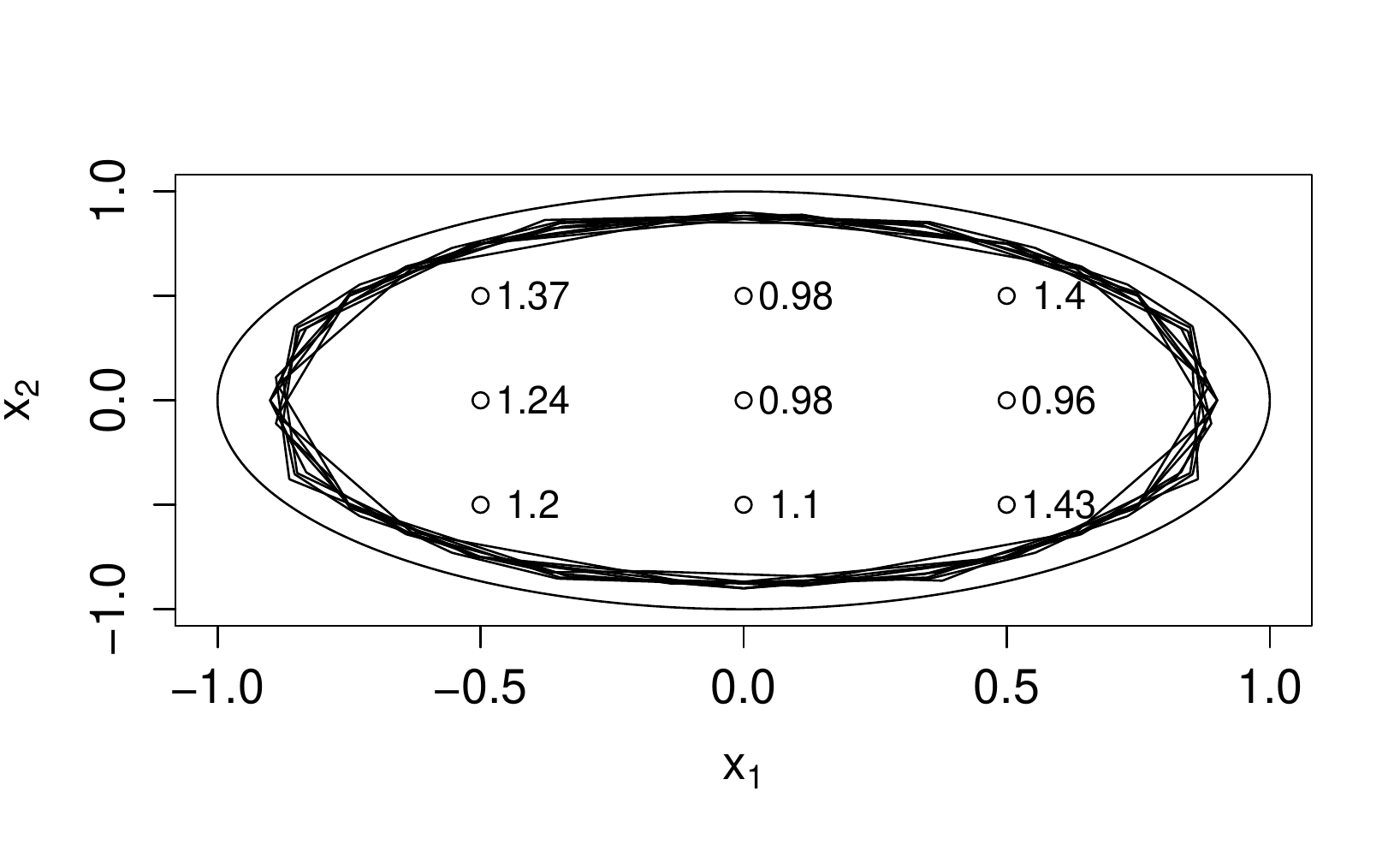}\qquad\includegraphics[width=6cm]{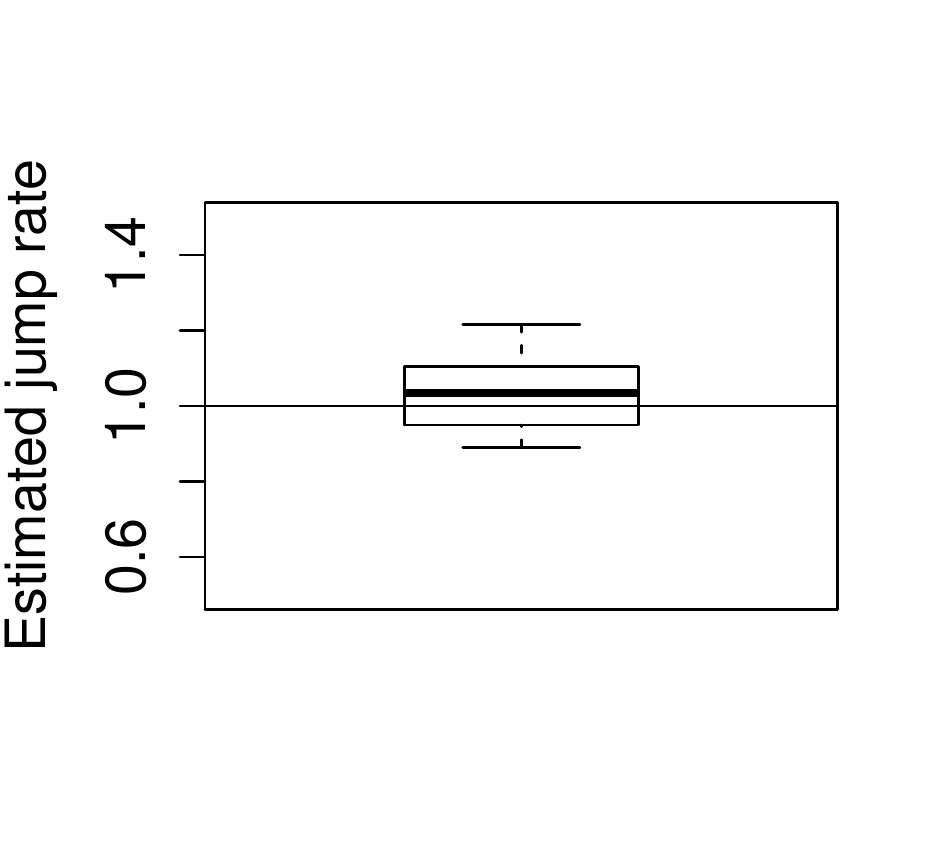}
\caption{For each of the 9 target points, the curve $\frak{C}_{x_1,x_2}$ as well as the estimate $\widehat{\lambda}(x_1,x_2)$ have been computed from $n=20\,000$ (top) and $n=50\,000$ (bottom) observed data and are provided (left). The boxplots of these 9 estimates are also given in both cases (right).}
\label{fig:fleursComparN20000}
\end{figure}

\begin{figure}[p]
\centering\includegraphics[width=9cm]{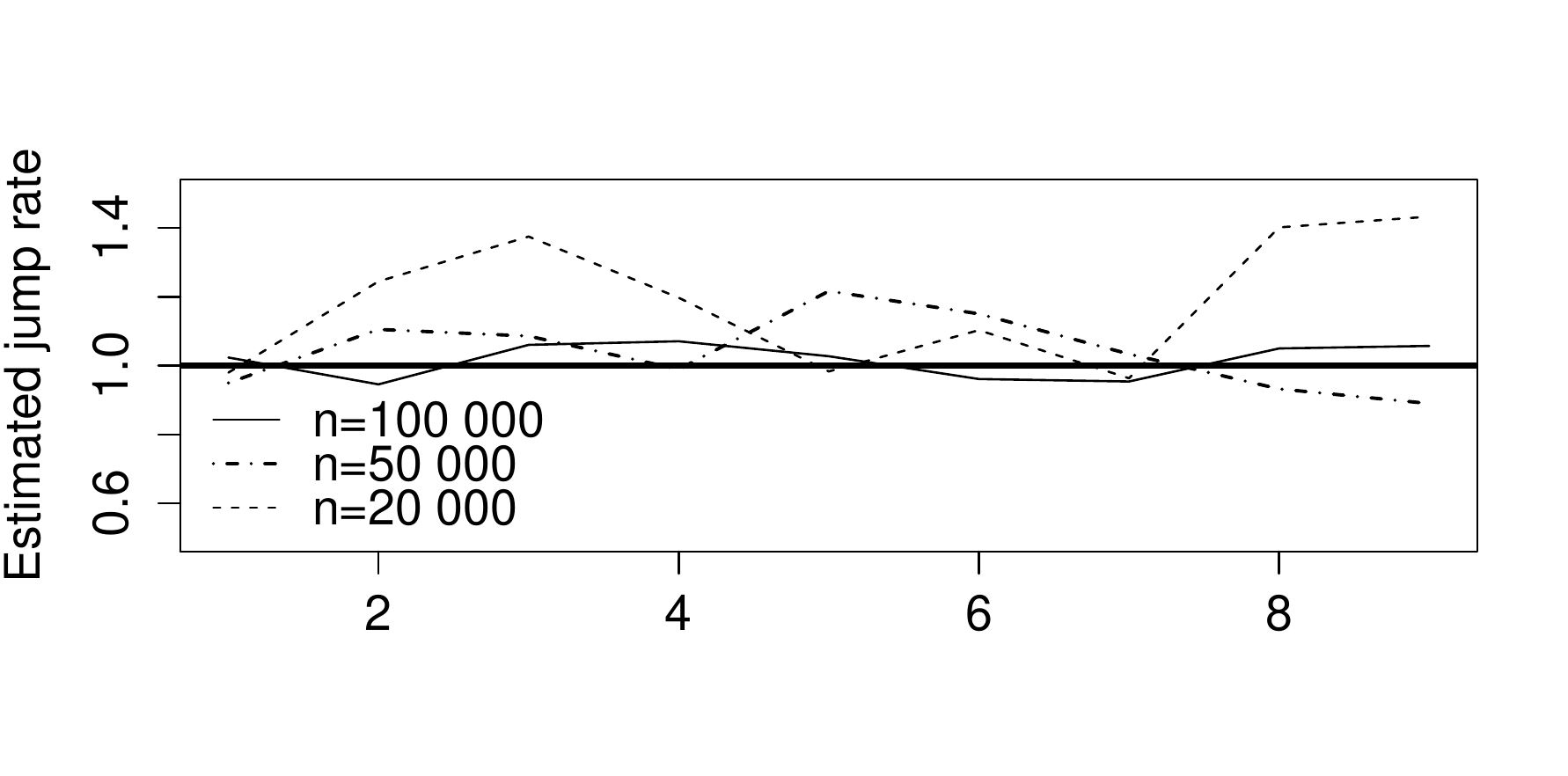}
\caption{For each of the 9 target point $(x_1,x_2)$ indexed from 1 to 9, estimated jump rates $\widehat \lambda(x_1,x_2)$ computed from different datasets.}
\label{fig:BacteriaComparN}
\end{figure}

\begin{figure}[p]
\centering\includegraphics[width=9cm]{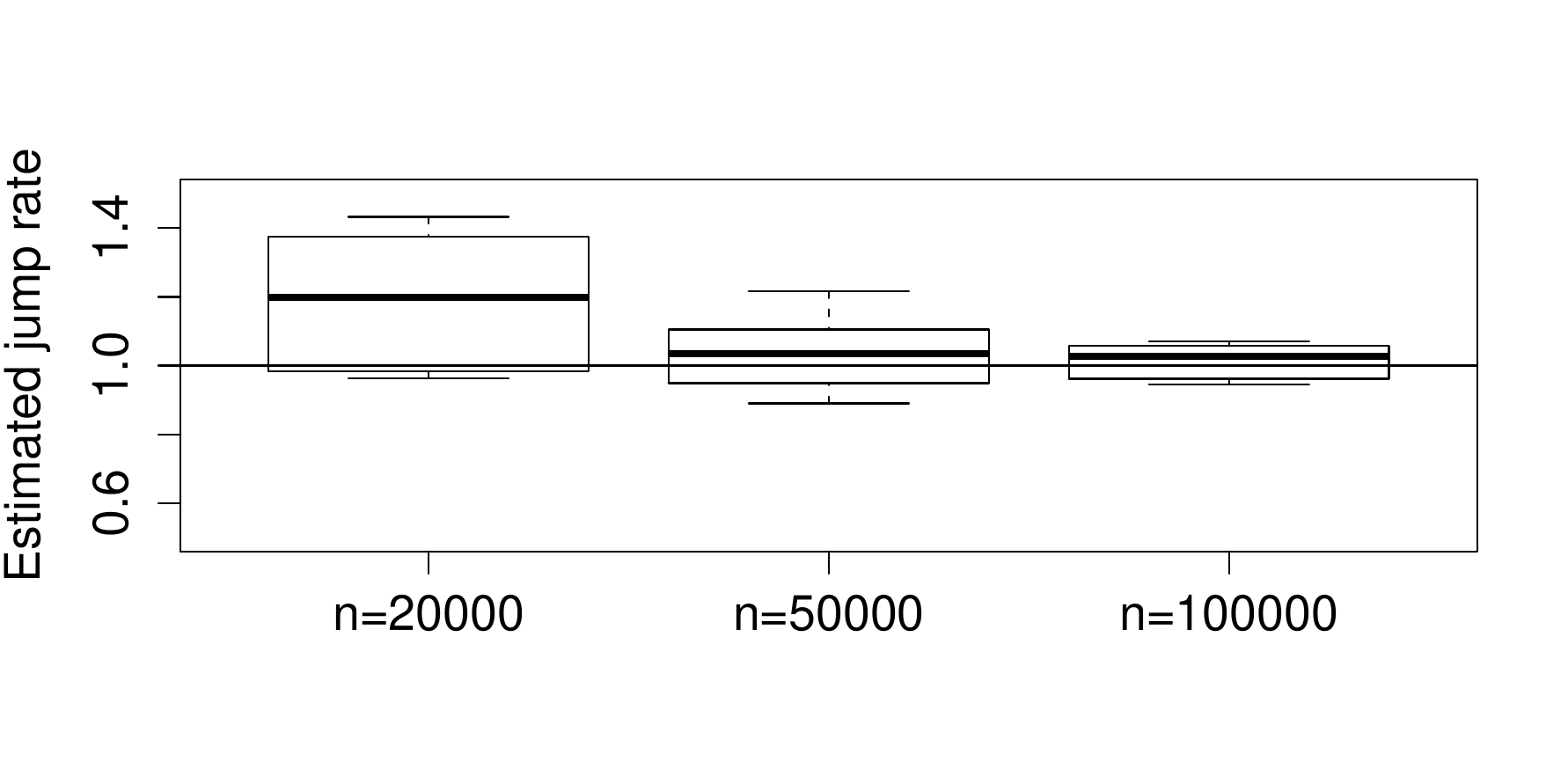}
\caption{Boxplots of the 9 estimated jump rates $\widehat \lambda(x_1,x_2)$ computed from datasets of different sizes.}
\label{fig:BoxplotComparN}
\end{figure}





\begin{figure}[p]
\centering
\includegraphics[width=9cm]{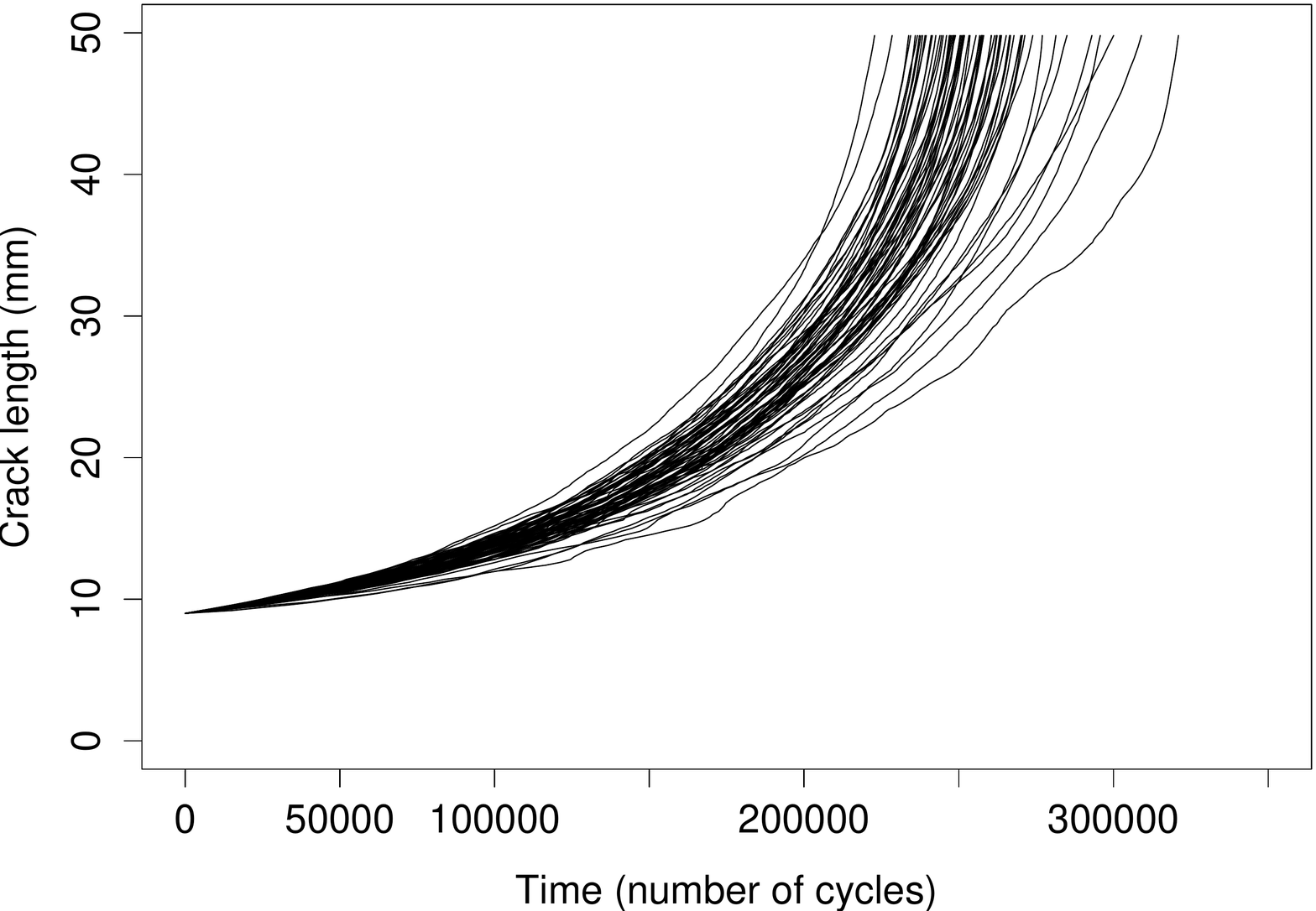}
\caption{Virkler's dataset contains 68 independent crack growth histories starting from $a_0=9\,\text{mm}$.}
\label{fig:virkler}
\end{figure}

\begin{figure}[p]
\centering
\includegraphics[width=9cm]{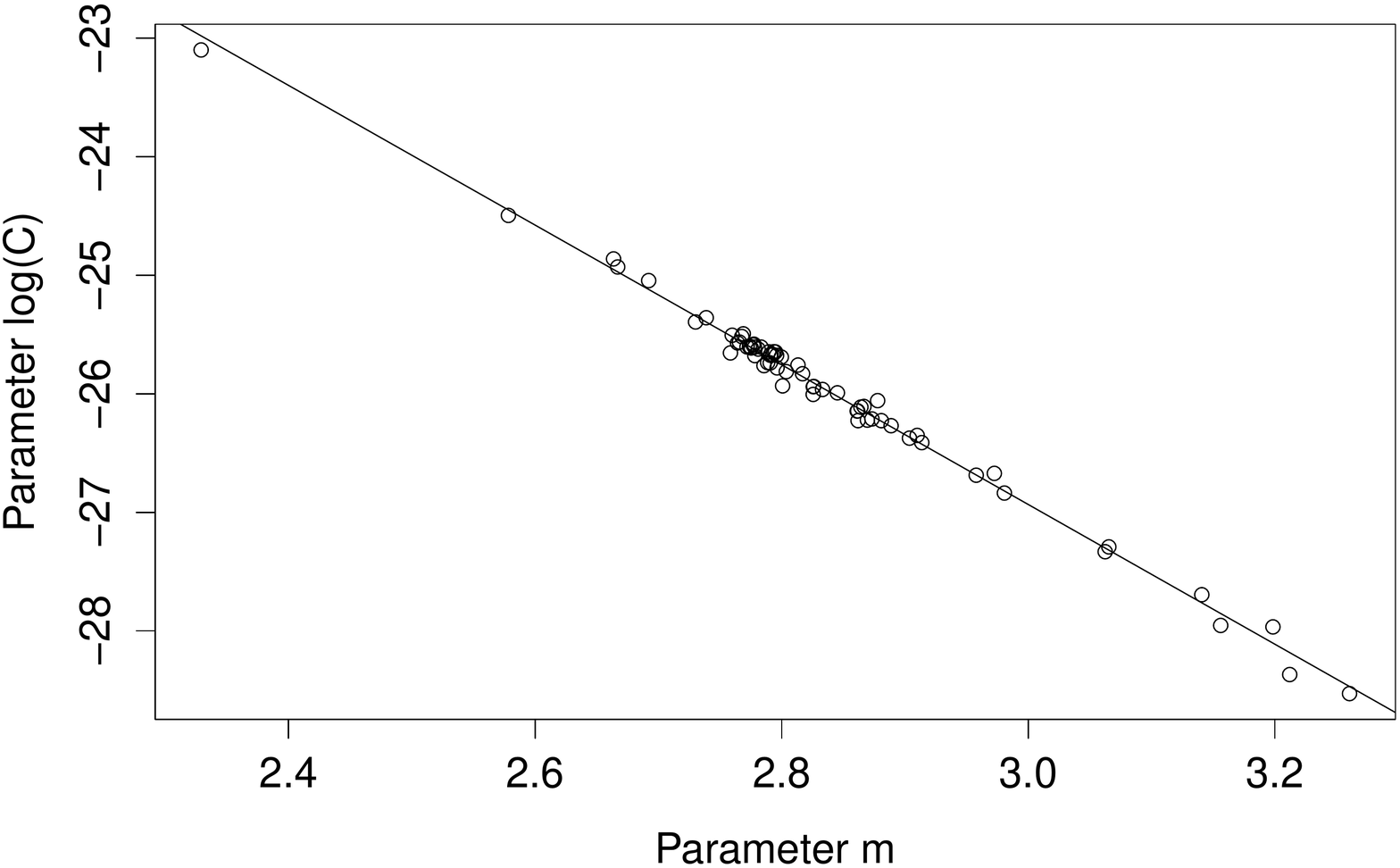}
\caption{Material parameters $m$ and $\log(C)$ are strongly linked by a linear relationship used to reduce the dimension of the underlying model: $\log(C)=-9.25-5.89\times m+\varepsilon$.}
\label{fig:mlogC}
\end{figure}

\begin{figure}[p]
\centering
\includegraphics[width=7cm]{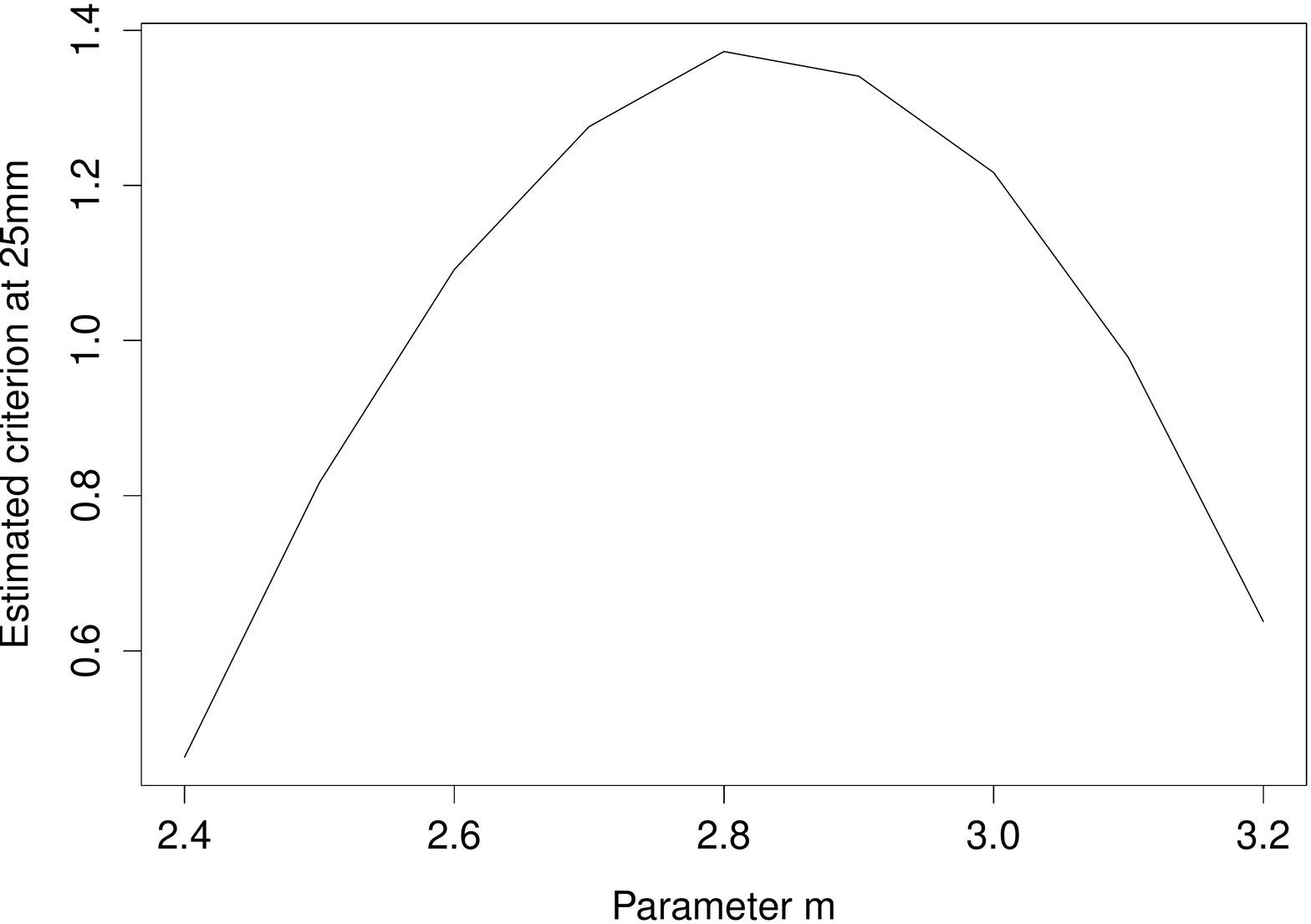}\quad\includegraphics[width=7cm]{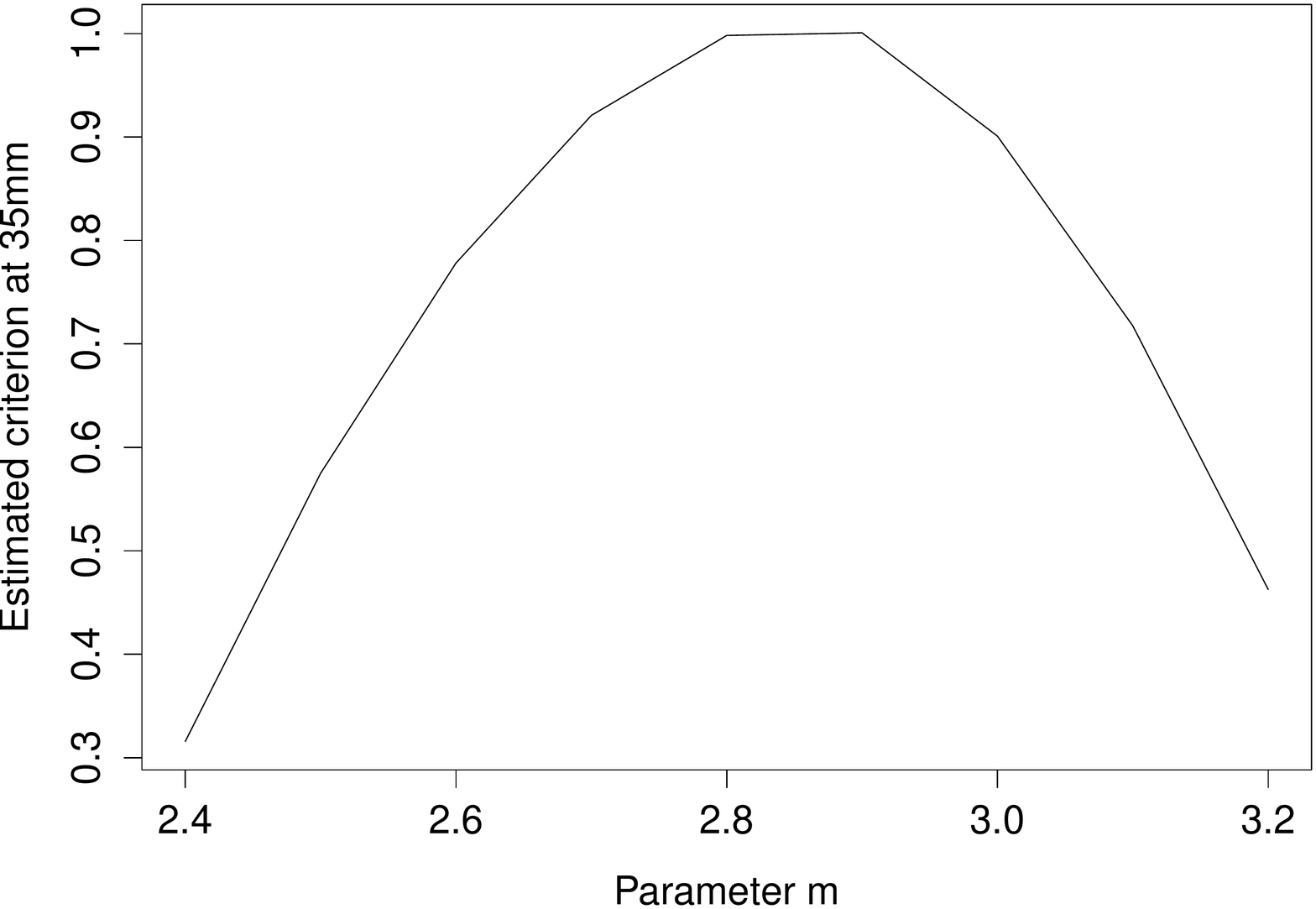}\\
\includegraphics[width=7cm]{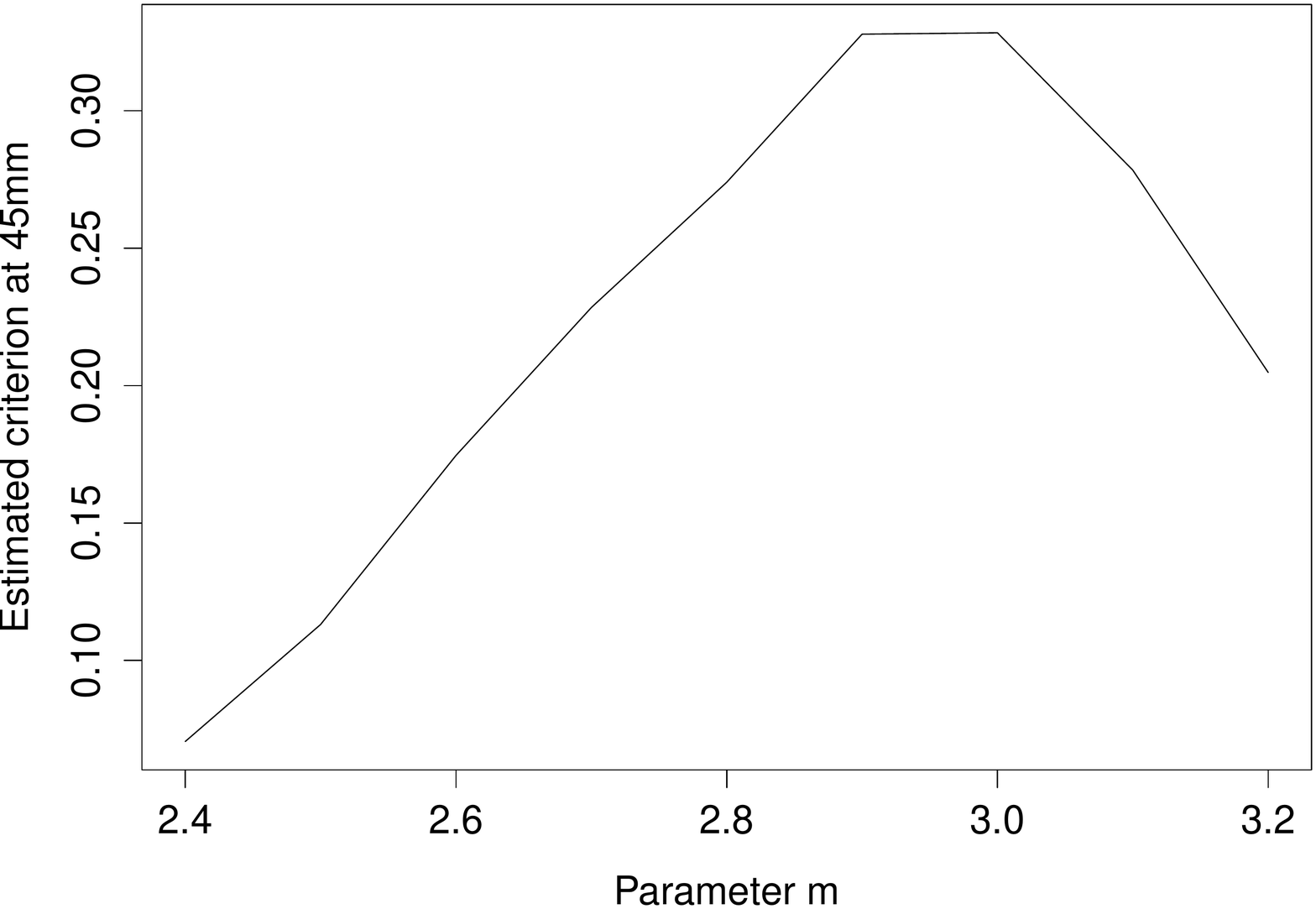}\quad\includegraphics[width=7cm]{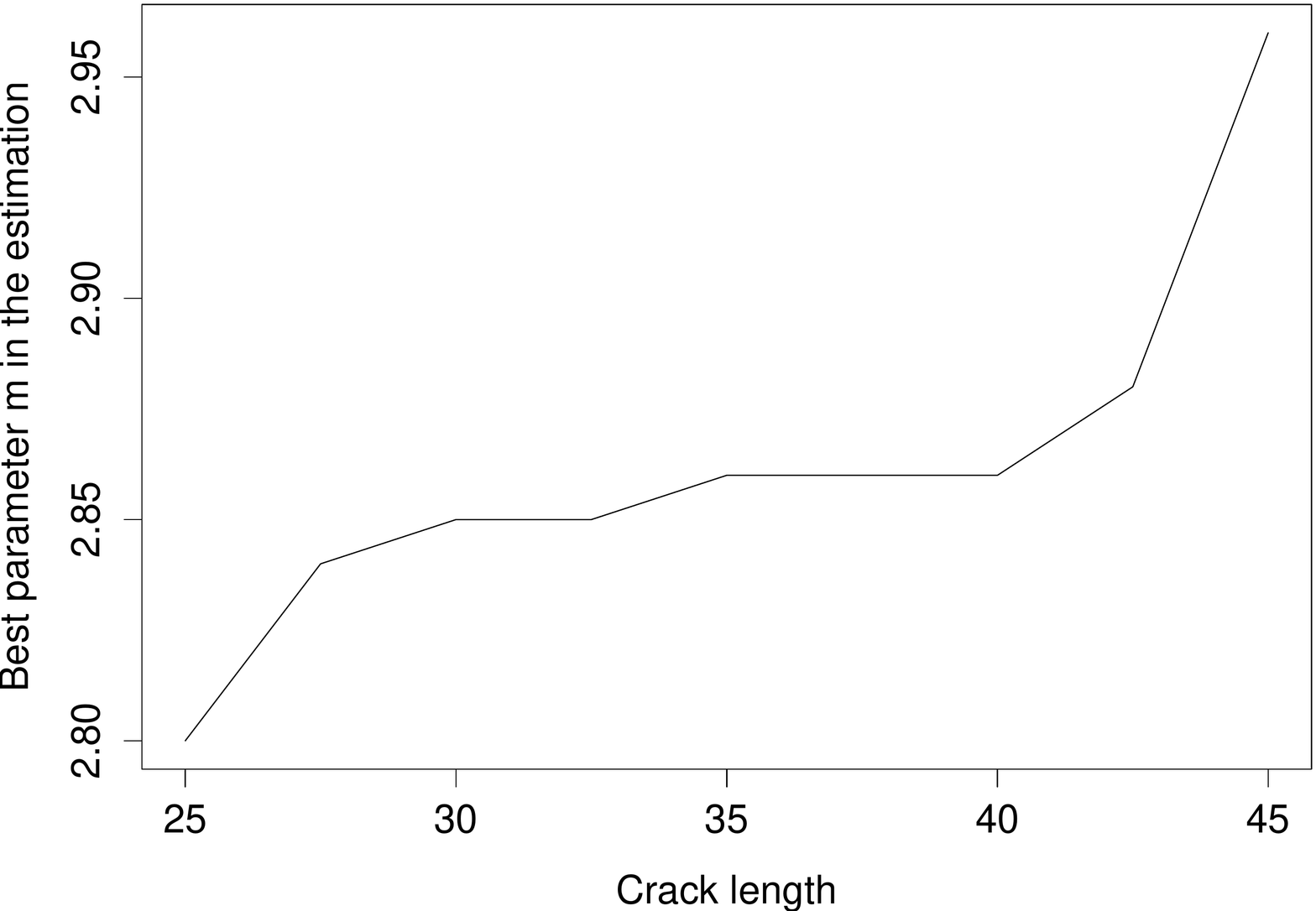}
\caption{Estimation of the criterion $\kappa_a(m)$ for different values of the target length $a$: $a=25\,\text{mm}$ (top, left), $a=35\,\text{mm}$ (top, right) and $a=45\,\text{mm}$ (bottom, left). The relationship between the optimal parameter $m$ maximizing $\kappa_a(m)$ and the target crack length $a$ is also presented (bottom, right).}
\label{fig:critere}
\end{figure}

\begin{figure}[p]
\centering
\includegraphics[width=9cm]{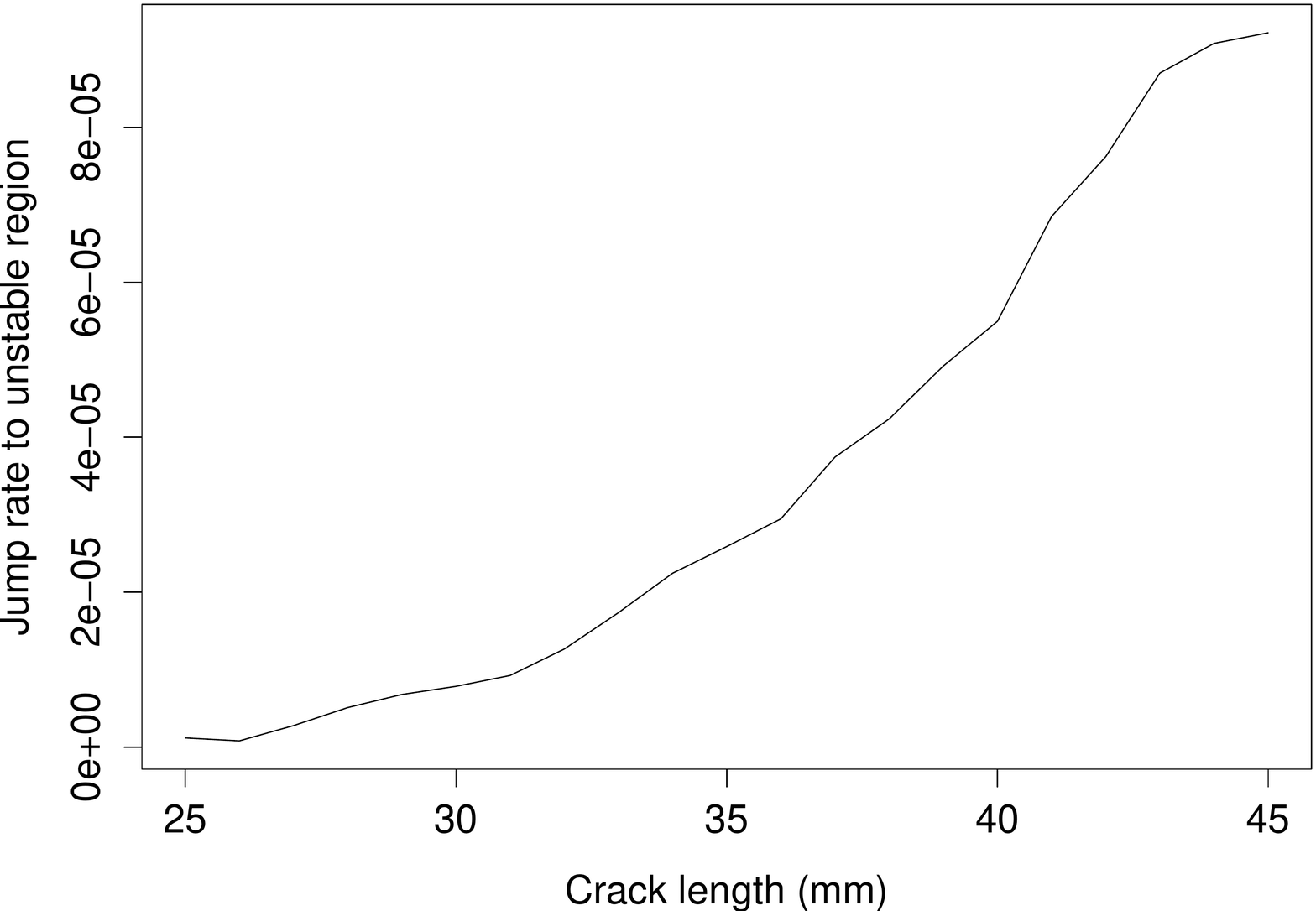}
\caption{Estimation of the jump rate $\lambda(a)$ for different crack lengths $a$ between $25\,\text{mm}$ and $45\,\text{mm}$ in the stochastic model of fatigue crack propagation from Virkler's dataset.}
\label{fig:fcglambda}
\end{figure}

\end{document}